\documentclass[12pt]{article}

\newcommand{\frommain}{}
\usepackage{tikz-cd}
\usepackage{amsthm} 
\newtheorem{lemma}{Lemma}

\usepackage{graphicx}
\usepackage[sans]{dsfont}
\usepackage[english]{babel}
\usepackage{xfrac}
\usepackage{color}
\usepackage{amsfonts, amsthm, amsmath}
\usepackage[all,cmtip]{xy}
\usepackage{tikz}
\usepackage{subcaption}

\usepackage{standalone}
\usepackage{import}

\usepackage[hypertexnames=false,
backref=page,
    pdftex,
    pdfpagemode=UseNone,
    breaklinks=true,
    extension=pdf,
    colorlinks=true,
    linkcolor=blue,
    citecolor=blue,
    urlcolor=blue,
]{hyperref}

\usepackage{cleveref}

\usepackage{mathbbol}
\DeclareSymbolFontAlphabet{\amsmathbb}{AMSb}

\renewcommand{\Re}{\operatorname{Re}}
\renewcommand{\Im}{\operatorname{Im}}

\newtheorem{theorem}{Theorem}
\newtheorem{corollary}[theorem]{Corollary}
\newtheorem{exercise}{Exercise}
\newtheorem*{problem}{Problem}

\DeclareMathOperator{\rk}{rk}
\DeclareMathOperator{\im}{im}

\def\Xint#1{\mathchoice
{\XXint\displaystyle\textstyle{#1}}
{\XXint\textstyle\scriptstyle{#1}}
{\XXint\scriptstyle\scriptscriptstyle{#1}}
{\XXint\scriptscriptstyle\scriptscriptstyle{#1}}
\!\int}
\def\XXint#1#2#3{{\setbox0=\hbox{$#1{#2#3}{\int}$ }
\vcenter{\hbox{$#2#3$ }}\kern-.6\wd0}}

\def\dashint{\Xint-}

\newcommand{\gsim}{\mathrel{\raisebox{0.3ex}{$>$}\kern-0.75em\raisebox{-0.5ex}{$\sim$}}}

\ifpdf
\fi

\title{Cohomological properties and Hermitian metrics of complex non-Kähler manifolds\thanks{
These notes expand on the course ``Cohomological properties and Hermitian metrics of complex non-Kähler manifolds'', delivered by the authors at the summer school on ``Singular Kählerian metrics and Hermitian geometry'', held at the Rényi Institute in Budapest, August 11-15, 2025. They were also conceived in view of the first edition of the Critical Math ``Cohomological and Homotopical Methods in Complex Geometry'', held on August 18-22, 2025, in the Black Forest.
The authors warmly thank the organizers and participants of both events for their kind invitations and for contributing to a stimulating and positive environment. 
We are especially grateful to Adriano Tomassini, Maurizio Parton and Victor Vuletescu, Simone Calamai and Cristiano Spotti, Francesco Pediconi, Valentino Tosatti, with whom many of the results discussed here have been developed, as well as all our other past and future collaborators.\\
The authors were partially supported by the PRIN2022 project \emph{``Real and Complex Manifolds: Geometry and Holomorphic Dynamics''} (code 2022AP8HZ9), and by GNSAGA of INdAM.
}}
\author{Daniele Angella\thanks{Dipartimento di Matematica e Informatica, Università di Firenze, viale Morgagni 67/A, 50134 Firenze, Italy.
Email: \href{mailto:daniele.angella@unifi.it}{\texttt{daniele.angella (at) unifi.it}}
}
\and
Nicoletta Tardini\thanks{Dipartimento di Scienze Matematiche, Fisiche e Informatiche, Università di Parma, Parco Area delle Scienze 53/A, 43124 Parma, Italy.
Email: \href{mailto:nicoletta.tardini@unipr.it}{\texttt{nicoletta.tardini (at) unipr.it}}
}}

\date{\today}

\begin{document}

\maketitle


\begin{abstract}
We give a partial account of some problems concerning cohomological invariants and metric properties of complex non-Kähler manifolds.
In \Cref{sec:holomorphic}, we review some basic notions and tools in complex geometry.
In \Cref{sec:cohomologies}, we investigate cohomological properties and topological aspects of complex manifolds. We introduce various complex cohomologies and their properties, focusing in particular on the structure of the double complex of differential forms. We also discuss the $\partial\overline\partial$-lemma and provide a numerical characterization in terms of Bott–Chern cohomology.
In \Cref{sec:metric}, we introduce Hermitian metrics and set the problems of studying canonical metrics. In particular, we review partial results on the classification of complex threefolds endowed with locally conformally Kähler metrics.
In \Cref{sec:chern-yamabe}, we focus on analytic problems related to the geometry of the Chern connection on Hermitian manifolds, such as the existence of metrics with constant Chern scalar curvature.
In \Cref{sec:flows}, we introduce geometric flows for Hermitian metrics, with particular attention to the behavior of the Chern–Ricci flow on Inoue surfaces.
At the end of each section, we collect some problems aimed at filling in details skipped during the lecture or extending the discussion toward other research directions.
The final \Cref{sec:open} gathers, in a list that is by no means comprehensive, some open problems in complex non-Kähler geometry which the authors, from their admittedly biased perspective, find particularly interesting.
\end{abstract}

\ifdefined\frommain
\else
    \bibliographystyle{alpha}
    \bibliography{biblio}
\fi

\tableofcontents

\section{The holomorphic landscape}\label{sec:holomorphic}

In this preliminary lecture, we introduce the basic tools that will be needed in the subsequent sections to study the geometry and topology of holomorphic manifolds, as well as the geography of their mysterious ``boundary regions'' \cite{stelzig-CoMa}.
There are many excellent books on complex geometry, including \cite{demailly-agbook, voisin-FR, voisin-1, voisin-2, griffiths-harris, morrow-kodaira, zheng, lee-complex}, as well as more specialized references such as \cite{demailly-analytic, kodaira-book, ma-marinescu, barth-hulek-peters-vandeven, bertin-demailly-illusie-peters, popovici-book}.

\subsection{Holomorphic manifolds and holomorphic maps}
A \emph{holomorphic manifold} $X^n$ of complex dimension $n$ is a topological space that is Hausdorff and second-countable, and such that every point has a neighborhood $U$ homeomorphic to an open subset of $\mathds{C}^n$. Such a homeomorphism is called a \emph{local holomorphic chart}, and is denoted  
\[
\varphi_U \colon U \xrightarrow{\sim} \varphi_U(U) \subseteq \mathds{C}^n.
\]  
In the definition, we also require that the transition maps between overlapping charts,  
\[
\varphi_V \circ \varphi_U^{-1}\big|_{\varphi_U(U \cap V)} \colon \varphi_U(U \cap V) \to \varphi_V(U \cap V),
\]  
are biholomorphic, meaning that they are holomorphic bijections with holomorphic inverses.

We briefly recall that a function between open subsets of complex Euclidean spaces is called \emph{holomorphic} if it is of class $\mathcal{C}^1$ and its differential is $\mathds{C}$-linear at every point.  
The theory of holomorphic functions in several complex variables shares many features with the one-variable theory, particularly from the analytic point of view. However, it also exhibits significant differences, especially concerning geometric aspects.

Following \cite{debartolomeis}, we prefer to use the term \emph{holomorphic manifold}, rather than the more common \emph{complex manifold}, as these spaces naturally arise as objects in the category of holomorphic maps.
A \emph{holomorphic map} $f \colon X^n \to Y^m$ between holomorphic manifolds is a function that is locally given by holomorphic functions. More precisely, for any point $x \in X$, any local holomorphic chart $(U, \varphi_U)$ of $X$ with $x \in U$, and any local holomorphic chart $(V, \psi_V)$ of $Y$ with $f(x) \in V$, the composition  $\psi_V \circ f \circ \varphi_U^{-1} \colon \varphi_U(U \cap f^{-1}(V)) \to \psi_V(V)$ is holomorphic between open subsets of the Euclidean spaces $\mathds{C}^n$ and $\mathds{C}^m$.

\subsection{Holomorphic vector bundles and line bundles}

One key aspect of holomorphic functions between Euclidean spaces, in one or several complex variables, is that the homogeneous Cauchy-Riemann system is elliptic.  
As a consequence, the maximum principle holds.  
In particular, it follows that the space of holomorphic functions $f \colon X \to \mathds{C}$ on a \emph{compact} holomorphic manifold $X$ consists only of constant functions.
For this reason, the geometry of compact holomorphic manifolds is often studied through the sections of holomorphic line bundles.

A \emph{holomorphic vector bundle} of rank $k$ over $X$ is given by a holomorphic surjective map  $\pi \colon E \to X$, where each fibre $E_x := \pi^{-1}(x)$ is a $k$-dimensional complex vector space.  
Moreover, the bundle is required to be \emph{locally trivial}, meaning that for every point $x \in X$ there exists an open neighborhood $U \ni x$ and a trivialization $\psi_U \colon \pi^{-1}(U) \xrightarrow{\sim} U \times \mathds{C}^k$, such that the restriction to each fibre yields a linear isomorphism $\psi_U|_{E_x} \colon E_x \xrightarrow{\sim} \{x\} \times \mathds{C}^k$.
On the overlap of two trivializations, the corresponding {\em transition function}  
\[
\psi_V \circ \psi_U^{-1}|_{\psi_U(U \cap V)} \colon (U \cap V) \times \mathds{C}^k \to (U \cap V) \times \mathds{C}^k
\]  
is of the form  $(x, v) \mapsto (x, g(x)(v))$, where $g(x) \in {\rm GL}(\mathds{C}^k)$ is holomorphic in $x$.
There is a natural notion of morphism between holomorphic vector bundles: it is a holomorphic map between the total spaces that commutes with the projections and restricts to linear maps between fibres.
A bundle is said to be \emph{trivial} if it is isomorphic to the product bundle $X \times \mathds{C}^k$.
The space of holomorphic \emph{sections} of $E$, that is, holomorphic maps $s \colon X \to E$ satisfying $\pi \circ s = \mathrm{id}_X$, is denoted by $H^0(X, E)$.

In particular, when $k = 1$, the bundle is called a \emph{holomorphic line bundle}.  
Holomorphic line bundles over $X$, up to isomorphism, form an abelian group under the tensor product, known as the \emph{Picard group} and denoted by ${\rm Pic}(X)$.
Let us briefly revisit the definition of a holomorphic line bundle.  
There exists an open covering $\{U_i\}_{i \in I}$ of $X$ together with trivializations $\psi_i \colon \pi^{-1}(U_i) \xrightarrow{\sim} U_i \times \mathds{C}$, such that the transition functions $\psi_j \circ \psi_i^{-1}\big|_{\psi_i(U_i \cap U_j)} \colon (U_i \cap U_j) \times \mathds{C} \to (U_i \cap U_j) \times \mathds{C}$ are of the form  
\[
(x, v) \mapsto (x, g_{ij}(x) \cdot v),
\]  
where each $g_{ij} \colon U_i \cap U_j \to \mathds{C}^*$ is a holomorphic function.
The collection of transition functions $\{g_{ij}\}$ satisfies the \emph{cocycle conditions}:
\[
g_{ii} = 1 \quad \text{on } U_i, \qquad \text{and} \qquad g_{ki} \cdot g_{jk} \cdot g_{ij} = 1 \quad \text{on } U_i \cap U_j \cap U_k.
\]
When a holomorphic line bundle is trivial, there exist holomorphic functions $\{g_i \colon U_i \to \mathds{C}^*\}_i$ such that
\[
g_{ij} = g_j \cdot g_i^{-1} \quad \text{on } U_i \cap U_j.
\]
Conversely, given an open covering $\{U_i\}_i$ of $X$ and a collection of nowhere vanishing holomorphic functions $\{g_{ij} \colon U_i \cap U_j \to \mathds{C}^*\}_{i,j}$ satisfying the cocycle conditions, one can construct a holomorphic line bundle over $X$.  
Such a bundle is trivial if and only if there exist functions $\{g_i\}_i$ as above such that $g_{ij} = g_j \cdot g_i^{-1}$.

This correspondence can be summarized by stating that the Picard group is isomorphic to the first \v{C}ech cohomology group $\check{H}^1(X, \mathcal{O}_X^*)$ with values in the sheaf of germs of invertible ({\itshape i.e.}, nowhere vanishing) holomorphic functions, which, in turn, is naturally isomorphic to the sheaf cohomology group $H^1(X, \mathcal{O}_X^*)$.
By considering the \emph{exponential sequence} of sheaves on $X$,
\[
0 \to \underline{\mathds{Z}}_X \stackrel{2\pi \sqrt{-1}}{\to} \mathcal{O}_X \stackrel{\exp}{\to} \mathcal{O}_X^* \to 0,
\]
and passing to the associated long exact sequence in cohomology, we obtain a map
\[
c_1 \colon {\rm Pic}(X) \simeq H^1(X, \mathcal{O}_X^*) \to H^2(X, \mathds{Z}),
\]
which is called the \emph{first Chern class}.

Central to classification results in complex geometry, see {\itshape e.g.} \cite{barth-hulek-peters-vandeven, ueno-LNM}, is the \emph{canonical bundle} \( K_X \), defined as the determinant of the holomorphic cotangent bundle, together with its tensor powers, the \emph{pluricanonical bundles}. A fundamental invariant is the \emph{Kodaira dimension}, given by
\[
\mathrm{Kod}(X) := \limsup_{\ell \to \infty} \frac{\log \dim H^0(X, K_X^{\otimes \ell})}{\log \ell} \in \{-\infty, 0, 1, \dots, \dim X\}.
\]
Thanks to work by Serre and Siegel, this invariant can also be interpreted as the maximal rank of the pluricanonical maps, away from their base locus.

In complex dimension $1$, that is, for Riemann surfaces, the Uniformization Theorem provides a complete classification: every complex curve is biholomorphic to a quotient of one of three model surfaces: the open unit disk, the complex plane, or the Riemann sphere.
In complex dimension $2$, the Enriques–Kodaira–Siu classification is well understood, except for the so-called class VII: see \cite{barth-hulek-peters-vandeven} for a comprehensive overview, and \cite{teleman-LNM} for recent developments.
In higher dimensions, we only mention the Minimal Model Program, which plays a central role in the birational classification of algebraic varieties, as well as its analytic counterpart with Ricci flow \cite{song-tian}.

\subsection{Almost-complex structures}
Any holomorphic manifold $X$ of complex dimension $n$ is, in particular, also a differentiable manifold of real dimension $2n$.  
Indeed, given a local holomorphic chart $\varphi \colon U \xrightarrow{\sim} \varphi(U) \subseteq \mathds{C}^n$ with coordinates $\varphi = (z^1, \ldots, z^n)$, one can write each complex coordinate as $z^i = x^i + \sqrt{-1} y^i$, where $x^i, y^i$ are real-valued functions on $U$.
Then the tuple $(U, (x^1, y^1, \ldots, x^n, y^n))$ defines a local differentiable chart for $X$.  
In particular, $\left\{ \sfrac{\partial}{\partial x^i}, \sfrac{\partial}{\partial y^i} \right\}_{i=1}^n$ forms a local frame for the real tangent bundle $TX$.
Since every complex linear map, when regarded as a real linear map, has non-negative determinant, it follows that the underlying differentiable manifold of $X$ is orientable.

We have a well-defined endomorphism $J$ of the tangent bundle $TX$, locally given by
$$ J \sfrac{\partial}{\partial x^i} = \sfrac{\partial}{\partial y^i}, \qquad J \sfrac{\partial}{\partial y^i} = -\sfrac{\partial}{\partial x^i}. $$
This endomorphism satisfies $J^2=-{\rm id}$ and encodes the pointwise linear complex structure on the tangent spaces, varying smoothly with the point. An endomorphism $J\in {\rm End}(TX)$ satisfying $J^2=-{\rm id}$ is called an {\em almost-complex structure}.

Conversely, an almost complex structure $J$ on a differentiable manifold $X$ is locally induced by holomorphic coordinates if and only if the bundle of $\sqrt{-1}$-eigenspaces, defined by  
\[
T^{1,0}_x X := \{ v - \sqrt{-1} Jv : v \in T_xX \} \subset T_xX \otimes_{\mathds{R}} \mathds{C},
\]  
is involutive. This is the content of the celebrated Newlander–Nirenberg theorem \cite{newlander-nirenberg}.
Equivalently, this condition holds if and only if $[V^{1,0},W^{1,0}]^{0,1}=0$ for any vector fields $V$ and $W$, where $V=V^{1,0}+V^{0,1}$ denotes the decomposition with respect to the decomposition $TX \otimes_{\mathds R}\mathds C = T^{1,0}X \oplus T^{0,1}X$ with $T^{0,1}X = \overline{T^{1,0}X}$. This is also equivalent to the vanishing of the so-called {\em Nijenhuis tensor} $N_J(V,W)=[V,W]+J[JV,W]+J[V,JW]-[JV,JW]$.

In even real dimensions greater than four, there are no known examples of orientable differentiable manifolds that admit almost-complex structures but do not admit any integrable complex structure.
For instance, it is an open problem whether the six-dimensional sphere $S^6$ admits an integrable complex structure, see \cite{agricola-bazzoni-goertsches-konstantis-rollenske} and the references therein. See also \cite{albanese-milivojevic} for a conjecture by Sullivan on the minimal sum of Betti numbers of a compact complex $n$-fold, with $n \geq 3$ being four. Further discussions about this problem will be collected in \Cref{sec:open}.

\subsection{Complex algebraic geometry and K\"ahler geometry}
The first examples of {\em compact} holomorphic manifolds are the complex torus $\sfrac{\mathds C^n}{\mathds Z^n}$ and the {\em complex projective space} $\mathds CP^n = \sfrac{\mathds C^{n+1}\setminus 0}{\mathds C\setminus 0}$. Another fundamental construction is as follows. Consider a homogeneous polynomial in $n+1$ variables with no multiple roots. Its zero set in $\mathds CP^n$ defines a compact holomorphic manifold of complex dimension $n-1$. More generally, the zero set of a finite number of homogeneous polynomials defines a compact holomorphic manifold, provided that the polynomials are sufficiently generic to assure smoothness. Many geometric properties of such manifolds are indeed encoded in the ideal generated by its defining polynomials (see {\itshape e.g.} \cite{hubsch}), which explains the name {\em algebraic projective manifolds}. For example, a smooth hypersurface $X$ of degree $d$ in $\mathds CP^n$ has canonical bundle isomorphic to $\mathcal O_X(-n-1+d)$. In particular, for $d=n+1$, the hypersurface $X$ is a Calabi-Yau manifold.

The examples $X$ discussed above share a common property: there exists a Hermitian metric $g$ that osculates to order $2$ the standard Hermitian inner product of $\mathds C^n$. Specifically, this means that for every point $x \in X$, there exist local holomorphic coordinates $(z^1, \ldots, z^n)$ on a neighborhood $U$ such that $g = \sum_{i,j=1}^n (\delta_{ij}+o(z)) \, dz^i \otimes d\bar z^j$ at $x$.
Such metrics are known as {\em K\"ahler metrics}, see \cite{schouten-vandantzig, kahler}, since \cite{weil}.
Equivalently, a Hermitian metric is K\"ahler if and only if its associated $2$-form $\omega$ is symplectic, meaning that $d\omega=0$.
Indeed, on $\mathds CP^n$ with homogeneous coordinates $[z^0:z^1:\cdots:z^n]$, the Fubini-Study metric \cite{fubini, study}, defined as $\omega_{FS}= \sfrac{1}{4\pi}dd^c\log(\sum_i|z^i|^2)$, is a K\"ahler metric. This metric naturally induces a K\"ahler metric on any submanifolds.
More precisely, K\"ahler geometry is a sort of transcendental analogue \cite{demailly-analytic} of algebraic geometry. More precisely, it is well-known \cite{kodaira} that a compact K\"ahler manifold is an algebraic projective manifold if and only if its associated $2$-form $\omega$ defines a rational cohomology class $[\omega]\in H^2(X,\mathds Q)$.
K\"ahler geometry lies at the intersection of complex, Riemannian, and symplectic geometries. Each of these perspectives provides specific tools, and their compatibility allows for easily switching between these points of view. This interplay makes K\"ahler geometry nearly as powerful as algebraic geometry itself.

Since K\"ahler geometry ``represents a perfect synthesis of the Symplectic and the Holomorphic worlds'', it is meaningful to perform ``a sort of chemical analysis of symplectic and holomorphic contribution [...] in order to better understand the role of the different components of the theory'' \cite{debartolomeis-tomassini}
For this reason, we are interested in investigating the broad class of complex {\em non-K\"ahler} manifolds.

\ifdefined\frommain
\else
    \bibliographystyle{alpha}
    \bibliography{biblio}
\fi

\subsection{Problems}

\begin{exercise}
Prove that every almost-complex manifold is orientable.
\end{exercise}

\begin{exercise}
Show that the complex projective space $\mathds{CP}^n$ is a complex manifold.
\end{exercise}

\begin{exercise}
Consider the sphere $S^2 \subseteq \mathds{R}^3$. Let $\wedge$ denote the vector cross product in $\mathds{R}^3$. Recall that, by identifying $\mathds{R}^3$ with the space of imaginary quaternions $\mathrm{Im}(H)$, the cross product can be expressed as $x \wedge y = \mathrm{Im}(x \cdot y)$, where the product is taken in the quaternion algebra $H$.
Observe that the tangent space at a point $p \in S^2$ is given by $T_p S^2 = \{p\}^\perp$. Prove that the endomorphism
$$ J_p(v) = v \wedge p $$
for $v \in T_p S^2$ defines an almost-complex structure on $S^2$, and that it is integrable. Finally, prove that $S^2$ is biholomorphic to $\mathds{CP}^1$.
\end{exercise}

\begin{exercise}[{\cite{kirchhoff}}]
Consider the sphere $S^6 \subseteq \mathds{R}^7 = \Im O$, where $O$ denotes the algebra of octonions.  
Define an endomorphism
$$ J_p(v) := \Im(v \cdot p), $$
for $v \in T_p S^6$.
Show that $J_p$ defines an almost-complex structure on $S^6$.
Prove that the failure of associativity of octonions implies that this almost-complex structure is \emph{not} integrable.
\end{exercise}

\begin{proof}[Remark]
By a theorem of Borel and Serre \cite{borel-serre}, the only spheres that admit almost-complex structures are those of dimension $2$ and $6$, see, for instance, \cite{konstantis-parton}. Regarding the existence of an \emph{integrable} almost-complex structure on the six-dimensional sphere, an open problem known as the \emph{Hopf problem}, we refer to \cite{agricola-bazzoni-goertsches-konstantis-rollenske}. Extending the exercise, note that there are no complex structures on the six-dimensional sphere that are compatible with the round metric \cite{blanchard, lebrun}, see also \cite{ferreira}.
\end{proof}

\begin{exercise}
On $\mathds{C}^n \setminus \{0\}$, consider the action of $\mathds{Z}$ generated by the holomorphic contraction
$$(z_1, \dots, z_n) \mapsto \left(\tfrac{1}{2}z_1, \dots, \tfrac{1}{2}z_n\right).$$
Show that the quotient $\sfrac{(\mathds{C}^2 \setminus \{0\})}{\mathds{Z}}$ is a compact complex manifold, called a \emph{Hopf manifold}, and that it is diffeomorphic to $S^1 \times S^{2n-1}$.
\end{exercise}

\ifdefined\frommain
\else
    \bibliographystyle{alpha}
    \bibliography{biblio}
\fi

\cleardoublepage

\section{Cohomological invariants of holomorphic manifolds}\label{sec:cohomologies}

In this section, we introduce the double complex of differential forms and its associated cohomologies.  
For recent and in-depth developments on the double complex of complex manifolds and its homotopical and cohomological invariants, we refer to \cite{stelzig-CoMa, stelzig-JLMS, stelzig-AdvMath22, stelzig-AdvMath25} and other works by Jonas Stelzig.  
We focus in particular on the cohomological decomposition property known as the \(\partial \overline{\partial}\)-Lemma property, providing a numerical characterization following \cite{angella-tomassini-InvMath}, and studying its behavior under natural geometric constructions.
Compact quotients of Lie groups provide a rich source of examples where explicit computations can be carried out.

\subsection{Double complex of forms}
The action of an almost-complex structure $J$ on a differentiable manifold $X$ induces a decomposition of the complexified tangent bundle,
\[
TX \otimes_{\mathds{R}} \mathds{C} = T^{1,0}X \oplus T^{0,1}X,
\]
into the eigenspaces of $J$ corresponding to the eigenvalues $\sqrt{-1}$ and $-\sqrt{-1}$, respectively.  
This decomposition extends naturally to its dual, the complexified cotangent bundle and, more generally, to all its exterior powers.
More precisely, the bundle of differential forms decomposes into bigraded components as
\[
\wedge^k T^*X \otimes_{\mathds R}\mathds C = \bigoplus_{p+q=k} \wedge^{p,q}X, \text{ where } \wedge^{p,q}X:=\wedge^p T^{*\,1,0}X \otimes \wedge^q T^{*\,0,1}X.
\] Denote by $\Omega^{p,q}(X)$ the space of the corresponding smooth sections. Accordingly, the exterior differential $d \colon \Omega^{p,q}(X) \to \Omega^{p+q+1}(X)$ on $(p,q)$-differential forms decomposes into components. In general, the exterior differential has four components:
\[
d \colon \Omega^{p,q}(X) \to \Omega^{p+2,q-1}(X) \oplus \Omega^{p+1,q}(X) \oplus \Omega^{p,q+1}(X) \oplus \Omega^{p-1,q+2}(X).
\]
The integrability of $J$ is equivalent to the condition that only the following differential operators remain non-zero:
\[
\partial \colon \Omega^{p,q}(X) \to \Omega^{p+1,q}(X), \qquad \overline\partial \colon \Omega^{p,q}(X) \to \Omega^{p,q+1}(X).
\]
In this case, the equation $d^2=0$ assures that
\[
\partial^2=\partial\overline\partial+\overline\partial\partial=\overline\partial^2=0.
\]
In other words, $(\Omega^{\bullet,\bullet}(X),\partial,\overline\partial)$ has the
structure of a {\em double complex}.

The following theorem describes the structure of a double complex in terms of its elementary components.

\begin{theorem}[{\cite{stelzig-JLMS, khovanov-qi}}]
Every double complex is a direct sum of indecomposable double complexes, each of which is either isomorphic to a square or to a zigzag.
\end{theorem}

In the following diagrams, we depict only the non-zero morphisms, corresponding to $\partial$ (horizontal arrows) and $\overline\partial$ (vertical arrows). A \emph{square} is a double complex consisting of anticommuting isomorphisms:
$$
\xymatrix{
\mathds{C} \ar[r]^\partial & \mathds{C} \\
\mathds{C} \ar[r]^\partial \ar[u]^{\overline\partial} & \mathds{C} \ar[u]_{\overline\partial}
}
$$
where all maps are the identity up to a sign.
A \emph{zigzag} is a double complex concentrated in at most two antidiagonals, where all non-zero vector spaces are isomorphic to $\mathds{C}$, all non-zero maps are the identity, and all non-zero vector spaces are connected by a chain of maps:
$$
\xymatrix{
\mathds{C} & \\
\mathds{C} \ar[u] \ar[r] & \mathds{C}
}
\quad \text{or} \quad
\xymatrix{
\mathds{C} \ar[r] & \mathds{C} \\
& \mathds{C} \ar[u]
}
\quad \text{or} \quad
\xymatrix{
\mathds{C} \ar[r] & \mathds{C} & \\
& \mathds{C} \ar[u] \ar[r] & \mathds{C}
}
\quad \text{or} \quad
\xymatrix{
\mathds{C} & \\
\mathds{C} \ar[u] \ar[r] & \mathds{C} \\
& \mathds{C} \ar[u]
}
$$
The total dimension of a zigzag, namely the number of vertices, is called its \emph{length}. A zigzag of length one is also called a \emph{dot}, and corresponds to a one-dimensional double complex with all differentials equal to zero.

For example, based on the results by Ugarte \cite{ugarte-GeomDed} on the Fr\"olicher spectral sequence of a hypothetical complex structure on the six-dimensional sphere, one can construct the double complex associated with such a structure \cite{angella-DGA}, as shown in \Cref{fig:S6}.

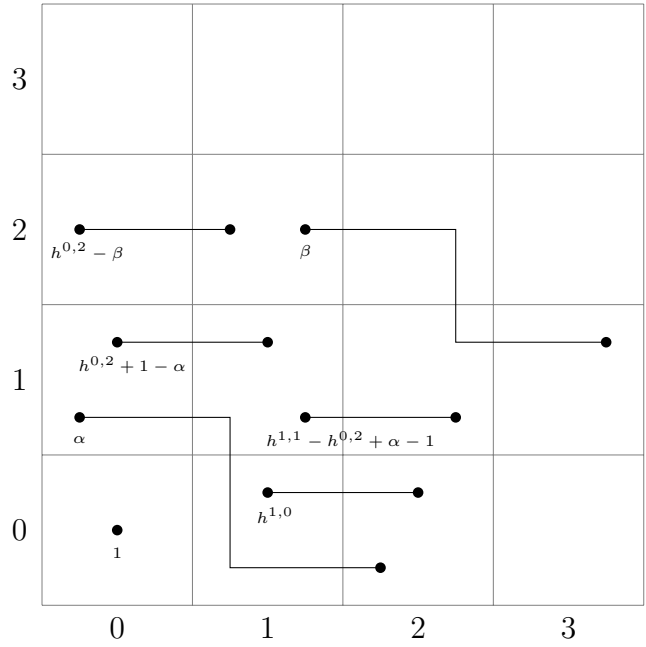
\begin{figure}[p]
\begin{center}
\begin{tikzpicture}
\newcommand\un{2}

\draw[help lines, step=\un] (0,0) grid (4*\un,4*\un);

\foreach \x in {0,...,3}
  \node at (\un*.5+\un*\x,-.3) {\x};
\foreach \y in {0,...,3}
  \node at (-.3,\un*.5+\un*\y) {\y};

\coordinate (A) at (0*\un+1/2*\un, 0*\un+1/2*\un);
\coordinate (B) at (0*\un+1/4*\un, 1*\un+1/4*\un);
\coordinate (C) at (0*\un+1/2*\un, 1*\un+3/4*\un);
\coordinate (D) at (0*\un+1/4*\un, 2*\un+1/2*\un);
\coordinate (E) at (1*\un+1/2*\un, 0*\un+3/4*\un);
\coordinate (F) at (1*\un+1/4*\un, 0*\un+1/4*\un);
\coordinate (G) at (1*\un+1/4*\un, 1*\un+1/4*\un);
\coordinate (H) at (1*\un+3/4*\un, 1*\un+1/4*\un);
\coordinate (I) at (1*\un+1/2*\un, 1*\un+3/4*\un);
\coordinate (L) at (1*\un+1/4*\un, 2*\un+1/2*\un);
\coordinate (M) at (1*\un+3/4*\un, 2*\un+1/2*\un);
\coordinate (N) at (2*\un+1/2*\un, 0*\un+3/4*\un);
\coordinate (O) at (2*\un+1/4*\un, 0*\un+1/4*\un);
\coordinate (P) at (2*\un+3/4*\un, 1*\un+1/4*\un);
\coordinate (Q) at (2*\un+3/4*\un, 1*\un+3/4*\un);
\coordinate (R) at (2*\un+3/4*\un, 2*\un+1/2*\un);
\coordinate (S) at (3*\un+3/4*\un, 1*\un+3/4*\un);

\newcommand{\raggio}{1*\un pt}
\fill (A) circle (\raggio);
\fill (B) circle (\raggio);
\fill (C) circle (\raggio);
\fill (D) circle (\raggio);
\fill (E) circle (\raggio);
\fill (H) circle (\raggio);
\fill (I) circle (\raggio);
\fill (L) circle (\raggio);
\fill (M) circle (\raggio);
\fill (N) circle (\raggio);
\fill (O) circle (\raggio);
\fill (P) circle (\raggio);
\fill (S) circle (\raggio);

\draw (B) -- (G) -- (F) -- (O);
\draw (C) -- (I);
\draw (D) -- (L);
\draw (E) -- (N);
\draw (H) -- (P);
\draw (M) -- (R) -- (Q) -- (S);

\begingroup\makeatletter\def\f@size{6}\check@mathfonts
\node at (0*\un+1/2*\un, 0*\un+1/2*\un-.3) {$1$};
\node at (0*\un+1/4*\un, 1*\un+1/4*\un-.3) {$\alpha$};
\node at (0*\un+1/2*\un+.2, 1*\un+3/4*\un-.3) {$h^{0,2}+1-\alpha$};
\node at (0*\un+1/4*\un+.1, 2*\un+1/2*\un-.3) {$h^{0,2}-\beta$};
\node at (1*\un+1/2*\un+.1, 0*\un+3/4*\un-.3) {$h^{1,0}$};
\node at (1*\un+3/4*\un+.6, 1*\un+1/4*\un-.3) {$h^{1,1}-h^{0,2}+\alpha-1$};
\node at (1*\un+3/4*\un, 2*\un+1/2*\un-.3) {$\beta$};
\endgroup

\end{tikzpicture}
\end{center}
\label{fig:S6}
\caption{
\small
Main structure of the double complex associated with a hypothetical complex structure on the $6$-dimensional sphere.
The labels indicate the multiplicity of the respective objects, and $\alpha$, $h^{0,2}$, $\beta$, $h^{1,0}$, $h^{1,1}$ denote undetermined non-negative integers.
In this diagram, infinite families of squares, as well as arrows resulting from symmetries have been omitted (complex conjugation yields a symmetry with respect to the bottom-left/top-right diagonal, and Serre duality yields a symmetry with respect to the bottom-right/top-left diagonal).
}
\end{figure}

\subsection{Complex cohomologies}

Associated to a double complex, several cohomologies can be defined. As we will see in \Cref{sec:hodge}, on {\em compact} holomorphic manifolds, they are always finite-dimensional.

The {\em Dolbeault cohomology}
\[
H^{\bullet,\bullet}_{\overline\partial}(X)=\frac{\ker\overline\partial}{\im\overline\partial}
\]
is the cohomology of the sheaf $\mathcal O_X$ of germs of holomorphic functions; its conjugate $H^{\bullet,\bullet}_{\partial}(X)$ is defined similarly.
In terms of indecomposable components, the Dolbeault cohomology is computed by removing all vertical arrows along with their endpoints from the double complex diagram, and then counting the remaining vertices. Similarly, the conjugate Dolbeault cohomology is obtained by erasing the horizontal arrows. In particular, note that squares do not contribute to either cohomology, while dots always contribute. Also note that zigzags of odd length do not contribute to the difference between Dolbeault and de Rham cohomology.

The {\em Bott-Chern cohomology} \cite{bott-chern} is defined as
\[
H^{\bullet,\bullet}_{BC}(X)= \frac{\ker \partial \cap \ker \overline\partial}{\im\partial\overline\partial}, \]
and its ``dual'', the {\em Aeppli cohomology} \cite{aeppli}, as 
\[
H^{\bullet,\bullet}_A(X)=\frac{\ker\partial\overline\partial}{\im\partial+\im\overline\partial}.
\]
For a sheaf hypercohomology interpretation of Bott-Chern and Aeppli cohomologies, see \cite{demailly-agbook, schweitzer}.  
In terms of indecomposable elements in the double complex diagram, Bott-Chern cohomology counts the corner elements that may have incoming arrows, except for the squares:  
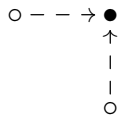
\[
\xymatrix{
\circ \ar@{-->}[r] & \bullet \\
& \circ \ar@{-->}[u]
}
\]
Dually, Aeppli cohomology counts the corner elements that may have outgoing arrows, except for the squares:  
\[
\xymatrix{
\circ & \\
\bullet \ar@{-->}[r] \ar@{-->}[u] & \circ
}
\]

\subsection{Fr\"olicher spectral sequence}
The structure of the double complex induces a natural filtration on the total complex of complex differential forms $(\Omega^{\bullet}(X; \mathds{C}), d)$, given by  
$$ F^p \Omega^k(X, \mathds{C}) = \bigoplus_{\substack{r + s = k \\ r \geq p}} \Omega^{r,s}(X). $$  
By the general theory of spectral sequences, see for instance \cite{mccleary}, one obtains the \emph{Frölicher spectral sequence} \cite{frolicher}, namely a sequence of differential complexes $\{(E_r^{\bullet,\bullet}, d_r)\}_r$, where $d_r \colon E_r^{p,q} \to E_r^{p+r, q-r+1}$, such that  
\begin{align*}
E_1^{\bullet,\bullet} &= H_{\overline{\partial}}^{\bullet,\bullet}(X), \quad \text{with } d_1 = [\partial], \\
E_{r+1}^{\bullet,\bullet} &= H^{\bullet,\bullet}(E_r, d_r), \\
E_r^{p,q} &\Rightarrow H^{p+q}_{dR}(X, \mathds{C}).
\end{align*}
The last condition means that $d_r = 0$ for all sufficiently large $r$, so that the sequence stabilizes at some $E_\infty^{\bullet,\bullet}$, and
$$ E_\infty^{p,q} \simeq \frac{F^p H^k_{dR}(X, \mathds{C})}{F^{p+1} H^k_{dR}(X, \mathds{C})},$$  
where $F$ denotes the induced filtration in cohomology.
In particular, since $H^k_{dR}(X, \mathds{C}) = \bigoplus_p \sfrac{F^p H^k_{dR}(X, \mathds{C})}{F^{p+1} H^k_{dR}(X, \mathds{C})}$, and since at each step of the spectral sequence we have $\dim E_{r+1} \leq \dim E_r$, we obtain the \emph{Frölicher inequality}:  
$$ b_k \leq \sum_{p+q=k} h^{p,q}, $$  
where $b_k$ denotes the $k$-th Betti number (that is, the dimension of the de Rham cohomology group $H^k_{dR}(X,\mathds R)$) and $h^{p,q}$ denotes the $(p,q)$-th Hodge number (that is, the dimension of the Dolbeault cohomology group $H^{p,q}_{\overline{\partial}}(X)$).
Recently, Bei and Piovani \cite{bei-piovani} proposed a new proof of the Fr\"olicher inequality based on the use of harmonic forms.

\subsection{\texorpdfstring{$\partial\overline\partial$}{partial overline partial}-Lemma property and formality}

Note that the identity map induces natural morphisms between the various cohomology groups introduced above. Together with the Fr\"olicher spectral sequence, this yields the following diagram:
$$ \xymatrix{
& H_{BC}(X) \ar[d] \ar[dl] \ar[dr] & \\
H_{\overline\partial} \ar@{=>}[r] \ar[dr] & H_{dR} \ar[d] & H_{\partial} \ar@{=>}[l] \ar[dl] \\
& H_{A} &
} $$

A compact holomorphic manifold is said to satisfy the \emph{$\partial\overline\partial$-lemma property} if the natural map $H_{BC}(X) \to H_{A}(X)$ induced by the identity is injective. This is equivalent to the condition that all the maps in the above diagram are isomorphisms, thanks to \cite[Lemma 5.15]{deligne-griffiths-morgan-sullivan}. In other words, every $d$-closed $(p,q)$-form is $d$-exact if and only if it is $\partial\overline\partial$-exact.
This condition is also equivalent to: {\itshape (i)} the degeneration of the Fr\"olicher spectral sequence at the first page, {\itshape i.e.}, $E_\infty = E_1$, {\itshape (ii)} together with the property that the filtration induces a Hodge structure of weight $k$ on the de Rham cohomology group $H^k_{dR}(X, \mathds{C})$, namely,
$$ H^k_{dR}(X, \mathds{C}) = \bigoplus_{p+q=k} \left( F^p H^k_{dR}(X, \mathds{C}) \cap \overline{F^q H^k_{dR}(X, \mathds{C})} \right), $$
see \cite[Remark 5.21]{deligne-griffiths-morgan-sullivan}.
We will denote $H^{p,q}(X)$ the subgroups in the right-hand side, which are canonically isomorphic to $H^{p,q}_{\overline\partial}(X)$, as well as to any other cohomology, under the $\partial\overline\partial$-lemma.
In terms of the indecomposable elements, this means that there are no zigzags of length greater than one, but only squares and dots.

This property is crucial, for instance, in describing the set of representatives of a real de Rham cohomology class of type $(1,1)$. Indeed, given $[\alpha] \in H^{1,1}(X)$, we have
$$
[\alpha] = \left\{ \alpha + \sqrt{-1} \, \partial\overline\partial f \,:\, f \in \mathcal{C}^\infty(X; \mathds{R}) \right\},
$$
so that a cohomological condition reduces to a scalar problem.

Another crucial consequence of the $\partial\overline\partial$-Lemma is the following.
According to Sullivan's theory \cite{sullivan}, the rational homotopy type of a differentiable manifold is encoded in the differential graded algebra (DGA) of differential forms. A manifold is said to be \emph{formal} if its de Rham complex $(\Omega^\bullet(X), d)$ is connected to its cohomology algebra $(H^\bullet(X,\mathds{R}), d = 0)$ by a chain of quasi-isomorphisms of DGAs. In particular, this means that the rational homotopy type of a formal manifold is completely determined by its de Rham cohomology.

\begin{theorem}[{\cite{deligne-griffiths-morgan-sullivan}}]
Compact complex manifolds satisfying the $\partial\overline\partial$-Lemma are formal.
\end{theorem}

\begin{proof}
Consider the operator \( d^c = J^{-1} d J = -\sqrt{-1}(\partial - \overline\partial) \), and the following diagram:
\[
\xymatrix{
& (\ker d^c, d) \ar[ld]_i \ar[rd]^p & \\
(\Omega^\bullet(X), d) && (H^\bullet_{d^c}(X), d)
}
\]
where \( H^\bullet_{d^c}(X) := \ker d^c / \operatorname{im} d^c \).
To prove formality, it is enough to show that, under the assumption of the \(\partial\overline\partial\)-Lemma property, both \(i\) and \(p\) induce isomorphisms in cohomology, and that the induced differential on \(H^\bullet_{d^c}(X)\) vanishes. Observe that the \(\partial\overline\partial\)-Lemma can also be formulated in terms of the real operators \( d \) and \( d^c \), in place of $\partial$ and $\overline\partial$, which is often preferable when working with real-valued forms.

We start by proving that \( i^* \) is surjective. Let \( a = [x] \in H_{dR}(X, \mathds{R}) \). Then \( d^cx \) is both \( d \)-closed and \( d^c \)-exact. The \(\partial\overline\partial\)-Lemma guarantees that \( d^cx \) is also \( dd^c \)-exact, that is, there exists \( z \) such that $d^cx = dd^cz$.
It follows that \( \alpha + dz \) is \( d^c \)-closed and represents the same de Rham class as \( [x] \), hence \( a = i^*([x + dz]) \). Therefore, \( i^* \) is surjective.
To show that \( i^* \) is also injective, suppose that \( y \in \ker d^c \) is \( d \)-exact. Then \( y = dz \), and the \(\partial\overline\partial\)-Lemma implies that \( y \) is also \( dd^c \)-exact, {\itshape i.e.} \( y = dd^cw \) for some \( w \). Hence, $y = d(d^cw)$, and \( d^cw \in \ker d^c \), so \( y \) is the \( d \)-differential of a \( d^c \)-closed form. Thus, \( \beta \) represents the zero class in \( H(\ker d^c, d) \), proving injectivity.

Note that the differential induced by \( d \) on \( H_{d^c}(X) \) is trivial. Indeed, if \( d^c y = 0 \), then the \(\partial\overline\partial\)-Lemma applied to \( d y \) implies that \( d y = d d^c w \) for some \( w \). In particular, \( d y \in \im d^c \), so \( [d y] = 0 \) in \( H_{d^c}(X) \).

Finally, we prove that \( p^* \) is an isomorphism. Let \([y] \in H_{d^c}^\bullet(X) \). Then \( dy \) is both \( d^c \)-closed and \( d \)-exact. By the \(\partial\overline\partial\)-Lemma, it follows that \( dy = dd^c w \) for some form \( w \). Thus,
\[
[y] = [y + d^c w] \quad \text{in } H^\bullet_{d^c}(X),
\]
and \( d(y + d^c w) = dy + d d^c w = dy - dy = 0 \), so \( y + d^c w \in \ker d \). This proves that \( p^* \) is surjective.
Conversely, suppose \( y \in \ker d^c \) satisfies \( dy = 0 \) and \([y] = 0 \) in \( H_{d^c}^\bullet(X) \), that is, \( y = d^c w \) for some \( w \). Applying the \(\partial\overline\partial\)-Lemma again, we get \( y = dd^c z \) for some \( z \), which implies that \( y \) is exact in the complex \( (\ker d^c, d) \). Hence, \( p^* \) is injective.
\end{proof}

The wedge product endows the de Rham cohomology with the structure of a graded-commutative algebra. However, in general, it is not possible to choose a set of differential form representatives that is closed under the wedge product. For instance, as first observed by Sullivan \cite{sullivan}, there exists an ``incompatibility between the wedge product and the harmonicity of forms'' on a compact differentiable Riemannian manifold.
By the homotopy transfer principle \cite{kadeishvili}, any choice of cohomology representatives inherits an $A_\infty$-algebra structure in the sense of Stasheff \cite{stasheff}. The first operation is \( m_1 = 0 \), the second operation \( m_2 \) is induced by the wedge product followed by projection, and the higher operations \( m_k \) for \( k \geq 3 \) are related to Massey products \cite{lu-palmieri-wu-zhang, bujis-morenofernandez-murillo, positselski}.
When this $A_\infty$-structure reduces to a strict algebra, {\itshape i.e.} all \( m_k = 0 \) for \( k \neq 2 \), the manifold is formal.
For a study of Massey products and $A_\infty$-structures in Bott-Chern and Aeppli cohomologies, see \cite{angella-tomassini-JGP, tardini-tomassini, stelzig-AdvMath25, cirici-garridovilallave-sopena, poirier-tradler-wilson, martinmerchan-stelzig}.

\subsection{Numerical characterization of the \texorpdfstring{$\partial\overline\partial$}{partial overline partial}-Lemma property}

The $\partial\overline{\partial}$-Lemma property can be numerically characterized solely in terms of the Bott-Chern cohomology (note that, on a compact manifold, $\dim H^{p,q}_{A}(X) = \dim H^{n-p,n-q}_{BC}(X)$, as we will see in \Cref{sec:hodge}).

\begin{theorem}[{\cite{angella-tomassini-InvMath}}]
Let $X$ be a compact holomorphic manifold. For any $k \in \mathds{Z}$, define the integer
\[
\Delta_k := \sum_{p+q=k} \left( \dim H^{p,q}_{BC}(X) + \dim H^{p,q}_{A}(X) \right) - 2 b_k,
\]
where $b_k$ is the $k$-th Betti number of $X$. Then, $\Delta_k \geq 0$ for all $k$. Moreover, $X$ satisfies the $\partial\overline{\partial}$-Lemma property if and only if $\Delta_k = 0$ for all $k$.
\end{theorem}

\begin{proof}
The idea is that Dolbeault cohomology does not take horizontal arrows into account, while conjugate Dolbeault cohomology ignores vertical arrows. Bott-Chern cohomology counts the indecomposable corner elements that may have incoming arrows, and Aeppli cohomology counts those that may have outgoing arrows, in all cases squares being excluded. 
Therefore, by purely combinatorial arguments, one sees that the sum of the dimensions of Bott-Chern and Aeppli cohomologies is greater than or equal to the sum of the dimensions of Dolbeault and conjugate Dolbeault cohomologies, which in turn is greater than or equal to twice the Betti number, by the Fr\"olicher inequality.
Moreover, both inequalities become equalities if and only if the double complex decomposes as a direct sum of squares and dots, that is, if the manifold satisfies the $\partial\overline{\partial}$-Lemma property.

For a more detailed argument, we recall the exact sequences introduced by Varouchas \cite{varouchas}. Define the following finite-dimensional bi-graded vector spaces:
\[
A^{\bullet,\bullet} := \frac{\operatorname{im}\, \overline\partial \cap \operatorname{im}\, \partial}{\operatorname{im}\, \partial\overline\partial}, \quad 
B^{\bullet,\bullet} := \frac{\ker \overline\partial \cap \operatorname{im}\, \partial}{\operatorname{im}\, \partial\overline\partial}, \quad 
C^{\bullet,\bullet} := \frac{\ker \partial\overline\partial}{\ker \overline\partial + \operatorname{im}\, \partial},
\]
\[
D^{\bullet,\bullet} := \frac{\operatorname{im}\, \overline\partial \cap \ker \partial}{\operatorname{im}\, \partial\overline\partial}, \quad 
E^{\bullet,\bullet} := \frac{\ker \partial\overline\partial}{\ker \partial + \operatorname{im}\, \overline\partial}, \quad 
F^{\bullet,\bullet} := \frac{\ker \partial\overline\partial}{\ker \overline\partial + \ker \partial}.
\]
For simplicity of notation, lowercase letters will denote the dimensions of the corresponding vector spaces, {\itshape e.g.} \(a^{p,q} := \dim A^{p,q}\), \(h^{p,q}_{BC} := \dim H^{p,q}_{BC}(X)\), and similarly for the others. We also define $a^{k} := \sum_{p+q=k} a^{p,q}$, $h^k_{BC} := \sum_{p+q=k} h^{p,q}_{BC}$, etcetera.

One has the following exact sequences \cite[\S3.1]{varouchas}:
\begin{equation}\label{eq:succesatta-1}
0 \to A^{\bullet,\bullet} \to B^{\bullet,\bullet} \to H^{\bullet,\bullet}_{\overline\partial}(X) \to H^{\bullet,\bullet}_{A}(X) \to C^{\bullet,\bullet} \to 0,
\end{equation}
\begin{equation}\label{eq:succesatta-2}
0 \to D^{\bullet,\bullet} \to H^{\bullet,\bullet}_{BC}(X) \to H^{\bullet,\bullet}_{\overline\partial}(X) \to E^{\bullet,\bullet} \to F^{\bullet,\bullet} \to 0 .
\end{equation}
As noted in \cite[\S 3.1]{varouchas}, conjugation and the isomorphisms
\[
\overline{\partial} : C^{\bullet,\bullet} \xrightarrow{\sim} D^{\bullet,\bullet+1}, \quad
\partial : E^{\bullet,\bullet} \xrightarrow{\sim} B^{\bullet+1,\bullet}
\]
imply the relations $a^{p,q} = a^{q,p}$, $f^{p,q} = f^{q,p}$, $d^{p,q} = b^{q,p}$, $e^{p,q} = c^{q,p}$, and $c^{p,q} = d^{p,q+1}$, $e^{p,q} = b^{p+1,q}$.

We are now ready to prove the Fr\"olicher-type inequality for Bott-Chern cohomology. For any \(p, q\), using the conjugation symmetries and the exact sequences above, we have
\begin{align*}
\lefteqn{ h^{p,q}_{BC} + h^{p,q}_{A} }\\
&= h^{p,q}_{BC} + h^{q,p}_{A} \\
&= h^{p,q}_{\overline{\partial}} + h^{q,p}_{\overline{\partial}} + f^{p,q} + a^{q,p} + d^{p,q} - b^{q,p} - e^{p,q} + c^{q,p} \\
&= h^{p,q}_{\overline{\partial}} + h^{p,q}_{\partial} + f^{p,q} + a^{p,q} \\
&\geq h^{p,q}_{\overline{\partial}} + h^{p,q}_{\partial}.
\end{align*}
Since the Fr\"olicher inequality states that $h^{k}_{\overline{\partial}} \geq b_k$, it follows that \(\Delta_k \geq 0\) for all \(k\).

Let us now assume that equality \(\Delta_k = 0\) holds for all \(k\). This implies that both \(a^k = f^k = 0\) and that the Frölicher inequality is an equality. The latter condition means that the Frölicher spectral sequence degenerates at the first page.
We will now prove that the condition \(a^{k+1} = 0\) implies that the natural map
\[
\bigoplus_{p+q=k} H^{p,q}_{dR}(X) \to H^k_{dR}(X, \mathds{C})
\]
is surjective.
This will suffice, since the inequality $h^{k}_{BC} \geq b_k$, together with the condition \(\Delta_k = 0\) and the duality $h^{p,q}_{A} = h^{n-p,n-q}_{BC}$ (see \Cref{sec:hodge}), implies that $h^{k}_{BC} = b_k$, meaning that the natural map is indeed an isomorphism. This is equivalent to the \(\partial\overline{\partial}\)-Lemma property.

To prove the claim, let \(a = [\alpha] \in H^k_{dR}(X, \mathds{C})\). We need to show that \(a\) admits a representative whose pure-type components are \(d\)-closed.
Decompose \(\alpha\) into its pure-type components:
\[
\alpha = \sum_{j=0}^k (-1)^j \alpha^{k-j, j}.
\]
The condition \(d\alpha = 0\) is equivalent to the system
\begin{align*}
\partial \alpha^{k,0} &= 0, \\
\overline{\partial} \alpha^{k-j, j} - \partial \alpha^{k-j-1, j+1} &= 0 \quad \text{for } j = 0, \ldots, k-1, \\
\overline{\partial} \alpha^{0,k} &= 0.
\end{align*}
Note that, for each \(j \in \{0, \ldots, k-1\}\), the \((k+1)\)-form $\overline{\partial} \alpha^{k-j, j} = \partial \alpha^{k-j-1, j+1}$ is both \(\partial\)-exact and \(\overline{\partial}\)-exact. By the assumption \(a^{k+1} = 0\), it follows that this form is also \(\partial \overline{\partial}\)-exact. Therefore, there exists \(\eta^{k-j-1, j} \in \wedge^{k-j-1, j} X\) such that
\[
\overline{\partial} \alpha^{k-j, j} = \partial \alpha^{k-j-1, j+1} = \partial \overline{\partial} \eta^{k-j-1, j} .
\]
It then follows straightforwardly that each pure-type component of
\[
\alpha + d \left( \sum_{j=0}^{k-1} (-1)^j \, \eta^{k-j-1, j} \right)
\]
is closed, thus proving the claim.
\end{proof}

We also have an upper bound for the dimension of Bott-Chern cohomology in terms of Hodge numbers. Note that a purely topological upper bound is not possible, since even-length zigzags contribute to Dolbeault cohomology but not to de Rham cohomology.

\begin{theorem}[{\cite{angella-tardini-PAMS}}]
Let $X$ be a compact holomorphic manifold, of complex dimension $n$. Then, for any $k\in\mathds{Z}$, there holds
\begin{align*}
h^k_{BC} &\leq \min\{k+1, (2n-k)+1\} \cdot \left( h^k_{\overline\partial} + h^{k-1}_{\overline\partial} \right) \\
& \leq (n+1) \cdot \left( h^k_{\overline\partial} + h^{k-1}_{\overline\partial} \right) .
\end{align*}
The same inequality holds true for the Aeppli cohomology. In particular,
$$ | h^k_{A}-h^{k}_{BC} | \leq 2(n+1) \, \left( h^{k}_{\overline\partial} + h^{k+1}_{\overline\partial} \right) . $$
Moreover, $X$ satisfies the $\partial\overline\partial$-Lemma property if and only if
$$ \sum_{k\in\mathds{Z}} \left| h^{k}_{BC} - h^{k}_{A} \right| \;=\; 0 \;. $$
\end{theorem}

\begin{proof}
The key idea is that any zigzag of length \(\ell + 1\), situated between total degrees \(k\) and \(k+1\), produces exactly two non-trivial classes in either Dolbeault or conjugate Dolbeault cohomology at degree \(k\) or \(k+1\), and at most \(\left\lfloor \frac{\ell}{2} \right\rfloor + 1 \leq \min\{k+1, (2n - k) + 1\} \leq n+1\) classes in Aeppli cohomology at degree \(k\).

For the characterization of the \(\partial\overline{\partial}\)-Lemma property, recall that Bott-Chern cohomology counts corners with possible incoming arrows, while Aeppli cohomology counts corners with possible outgoing arrows, except for squares. Therefore, the hypothesis can be restated as follows: for any anti-diagonal, the number of incoming arrows equals the number of outgoing arrows, except for squares. Since no incoming arrow can enter the anti-diagonal of total degree zero, it follows that there are no positive-length zigzags in the entire diagram. This precisely characterizes the \(\partial\overline{\partial}\)-Lemma property.
\end{proof}

The wedge product on differential forms induces a natural algebra structure on Bott-Chern cohomology, while Aeppli cohomology carries only a module structure over \(H_{BC}\). On a \emph{compact} manifold, integration defines a non-degenerate pairing on both de Rham and Dolbeault cohomologies, see \Cref{sec:hodge}. The corresponding duality combines Bott-Chern and Aeppli cohomologies, as established by Schweitzer \cite{schweitzer}.
This duality motivated the construction in \cite{angella-tomassini-JGP}, where both Bott-Chern and Aeppli cohomologies were used to extend the notion of triple Massey products. In the broader context of developing an \(A_\infty\)-algebra structure on Bott-Chern cohomology, potentially involving higher-order Massey products, one might consider assuming that the Bott-Chern pairing is non-degenerate. However, this turns out to be too strong an assumption.

\begin{corollary}[{\cite{angella-tardini-PAMS}}]
Let \(X\) be a compact complex manifold. The natural pairing
\[
H^{\bullet,\bullet}_{BC}(X) \times H^{\bullet,\bullet}_{BC}(X) \to \mathds{C}, \qquad ([\alpha],[\beta]) \mapsto \int_X \alpha \wedge \beta,
\]
is non-degenerate if and only if \(X\) satisfies the \(\partial\overline{\partial}\)-Lemma property.
\end{corollary}

\subsection{\texorpdfstring{$\partial\overline\partial$}{partial overline partial}-Lemma property under geometric transformations}

Here, we briefly recall some results concerning the behaviour of the $\partial\overline\partial$-Lemma property under natural geometric operations.

\begin{theorem}[{\cite{deligne-griffiths-morgan-sullivan}}]
Let $X$ and $Y$ be compact holomorphic manifolds of the same dimension, and let $f \colon X \to Y$ be a holomorphic bimeromorphic map. If $X$ satisfies the $\partial\overline\partial$-Lemma property, then so does $Y$. In particular, Mo\v{\i}\v{s}ezon manifolds \cite{moishezon} and manifolds in class $\mathcal C$ of Fujiki \cite{fujiki-InvMath} satisfy the $\partial\overline\partial$-Lemma property.
\end{theorem}

\begin{theorem}[{\cite{stelzig-JLMS}; see also \cite{angella-suwa-tardini-tomassini, rao-yang-yang}}]
Let $X$ be a compact holomorphic manifold, and let $Z \subseteq X$ be a holomorphic submanifold. If both $X$ and $Z$ satisfy the $\partial\overline\partial$-Lemma property, then the blowup of $X$ along $Z$ also satisfies it. 
In particular, for compact complex threefolds, the $\partial\overline\partial$-Lemma property is invariant under modifications.
\end{theorem}

\begin{theorem}[{\cite{voisin-FR, wu, angella-tomassini-InvMath}}]
Satisfying the $\partial\overline\partial$-Lemma property is a stable property under small deformations of the
complex structure, in the sense of \cite{kodaira-book}.
\end{theorem}

\subsection{Computations of complex cohomologies for nilmanifolds and solvmanifodls}

The explicit computation of cohomology groups of holomorphic manifolds is not an easy problem in general. However, there are classes of manifolds where such computations can be reduced to algebraic equations. This is the case for instance of nilmanifolds and (some) solvmanifolds. First of all, we recall the Chevalley-Eilenberg complex of a Lie algebra.

Let $\mathfrak{g}$ be a Lie algebra and denote with $\left\lbrace x_1,\dots,x_s\right\rbrace$ a basis for $\mathfrak{g}$ and with $\left\lbrace \alpha_1,\dots,\alpha_s\right\rbrace$ the dual basis for $\mathfrak{g}^*$. Define a differential operator $d$ on the exterior algebra $\wedge\mathfrak{g}^*$ as follows. On degree $1$ elements, set
$$
d\alpha^k(x_i,x_j)=-\alpha^k([x_i,x_j]),
$$
and then extend this definition to $\wedge\mathfrak{g}^*$ as a graded bilinear derivation. Setting $[x_i,x_j]=\sum_\ell c_{ij}^\ell x_\ell$, by duality one has $d\alpha^k(x_i,x_j)=-c_{ij}^k$ and so on generators
$$ d\alpha^k=-\sum_{i<j}c_{ij}^k\alpha^i\wedge\alpha^j. $$
As a general formula for $\alpha\in\wedge^k\mathfrak{g}^*$ and $y_1,\dots,y_{k+1}\in\mathfrak{g}$, one has
$$
d\alpha(y_1,\dots,y_{k+1})=
\sum_{1\leq i<j\leq k+1} (-1)^{i+j}\alpha([y_i,y_j], y_1,\dots,\hat{y_i},\dots, \hat{y_j},\dots,y_{k+1})\,.
$$ The Jacobi identity turns out then to be equivalent to $d^2=0$.

The complex $(\wedge\mathfrak{g}^*,d)$ is called the \emph{Chevalley-Eilenberg complex of $\mathfrak{g}$}. If $G$ is a Lie group with Lie algebra $\mathfrak g\simeq T_eG$, any element of $a\in\wedge\mathfrak g^*$ induces a differential form $\alpha$ on $G$ such that $\alpha_e=v$, by means of the differential operators associated to left-translations, and $(d\alpha)_e=dv$. Conversely, any form $\alpha$ on $G$ (or any of its smooth quotients) that is invariant under the action of left-translations (namely $L_g^*\alpha=\alpha$ for any $g\in G$) is completely determined by its value $\alpha_e\in \wedge\mathfrak g^*$. For this reason, the Chevalley-Eilenberg complex can be identified with the subspace of invariant forms (with respect to left-translations, also called left-invariant forms).
The cohomology of the Chevalley-Eilenberg complex $(\wedge\mathfrak{g}^*,d)$ will be denoted by $H^{\bullet}_d(\mathfrak{g},\mathds{R})$.

In the following, we will be interested in the special case where $\mathfrak{g}$ is nilpotent. Recall that if $\mathfrak{g}$ is a Lie algebra we can consider the ascending central series $\mathcal{Z}^0\subseteq \mathcal{Z}^1\subseteq\mathcal{Z}^2\subseteq\cdots$ where $\mathcal{Z}^0:=\mathfrak{g}$ and
$\mathcal{Z}^{i+1}:=\left\lbrace
x\in\mathfrak{g}\,:\,[x,\mathfrak{g}]\subseteq \mathcal{Z}^i\right\rbrace$. In particular, notice that $\mathcal{Z}^1=\mathcal{Z}$ is the center of $\mathfrak{g}$.
The Lie algebra $\mathfrak{g}$ is called \emph{nilpotent} if there exists $n\in\mathds{N}$ such that $\mathcal{Z}^n=\mathfrak{g}$ and the smallest of such $n$ is called the \emph{nilpotency step of} $\mathfrak{g}$. A similar definition can be given at the level of Lie groups and the corresponding subalgebras and subgroups are related via the exponential map. A Lie group is in fact called nilpotent if its Lie algebra is nilpotent.

A \emph{nilmanifold} $\Gamma\backslash G$ is the quotient of a connected, simply-connected nilpotent Lie group $G$ by a lattice $\Gamma$ ({\itshape i.e.} a co-compact discrete subgroup). The following theorem states that the de Rham cohomology of nilmanifolds can be computed using only invariant forms.

\begin{theorem}[{\cite{nomizu}}]
Let $X=\Gamma\backslash G$ be a nilmanifold and denote with $\mathfrak{g}$ the Lie algebra of $G$. Then, the inclusion of invariant forms in the space of forms $\wedge\mathfrak{g}^* \hookrightarrow \Omega^\bullet(X)$ induces an isomorphism in cohomology, namely
$$
H^{\bullet}_d(\mathfrak{g},\mathds{R})
\simeq
H^{\bullet}_{dR}(X,\mathds{R})\,.
$$
\end{theorem}
We recall now the main steps of the proof and we refer also to \cite{rollenske-Nilmanifold} for a survey on the subject.

\begin{proof}
First of all notice that, denoting with $ZG$ the center of $G$, then $Z\Gamma:=\Gamma\cap ZG$ is still a lattice. Since elements in $ZG$ commute with elements in $\Gamma$, the principal $ZG$-bundle
\[
\begin{tikzcd}
ZG \arrow[r] & G \arrow[d] \\
           & \frac{G}{ZG}
\end{tikzcd}
\]
descends to a principal torus bundle
\[
\begin{tikzcd}
T \arrow[r] & X \arrow[d] \\
           & X'
\end{tikzcd}
\]
with fibers real tori $T=\frac{ZG}{Z\Gamma}$.

Since $G$ is nilpotent, this process can be iterated using the ascending central series
$$0=\mathcal{Z}^0\subseteq \mathcal{Z}^1\subseteq\mathcal{Z}^2\subseteq\cdots\subseteq \mathcal{Z}^n=\mathfrak{g}.$$
Indeed, all $\mathcal{Z}^i$ are ideal and $\mathcal{Z}^{i-1}\backslash\mathcal{Z}^i$ is central in $\mathcal{Z}^{i-1}\backslash \mathfrak{g}$ and in particular $\mathcal{Z}^{i-1}\backslash\mathcal{Z}^i$ is abelian.
This corresponds to a filtration of normal subgroups
$$1=Z^0\subseteq Z^1\subseteq Z^2\subseteq\cdots\subseteq Z^n=G$$
with $Z^{i-1}\backslash Z^i$ abelian and $Z^i\Gamma:=\Gamma\cap Z^i$ lattice in $Z^i$.
More precisely, denoting with $
T_i$ the torus obtained as a quotient of
$Z^{i-1}\backslash Z^i$ by $Z^{i-1}\Gamma\backslash Z^i\Gamma$, then there is a tower of principal $T_i$-bundles
\[
\begin{tikzcd}
T_1 \arrow[r] & X_1 \arrow[d, "\pi_1"] \\
T_2 \arrow[r]& X_2 \arrow[d, "\pi_2"] \\
& \vdots \arrow[d] \\
T_{n-1} \arrow[r]& X_{n-1} \arrow[d, "\pi_{n-1}"] \\
& X_n = T_n
\end{tikzcd}
\]
where $X_i$ are nilmanifolds obtained as quotients of $Z^{i-1}\backslash G$ and $Z^{i-1}\Gamma\backslash \Gamma$.

For the proof one can argue now by induction on the nilpotency step $n$. If $n=1$, then $G$ is abelian and so $X$ is a torus and the result is well-known.
For the induction step, assume that we have a principal $T$-bundle over a nilmanifold of lower nilpotency step
\[
\begin{tikzcd}
T \arrow[r] & X \arrow[d] \\
           & X'\,.
\end{tikzcd}
\]
The idea is to describe the cohomology of $X$ using the cohomologies of the fiber $T$ and the base $X'$ and 
this is done using two spectral sequences.

First of all, consider the following filtration on the space of forms $\Omega^\bullet(X)$,
$$
F^p\Omega^{p+q}(X):=
\left\lbrace
\omega\in\Omega^{p+q}(X)\,:\,
\omega(v_1,\dots,v_{p+q})=0 \text{\,for any\,} v_{i_1},\dots,v_{i_{q+1}}\in\ker d\pi
\right\rbrace\,.
$$
This filtration determines a spectral sequence $\left\lbrace (E_r^{\bullet,\bullet},d_r)\right\rbrace_{r\geq 0}$ which computes the cohomology of the complex $(\Omega^\bullet(X),d)$. This spectral sequence is called \emph{Leray-Serre spectral sequence} and in this case one has
$$
E^{p,q}_2=H^{p}_{dR}(X')\otimes H^q_{dR}(T)\Rightarrow H^{p+q}_{dR}(X)\,.
$$
Now one can argue similarly at the level of invariant forms. The fibration
\[
\begin{tikzcd}
T \arrow[r] & X \arrow[d] \\
           & X'\,.
\end{tikzcd}
\]
corresponds to a short exact sequence of Lie algebras
$$
0\to \mathfrak{h}\to \mathfrak{g}\to \frac{\mathfrak{g}}{\mathfrak{h}}\to 0\,.
$$
For $\wedge^\bullet\mathfrak{g}^*$ one obtains the \emph{Hochschild-Serre spectral sequence} $\left\lbrace (\tilde E_r^{\bullet,\bullet},d_r)\right\rbrace_{r\geq 0}$ and in this case one has
$$
\tilde E^{p,q}_2=H^{p}_{d}\left(\frac{\mathfrak{g}}{\mathfrak{h}}\right)\otimes H^q_{d}(\mathfrak{h})\Rightarrow H^{p+q}_{d}(\mathfrak{g})\,.
$$
Now, since the inclusion of invariant forms in the space of forms, $\wedge\mathfrak{g}^* \hookrightarrow \Omega^\bullet(X)$, is compatible with the filtrations, one has a homomorphism of spectral sequences. In particular, at the second page, one has an isomorphism by induction hypothesis:
$$
H^{p}_{dR}(X')\otimes H^q_{dR}(T)\to H^{p}_{dR}\left(\frac{\mathfrak{g}}{\mathfrak{h}}\right)\otimes H^q_{dR}(\mathfrak{h})\,.
$$
As a consequence, passing to the limit, we have the desired isomorphism,
$$
H^{p+q}_{dR}(X)\simeq H^{p+q}_{d}(\mathfrak{g}) ,
$$
completing the proof of the theorem
\end{proof}

Nomizu's theorem states that the de Rham cohomology of the nilmanifold \( M \) is entirely determined by the finite-dimensional complex \( \wedge^\bullet \mathfrak{g}^* \), which indeed serves as a Sullivan model \cite{sullivan}. This result extends more generally to compact quotients of completely-solvable Lie groups \cite{hattori}, but not to arbitrary solvmanifolds \cite{debartolomeis-tomassini-AIF}. Nevertheless, various techniques have been developed to compute the de Rham cohomology for solvmanifolds \cite{guan, console-fino-SNS, console-fino-kasuya}.

Analogous results hold for the Dolbeault cohomology of special classes of complex structures on nilmanifolds, see  for instance \cite{console-fino, rollenske-Nilmanifold} and the references therein. More precisely,
if $J$ is a complex structure on $\mathfrak{g}$, then denoting with $\mathfrak{g}_{\mathds{C}}$ the complexified Lie algebra, one has the usual decompositions
$$
\mathfrak{g}_{\mathds{C}}^*=
\mathfrak{g}^{*1,0}\oplus\mathfrak{g}^{*0,1}
$$
and
$$
\wedge^{p,q}\mathfrak{g}^*=
\wedge^p\mathfrak{g}^{*1,0}\otimes\wedge^q\mathfrak{g}^{*0,1}
$$
and accordingly $d=\partial+\overline\partial$. Then one has the complex $(\wedge^{p,\bullet}\mathfrak{g}^*,\overline\partial)$ and the associated Dolbeault cohomology $H^{p,q}_{\overline\partial}(\mathfrak{g}_{\mathds{C}})=H^q(\wedge^{p,\bullet}\mathfrak{g}^*,\overline\partial)$.
It is conjectured that, for a nilmanifold $X=\Gamma\backslash G$ with invariant complex structure $J$, the analogue of Nomizu's Theorem should hold for the Dolbeault cohomology, that is
$$
H^{p,q}_{\overline\partial}(\mathfrak{g}_{\mathds{C}})\simeq H^{p,q}_{\overline\partial}(X).
$$
The problem in adapting Nomizu's argument is that not every nilmanifold with invariant complex structure is a holomorphic principal bundle, for instance the projections in the tower bundle above could have odd-dimensional center. The conjecture is known to be true in some special cases, for instance when:
\begin{itemize}
    \item[-] $X$ is complex parallelizable \cite{sakane};
    \item[-] $J$ is an abelian complex structure \cite{console-fino-poon};
    \item[-] $J$ is a nilpotent complex structure \cite{cordero-fernandez-gray-ugarte};
    \item[-] $J$ is a rational complex structure \cite{console-fino};
    \item[-] $X$ is a $6$-dimensional nilmanifold \cite{fino-rollenske-ruppenthal}.
\end{itemize}

Kasuya \cite{kasuya-MathZ} developed techniques for computing the Dolbeault cohomology of solvmanifolds. These techniques can also be extended to Bott-Chern and Aeppli cohomologies \cite{angella-JGA, angella-kasuya-AGAG}.
\begin{theorem}[\cite{angella-JGA, angella-kasuya-AGAG}]
Let $X=\Gamma\backslash G$ be a solvmanifold endowed with an invariant complex structure $J$. If 
$$
H^{\bullet}_{d}(\mathfrak{g})\simeq H^{\bullet}_{dR}(X)
\quad\text{and}\quad
H^{\bullet,\bullet}_{\overline\partial}(\mathfrak{g}_{\mathds{C}})\simeq H^{\bullet,\bullet}_{\overline\partial}(X)
$$
then
$$
H^{\bullet,\bullet}_{BC}(\mathfrak{g}_{\mathds{C}})\simeq H^{\bullet,\bullet}_{BC}(X)\,.
$$
\end{theorem}

\ifdefined\frommain
\else
    \bibliographystyle{alpha}
    \bibliography{biblio}
\fi

\subsection{Problems}

\begin{exercise}
Let $X$ be a differentiable manifold endowed with an almost-complex structure $J$. Prove that $J$ is integrable if and only if $d=\partial+\overline\partial$.
\end{exercise}

\begin{exercise}[{\cite[Lemma 5.15]{deligne-griffiths-morgan-sullivan}}]
Let $(A^{\bullet,\bullet},\partial,\overline\partial)$ be a bounded double complex of vector spaces, and denote by $(A^\bullet:=\bigoplus_{p+q=\bullet}A^{p,q},d=\partial+\overline\partial)$ the corresponding total complex.
Prove that, for any $k$, the following conditions are equivalent:
\begin{description}
\item[$(a)_n$] $\ker\partial\cap\ker\overline\partial\cap\im d=\im\partial\overline\partial$ in $A^k$,
\item[$(b)_n$] $\ker\overline\partial\cap\im\partial=\ker\partial\cap\im\overline\partial=\im\partial\overline\partial$ in $A^k$,
\item[$(c)_n$] $\ker\partial\cap\ker\overline\partial\cap(\im\partial+\im\overline\partial)=\im\partial\overline\partial$ in $A^k$,
\item[$(a^*)_{n-1}$] $\im\partial+\im\overline\partial+\ker d=\ker\partial\overline\partial$ in $A^{k-1}$,
\item[$(b^*)_{n-1}$] $\im\overline\partial+\ker\partial=\im\partial+\ker\overline\partial=\ker\partial\overline\partial$ in $A^{k-1}$,
\item[$(c^*)_{n-1}$] $\im\partial+\im\overline\partial+(\ker\partial\cap\ker\overline\partial)=\ker\partial\overline\partial$ in $A^{k-1}$.
\end{description}
\end{exercise}

\begin{exercise}[{\cite[Lemma 9]{console-fino}}]
Let $G$ be a Lie group, and denote by $\mathfrak g$ its associated Lie algebra. Let $\Gamma$ be a discrete subgroup, acting on $G$ on the left, such that the quotient $X=\sfrac{G}{\Gamma}$ is compact.
Consider the space of left-invariant forms on $G$, which is isomorphic to the finite-dimensional vector space $\wedge^\bullet\mathfrak g^*$ and induces a subcomplex $\Omega^\bullet (X)^G$ of the de Rham complex $(\Omega^\bullet(X),d)$. Prove that the inclusion induces an injective map in cohomology.
\end{exercise}

\begin{proof}[Hint]
Construct a left-invariant metric on $X$, and use the fact that each cohomology class admits a unique harmonic representative.
\end{proof}

\begin{exercise}
Let \( G = \mathrm{Heis}(3, \mathds{C}) \) be the three-dimensional Heisenberg group, consisting of matrices of the form
\[
\begin{pmatrix}
1 & z_1 & z_3 \\
0 & 1 & z_2 \\
0 & 0 & 1
\end{pmatrix},
\]
where \( z_1, z_2, z_3 \in \mathds{C} \).  
Consider the subgroup \( \Gamma \subset G \) consisting of matrices whose entries are Gaussian integers, acting on \( G \) by left multiplication.  
The quotient \( X = \sfrac{G}{\Gamma} \) is a compact complex manifold called the \emph{Iwasawa manifold}, whose holomorphic tangent bundle is trivial \cite{nakamura}.  
Compute the complex cohomologies of \( X \).
\end{exercise}

\begin{proof}[Hint]
Fix the left-invariant coframe
\[ \varphi^1 = dz_1 , \quad \varphi^2 = dz_2 , \quad \varphi^3 = dz_3 - z_1\,dz_2 \]
for the holomorphic cotangent bundle. The corresponding structure equations are:
\[ d\varphi^1 = 0, \quad d\varphi^2 = 0, \quad d\varphi^3 = -\varphi^1 \wedge \varphi^2. \]
Show that the double complex of left-invariant forms has the structure in \Cref{fig:iwasawa}.  
Then compute the various cohomologies by counting the contributions from each component, as depicted in \Cref{fig:cohomology-iwasawa}.
\end{proof}

\begin{figure}[p]
\centering
\includegraphics[width=5.5cm]{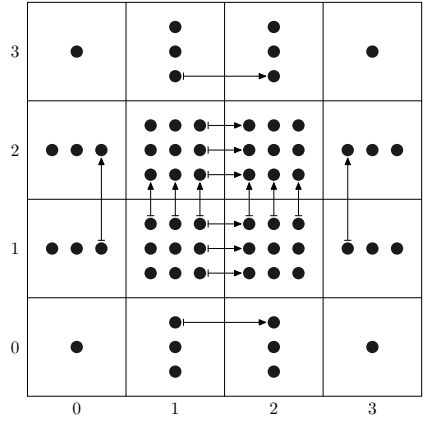}
\caption{\small The double complex of left-invariant forms on the Iwasawa manifold.}
\label{fig:iwasawa}
\end{figure}

\begin{figure}[p]
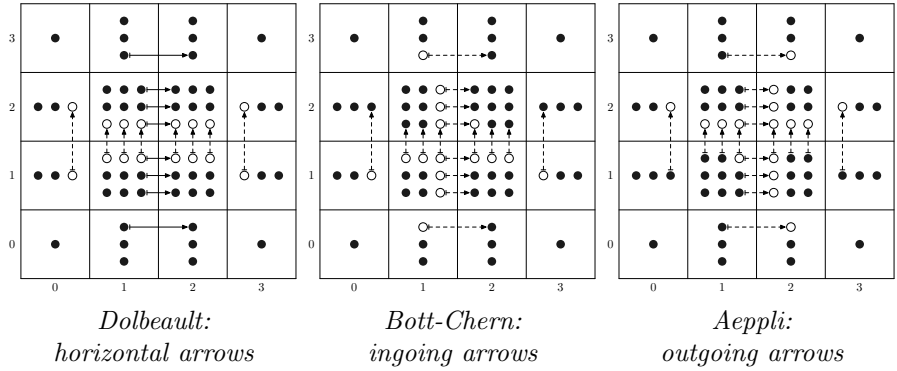

    \centering
    \begin{subfigure}[b]{0.28\textwidth}
        \includegraphics[width=\textwidth]{./images/iwasawa.2}
        \captionsetup{labelformat=empty, font={footnotesize,it}, justification=centering}
        \caption{Dolbeault:\\horizontal arrows}
    \end{subfigure}
    \begin{subfigure}[b]{0.28\textwidth}
        \includegraphics[width=\textwidth]{./images/iwasawa.4}
        \captionsetup{labelformat=empty, font={footnotesize,it}, justification=centering}
        \caption{Bott-Chern:\\ingoing arrows}
    \end{subfigure}
    \begin{subfigure}[b]{0.28\textwidth}
        \includegraphics[width=\textwidth]{./images/iwasawa.5}
        \captionsetup{labelformat=empty, font={footnotesize,it}, justification=centering}
        \caption{Aeppli:\\outgoing arrows}
    \end{subfigure}
    \caption{Cohomologies of the Iwasawa manifold.}
    \label{fig:cohomology-iwasawa}
\end{figure}

\begin{exercise}
Let \( X \) be a differentiable manifold. Define a partially defined operation on the de Rham cohomology \( H^\bullet_{dR}(X, \mathds{R}) \) as follows.
Let $a_{12} = [\alpha_{12}] \in H^{|\alpha_{12}|}_{dR}(X, \mathds{R})$, $a_{23} = [\alpha_{23}] \in H^{|\alpha_{23}|}_{dR}(X, \mathds{R})$, $a_{34} = [\alpha_{34}] \in H^{|\alpha_{34}|}_{dR}(X, \mathds{R})$ be three cohomology classes such that
\[
a_{12} \cdot a_{23} = 0 \quad \text{and} \quad a_{23} \cdot a_{34} = 0.
\]
This means there exist forms \( \alpha_{13} \) and \( \alpha_{24} \) such that
\[
(-1)^{|\alpha_{12}|}\alpha_{12} \wedge \alpha_{23} = d\alpha_{13}, \quad
(-1)^{|\alpha_{23}|}\alpha_{23} \wedge \alpha_{34} = d\alpha_{24}.
\]
Then the \emph{triple Massey product} is defined as
\[
\langle a_{12}, a_{23}, a_{34} \rangle := \left[ (-1)^{|\alpha_{12}|} \alpha_{12} \wedge \alpha_{24} + (-1)^{|\alpha_{13}|} \alpha_{13} \wedge \alpha_{34} \right] .
\]
Show that this class is well-defined in $H^{|\alpha_{12}| + |\alpha_{23}| + |\alpha_{34}| - 1}_{dR}(X, \mathds{R})$ modulo the indeterminacy
\[
a_{12} \cdot H^{|\alpha_{23}| + |\alpha_{34}| - 1}_{dR}(X, \mathds{R}) + a_{34} \cdot H^{|\alpha_{12}| + |\alpha_{23}| - 1}_{dR}(X, \mathds{R}).
\]
Inductively extend this construction to define higher-order Massey products.
Finally, prove that on a formal manifold, all Massey products of any order vanish.
\end{exercise}

\begin{exercise}[{\cite{cordero-fernandez-gray, benson-gordon, hasegawa}}]
Prove that a nilmanifold satisfies the \(\partial\overline\partial\)-Lemma if and only if it is a complex torus.
\end{exercise}

\begin{exercise}[{\cite{voisin-FR, wu, angella-tomassini-InvMath}}]
Prove that the $\partial\overline\partial$-Lemma property is stable under small deformations of the
complex structure, in the sense of \cite{kodaira-book}.
\end{exercise}

\begin{proof}[Hint]
Use the fact that, thanks to Hodge theory, the dimensions of the cohomology groups vary upper semi-continuously under small deformations. Then exploit the numerical characterization of the $\partial\overline\partial$-Lemma property.
Note that the property is not necessarily closed under deformations, see for example \cite{angella-kasuya-NWEJM}.
\end{proof}

\ifdefined\frommain
\else
    \bibliographystyle{alpha}
    \bibliography{biblio}
\fi

\cleardoublepage

\section{Canonical Hermitian metrics and connections}\label{sec:metric}

In this section, we introduce the fundamental tools of Hermitian geometry and review several results concerning canonical Hermitian metrics and connections, largely inspired by the foundational work of Paul Gauduchon. The exposition partly follows the recent survey \cite{angella-LeMat}. In the final section, we focus on threefolds admitting locally conformally Kähler metrics and present some partial results toward their classification, following \cite{angella-parton-vuletescu, angella-petcu-vuletescu}.

\subsection{Hermitian geometry}\label{sec:hodge}

Thanks to the paracompactness property, we can take the (pullback of) the standard Hermitian inner product on each holomorphic chart and then glue them together using a partition of unity. This construction yields a {\em Hermitian metric} on $X$, defined as a family of Hermitian inner products $h_x$ on each complex vector space $(T_xX, J_x)$, varying smoothly with $x\in X$.
Via the linear isomorphisms $(T_xX, J_x) \simeq (T^{1,0}_xX, \sqrt{-1}) \simeq (T^{0,1}X, -\sqrt{-1})$, this induces a family of Hermitian inner products on each fibre of $TX\otimes_{\mathds R}\mathds C$. Decomposing this inner product into real and imaginary parts, say $g-\sqrt{-1}\omega$, we obtain: a Riemannian metric $g$ on $X$, for which $J$ acts as isometry; a $J$-invariant $2$-form $\omega$ on $X$, satisfying $\omega(V,JV)>0$ pointwise for any nowhere-zero vector field $V$. Since $\omega=g(J\_,\_)$, we will refer to a Hermitian metric either as $g$ or $\omega$ without ambiguity.

Once a Hermitian metric is fixed on a compact holomorphic manifold $X$, the associated volume form $\omega^n$ induces an inner product on the space of global differential forms. Indeed, the pointwise Hermitian inner product extends naturally to tensor spaces, and hence to the space of differential forms. For $\alpha, \beta \in \Omega^k(X)$, we define
$$
\langle \alpha, \beta \rangle := \int_X g_x(\alpha_x, \beta_x)\, \omega^n .
$$
This allows us to define the codifferential operator $d^* \colon \Omega^k(X) \to \Omega^{k-1}(X)$ as the formal adjoint of the exterior differential:
$$
\langle d^* \alpha, \beta \rangle = \langle \alpha, d\beta \rangle \quad \text{for all } \alpha \in \Omega^k(X),\, \beta \in \Omega^{k-1}(X).
$$
Analogously, we define the formal adjoints of the Dolbeault operators:
$$
\partial^* \colon \Omega^{p,q}(X) \to \Omega^{p-1,q}(X), \qquad \overline\partial^* \colon \Omega^{p,q}(X) \to \Omega^{p,q-1}(X).
$$
We can give an explicit expression for the adjoint operators by using the Hodge star operator. Define the $\mathds{C}$-linear map
$$ * \colon \Omega^{p,q}(X) \to \Omega^{n-q, n-p}(X) $$
by the condition that, for any $\alpha, \beta \in \Omega^{p,q}(X)$,
$$ \alpha \wedge *\bar{\beta} = g(\alpha, \beta) \, \omega^n. $$
Therefore, (noting that the real dimension of $X$ is even,) we obtain the identities:
$$
d^* = - * d *, \qquad
\partial^* = - * \overline{\partial} *, \qquad
\overline{\partial}^* = - * \partial *.
$$

Lastly, we define the de Rham Laplacian by
$$
\Delta := [d, d^*] = dd^* + d^*d,
$$
which, locally and up to a sign, coincides with the classical analytic Laplacian on functions.
Similarly, we define the Dolbeault Laplacians:
$$
\overline{\Box} := \overline{\partial} \, \overline{\partial}^* + \overline{\partial}^* \, \overline{\partial},
\qquad
\Box := \partial \, \partial^* + \partial^* \, \partial.
$$
All these operators are self-adjoint elliptic differential operators, see {\itshape e.g.} \cite[Section~IV.2]{wells}.
In particular, we can apply the Hodge theorem \cite{hodge} to obtain an orthogonal decomposition (analogously for the other operators):
$$
\Omega^{\bullet,\bullet}(X) = \ker \overline{\Box} \oplus \im \overline{\partial} \oplus \im \overline{\partial}^*,
$$
from which we deduce the isomorphism
$$
H^{\bullet,\bullet}_{\overline{\partial}}(X) \simeq \ker \overline\Box.
$$
We can also construct a Laplacian suitable for computing Bott–Chern cohomology. Starting from the non-elliptic operator $\overline{\partial}^* \overline{\partial} + \partial^* \partial + (\partial \overline{\partial})(\partial \overline{\partial})^*$, one can modify it to obtain an elliptic operator without altering its kernel:
$$
\tilde{\Delta}_{BC} := (\partial \overline{\partial})(\partial \overline{\partial})^* + (\partial \overline{\partial})^*(\partial \overline{\partial}) + (\overline{\partial}^* \partial)(\overline{\partial}^* \partial)^* + (\overline{\partial}^* \partial)^*(\overline{\partial}^* \partial) + \overline{\partial}^* \overline{\partial} + \partial^* \partial.
$$
This operator was first introduced in \cite{kodaira-spencer-3}, and later studied in detail by Schweitzer \cite{schweitzer}. Also in this case, one has a Hodge-type decomposition and an isomorphism
$$ H^{\bullet,\bullet}_{BC}(X) \simeq \ker \tilde{\Delta}_{BC}, $$
see \cite{schweitzer}. Analogous results hold for Aeppli cohomology, by considering
$$
\tilde{\Delta}_{A} := \partial\partial^*+\overline\partial\overline\partial^*+(\partial \overline{\partial})^*(\partial \overline{\partial}) + (\partial \overline{\partial})(\partial \overline{\partial})^* + (\overline{\partial} \partial^*)^*(\overline{\partial} \partial^*) + (\overline{\partial} \partial^*)(\overline{\partial} \partial^*)^*.
$$

There are some immediate applications of the harmonic theory. First, on a compact holomorphic manifolds, de Rham cohomology, Dolbeault cohomology, Bott-Chern and Aeppli cohomologies are finite-dimensional vector spaces.
Second, note that $*\Delta=\Delta*$, while $*\Box=\overline\Box*$ and $*\tilde\Delta_{BC}=\tilde\Delta_{A}*$, which realize the dualities
\begin{align*}
H^k_{dR}(X,\mathds C) &\simeq H^{2n-k}_{dR}(X,\mathds C) \quad \text{(Poincaré)} ,\\
H^{p,q}_{\overline\partial}(X) &\simeq H^{n-p,n-q}_{\overline\partial}(X) \quad \text{(Serre)} , \\
H^{p,q}_{BC}(X) &\simeq H^{n-p,n-q}_{A}(X) \quad \text{(Schweitzer)} . \\
\end{align*}

\subsection{Riemannian geometry in a nutshell}

To fix notation and motivate the next section, we begin by recalling the underlying Riemannian framework. For further details and a comprehensive treatment, we refer the reader to the classic reference \cite{besse}.

Let $g$ be a Riemannian metric on a differentiable manifold $M$. The {\em Levi-Civita connection} is the unique affine connection $D$ on $M$ that preserves $g$ ({\itshape i.e.}, $D g=0$) and is torsion-free ({\itshape i.e.}, $T=0$ where $T(V, W) := D_V W - D_W V - [V, W]$). This connection is fully determined by the Koszul formula: $2g(D_VW,Z) = Vg(W, Z) + W g(Z, V) - Zg(V, W) + g([V, W], Z) - g([W, Z], V) + g([Z, V], W)$.
It defines a {\em curvature} tensor $R(V,W)= D_{[V,W]}-[D_V, D_W]$, which can be interpreted as a $\mathrm{End}(TX)$-valued $2$-form.
We can also interpret the curvature in the following way. Consider $\alpha$ to be a differential $k$-form with values in $TM$. Define the exterior derivative $d^D$ associated with $D$ by $(d^D\alpha)(V_0,\ldots,V_k)=\sum_i (-1)^i D_{V_i}(\alpha(V_0, \ldots, \hat{V_i}, \ldots, V_n)) + \sum_{i\neq j} (-1)^{i+j} \alpha([V_i,V_j], V_0, \ldots, \hat{V_i}, \ldots, \hat{X_j}, \ldots, X_n)$, where ``$\hat{X_i}$'' means skipped. Then $R=-(d^D)^2$ measure the failure of $d^D$ to be a genuine differential.
The curvature tensor satisfies several additional symmetries. It is clear that $R(V,W)=-R(W,V)$, and we also have the identity $g(R(V, W )Z, Y) = -g(R(V, W) Y, Z)$. Moreover, $R\wedge \mathrm{id}=0$ (known as the first, or algebraic, Bianchi identity) and $d^D R = 0$ (the second, or differential, Bianchi identity), see \cite{gauduchon-book} for details.
The curvature tensor contains exactly the information about the {\em sectional curvatures} $K(V,W)= \sfrac{g(R(V,W)V,W)}{(g(V,V)g(W,W)-g(V,W)^2)}$ of any plane $\sigma_p=\mathrm{span}\{V,W\} \subseteq T_pX$ at $p$, which correspond to the Gaussian curvature of the surface spanned by geodesics whose tangent vector at $p$ is in $\sigma_p$.
Still, interesting information is encoded in the {\em Ricci curvature}, defined as $\mathrm{Ric}(V,W):=\mathrm{tr}(Z \mapsto R(V,Z)W)$.
Lastly, another important metric invariant is the function $s=\mathrm{tr}_gr$, called the {\em scalar curvature}.

In order to find {\em canonical metrics}, unique up to automorphisms and encoding the topological or differentiable properties of the manifold, several problems must be addressed.

\begin{itemize}
\item Complete Riemannian metrics with constant sectional curvature give rise to the notion of {\em space forms}. In each dimension $n$, there are only three simply-connected model space forms, up to homothety: the Euclidean space $\mathds{R}^n$ with its flat metric ($R = 0$); the unit sphere $S^n$ with the round metric of constant positive sectional curvature ($R = +1$); and the hyperbolic space $H^n$, which has constant negative sectional curvature ($R = -1$).
 
\item 
Metrics \( g \) satisfying 
\[
\mathrm{Ric}_g = \lambda \cdot g
\]
for some constant \(\lambda\) are called \emph{Einstein metrics}. This condition relates closely to the Einstein field equations \(\mathrm{Ric} - \tfrac{1}{2} s g = T\), where \(T\) is the energy-momentum tensor; in vacuum (\(T=0\)) solutions correspond to Ricci-flat metrics. Einstein metrics are critical points of the total scalar curvature functional \( S(g) = \int_M s_g \, \mu_g \) restricted to metrics with fixed volume. As noted in \cite[page 6]{besse}, finding explicit examples is challenging, and known constructions often rely on large symmetry groups or coupled equations.
A key existence result is the Aubin-Yau theorem \cite{aubin, yau}, proving the existence of K\"ahler-Einstein metrics on compact K\"ahler manifolds with non-positive first Chern class \(c_1(X) \leq 0\). Recent advances \cite{chen-donaldson-sun-1, chen-donaldson-sun-2, chen-donaldson-sun-3, tian} also address existence on Fano manifolds (\(c_1(X) > 0\)). 
Simply-connected manifolds with trivial canonical bundle, called {\em Calabi-Yau manifolds}, admit Ricci-flat K\"ahler metrics with holonomy in \(\mathrm{SU}(n)\). These manifolds play a central role in string theory compactifications due to their supersymmetry-preserving properties \cite{candelas-horowitz-strominger-witten}.
Einstein metrics are stationary points of the normalized Ricci flow
\[
\frac{\partial}{\partial t} g(t) = -2 \mathrm{Ric}_{g(t)} + \frac{2}{n} r_{g(t)} \cdot g(t),
\]
where \(n = \dim M\) and \(r_g\) is the average scalar curvature. Hamilton \cite{hamilton} proved long-time existence and convergence to Einstein metrics for compact three-dimensional Riemannian manifolds with positive Ricci curvature. The Ricci flow preserves the K\"ahler condition, yielding the K\"ahler-Ricci flow. Cao \cite{cao} used this flow to reprove the Aubin-Yau theorem by showing convergence to the unique K\"ahler-Einstein metric for compact K\"ahler manifolds with \(c_1(X) \leq 0\).

\item The study of {\em constant scalar curvature metrics} was initiated by Yamabe \cite{yamabe} and resolved by combined efforts of Trudinger \cite{trudinger}, Aubin \cite{aubin-2}, and Schoen \cite{schoen}, who used the positive mass theorem from General Relativity \cite{schoen-yau}. They proved that any Riemannian metric \(g\) on a compact \(n\)-dimensional manifold can be conformally rescaled to \(u^{\frac{4}{n-2}} g\) with constant scalar curvature, reducing the problem to solving the PDE
\[
\frac{4(n-1)}{n-2} \Delta_g u - s_g u + c u^{\frac{n+2}{n-2}} = 0,
\]
where \(\Delta_g\) is the Laplace–Beltrami operator and \(c\) a constant.
In the K\"ahler setting, scalar curvature admits a symplectic interpretation as a moment map for an infinite-dimensional Hamiltonian action \cite{donaldson, fujiki}. The existence of constant scalar curvature K\"ahler (cscK) metrics in a fixed K\"ahler class was recently established by Chen and Cheng \cite{chen-cheng-1, chen-cheng-2}. These are special cases of {\em extremal metrics} \cite{calabi}, which minimize the Calabi functional. For a comprehensive introduction, see \cite{gauduchon-book, szekelyhidi, guedj-zeriahi}.
\end{itemize}

\subsection{Hermitian connections and curvatures}

In the Hermitian setting, a key issue is that the Levi-Civita connection generally does not preserve the Hermitian structure unless the metric is K\"ahler. Let \(X\) be a holomorphic manifold of dimension \(n\), endowed with a Hermitian metric \(g\) and associated \((1,1)\)-form \(\omega = g(J\_, \_)\), where \(J\) is the complex structure. To respect both \(g\) and \(J\), one must consider linear connections \(\nabla\) satisfying \(\nabla g = 0\) and \(\nabla J = 0\), which in general come with torsion \(T \ne 0\). The space of such Hermitian connections forms an affine space modeled on \(\Omega^{1,1}(X, TX)\).

In his celebrated paper \cite{gauduchon-BUMI}, Gauduchon introduced a one-parameter family of Hermitian connections \(\{\nabla^t\}_{t\in\mathds{R}}\), now known as the \emph{Gauduchon connections}, defined by
\[
g(\nabla^t_V W, Z) = g(D_V W, Z) - \frac{1 - t}{4} d\omega(JV, JW, JZ) + \frac{1 + t}{4} d\omega(JV, W, Z).
\]
All Gauduchon connections coincide with the Levi-Civita connection when the metric is Kähler. Some notable cases include:
\begin{description}
  \item[\(t = 1\):] the \emph{Chern connection}, the unique Hermitian connection whose \((0,1)\)-part equals the Cauchy-Riemann operator \(\overline\partial_{TX}\) on the holomorphic tangent bundle;
  \item[\(t = -1\):] the \emph{Bismut connection} \cite{bismut}, the unique Hermitian connection with totally skew-symmetric torsion;
  \item[\(t = 0\):] the orthogonal projection of the Levi-Civita connection onto the affine space of Hermitian connections, also known as the \emph{first canonical connection};
  \item[\(t = \sfrac{1}{2}\):] the \emph{conformal connection}, in the sense of Libermann;
  \item[\(t = \sfrac{1}{3}\):] the Hermitian connection with minimal torsion.
\end{description}

Associated to any such Hermitian connection \(\nabla\), there is a \emph{curvature operator}
\[
R \in \Omega^2(X, \mathrm{End}(TX)), \qquad R(V,W) := [\nabla_V, \nabla_W] - \nabla_{[V,W]}.
\]
By means of the metric, one can also view it as a $4$-covariant tensor:
\[
R(V, W, Z, Y) := g(R(V, W)Z, Y) \in \Omega^2(X) \otimes \Omega^2(X).
\]

In the Kähler case, all Gauduchon connections coincide with the Levi-Civita connection, which also equals the Chern connection. Consequently, the curvature tensor satisfies additional symmetries: \(R \in S^2\Omega^2(X)\) and \(R \in \Omega^{1,1}(X, \mathrm{End}(T^{1,0}X))\). The extent to which these symmetries persist in the non-Kähler setting has been investigated by Gray \cite{gray}, Yang and Zheng \cite{yang-zheng-2}, and more recently in \cite{angella-otal-ugarte-villacampa}.

The failure of these symmetries in the general case leads to several notions of \emph{Ricci curvature}, obtained via distinct traces:
\begin{align*}
\mathrm{Ric}^{(1)}(V, W) &:= \mathrm{tr}\, R(V, W) \in \Omega^2(X), \\
\mathrm{Ric}^{(2)}(Z, Y) &= \mathrm{tr}_g R(\cdot, \cdot, Z, Y) \in \mathrm{End}(TX).
\end{align*}
Both yield the same \emph{scalar curvature}:
\[
s := \mathrm{tr}_g \mathrm{Ric}^{(1)} = \mathrm{tr}\, \mathrm{Ric}^{(2)} \in \mathcal{C}^\infty(X, \mathds R).
\]
A third trace can also be considered, though no compelling geometric interpretation has been identified.

In the case of the Chern connection, the first Ricci form \(\mathrm{Ric}^{(1)}\) is a closed real \((1,1)\)-form representing the first Bott–Chern class \(c_1^{BC}(X) \in H^{1,1}_{BC}(X, \mathds{R})\), which maps to the usual first Chern class of the anti-canonical bundle \(K_X^{-1}\).

\subsection{Special Hermitian metrics}

We continue with the same notation. Consider \( D\omega \), the covariant derivative of the \((1,1)\)-form \(\omega\) with respect to the Levi-Civita connection. This can be viewed as a tensor in \( \Omega^2(X, TX) \) via the identity $g((D\omega)(V,W),Z) = (D_V \omega)(W,Z)$.
For a detailed decomposition in terms of complex bi-degrees, see \cite[Proposition 1]{gauduchon-BUMI}.

Under the action of the unitary group \( U(n) \), the tensor \( D\omega \) decomposes into four irreducible components, one of which vanishes in the integrable (complex) case. The vanishing of specific subsets of these components gives rise to the celebrated Gray–Hervella classification \cite{gray-hervella} of almost-Hermitian structures into \( 2^4 = 16 \) classes. We now describe some of the Hermitian metric types most relevant to our discussion.

\begin{description}
\item[Kähler] A Hermitian metric is called \emph{Kähler} if \( D\omega = 0 \), which is equivalent to \( d\omega = 0 \). In this case, \(\omega\) defines a symplectic structure, and the complex structure is parallel, i.e., \( DJ = 0 \). Kähler metrics exhibit strong geometric, cohomological, and topological properties.

\item[balanced] Hermitian metrics satisfying \( d\omega^{n-1} = 0 \), where \( n \) is the complex dimension, or equivalently \( d^*\omega = 0 \), are called \emph{balanced metrics} \cite{michelsohn}. They are notable for being stable under modifications \cite{alessandrini-bassanelli}.

\item[SKT] Hermitian metrics for which the torsion $3$-form of the Bismut connection, given by \( d^c\omega \), is closed, {\itshape i.e.}, \( dd^c\omega = 0 \), are called \emph{strong K\"ahler with torsion}, or equivalently, \emph{pluriclosed} metrics.

\item[Gauduchon]  A Hermitian metric \(\omega\) is called standard or \emph{Gauduchon} \cite{gauduchon-cras} if it satisfies \( \partial \overline{\partial} \omega^{n-1} = 0 \). In complex dimension \(2\), this condition coincides with the pluriclosed condition. The significance of Gauduchon metrics lies in the fact that, on any compact complex manifold, every conformal class of Hermitian metrics contains a unique Gauduchon metric of volume one, see \cite[Th\'eor\`eme 1]{gauduchon-cras} (cf. Section \ref{subsec:gauduchon}).

\item[LCK] A Hermitian metric is called \emph{locally conformally K\"ahler} if it is locally conformal to a K\"ahler metric. By the Poincaré lemma, the conformal factors can be encoded in a global closed $1$-form \(\theta\), called the Lee form, satisfying
\[
d\omega = \theta \wedge \omega, \quad \text{with } d\theta = 0.
\]
LCK geometry can be understood as an equivariant (homothetic) version of Kähler geometry \cite{gini-ornea-parton-piccinni}. If \(\omega\) is an LCK metric on \(X\) with Lee form \(\theta\), then there exists a covering \(\tilde{X}\) endowed with a Kähler metric \(\tilde{\omega}\), such that the pullback of \(\theta\) is exact, \(\tilde{\omega}\) is conformal to the pullback of \(\omega\), and the deck transformation group acts on \(\tilde{\omega}\) by homotheties.
Note that most compact complex surfaces, with the only exception of certain Inoue surfaces, admit either K\"ahler or LCK metrics, see \cite{belgun}.
For a recent survey on LCK geometry, see \cite{ornea-verbitsky}.
\end{description}

We remark that, for compact holomorphic surfaces, the existence of a Kähler metric is a topological property: it depends solely on whether the first Betti number is even. This follows from the classification results of Enriques, Kodaira, and Siu, and was later independently and intrinsically proven by Lamari \cite{lamari} and Buchdahl \cite{buchdahl}.

\subsection{Existence of Gauduchon metrics}\label{subsec:gauduchon}

In general, special Hermitian metrics do not exist on compact holomorphic manifolds, except for Gauduchon metrics, as proven in \cite[Th\'eor\`eme 1]{gauduchon-cras}. Let $X$ be an $n$-dimensional compact holomorphic manifold and $\omega$ an Hermitian metric on $X$. If $\theta$ denotes the Lee form of the metric, namely the unique $1$-form defined by $d\omega^{n-1}=\theta\wedge\omega^{n-1}$, then $\omega$ is Gauduchon if and only if $d^*\theta=0$.
We are going to recall the proof of the following

\begin{theorem}[{\cite{gauduchon-cras}}]\label{thm:gauduchon}
Let $(X,J)$ be a compact holomorphic manifold of dimension $n>1$. Then, there exists and it is unique, up to multiplication with positive constant, a Gauduchon metric in the conformal class of any Hermitian metric.
\end{theorem}

\begin{proof}
Fix an arbitrary Hermitian metric $\omega$ on $X$. We wish to prove that there exists and it is unique (up to constant) a smooth function $\psi:X\to \mathds{R}^+$ such that $\psi\omega$ is Gauduchon, {\itshape i.e.} $\partial\overline{\partial}(\psi^{n-1}\omega^{n-1})=0$. First of all, consider the \emph{Chern Laplacian} acting on smooth functions
$$\Delta^{\text{Ch}}f:=\Delta_{\omega}f+g(df,\theta_{\omega})=2\,\text{tr}_{\omega}(\sqrt{-1}\partial\overline{\partial}f)
$$
and its formal adjoint 
$$
(\Delta^{\text{Ch}})^*f=\Delta_{\omega}f-g(df,\theta_{\omega})+
d^{*_{\omega}}\theta_{\omega}f\,.
$$
In particular, notice that $d^{*_{\omega}}\theta_{\omega}=(\Delta^{\text{Ch}})^*(1)$. Consider now the conformal change $\tilde\omega=\Phi^{\frac{2}{n-1}}\omega$ and denote by $\Delta^{\text{Ch}}_\Phi$ the associated Chern Laplacian. Then
$$
\Delta^{\text{Ch}}_\Phi=\Phi^{-\frac{2}{n-1}}\Delta^{\text{Ch}}
\quad\text{and}\quad
(\Delta^{\text{Ch}}_{\Phi})^{\tilde *}=
\Phi^{-\frac{2n}{n-1}} (\Delta^{\text{Ch}})^*\circ \Phi^2\,.
$$
Therefore, $\tilde\omega$ is Gauduchon if and only if $d^{*_{\tilde\omega}}\tilde\theta_{\tilde\omega}=0$ if and only if $(\Delta^{\text{Ch}}_{\Phi})^{\tilde *}(1)=0$ if and only if $(\Delta^{\text{Ch}})^*(\Phi^2)=0$ if and only if $\Phi^2\in\ker\,(\Delta^{\text{Ch}})^*$. One needs to show then the existence and uniqueness of a positive function in the kernel of the linear elliptic operator $(\Delta^{\text{Ch}})^*$. Indeed, if $f\in\ker\,(\Delta^{\text{Ch}})^*$ is positive, then by the previous argument $\tilde\omega:=f^{\frac{1}{n-1}}\omega$ is a Gauduchon metric.

Notice that the Chern Laplacian $\Delta^{\text{Ch}}$ is a real, elliptic operator of second order with smooth coefficients and without zero order terms, {\itshape i.e.} $\Delta^{\text{Ch}}(1)=0$. Clearly, $\mathds{C}\subseteq \ker(\Delta^{\text{Ch}})$. Since $X$
is compact and connected, if $f\in \ker(\Delta^{\text{Ch}})$, by Hopf maximum principle, $f$ is constant. Hence, $\ker(\Delta^{\text{Ch}})=\mathds{C}$. Notice that $(\Delta^{\text{Ch}})^*$ has zero order terms, however $(\Delta^{\text{Ch}})^*$ is elliptic and both the index of $\Delta^{\text{Ch}}$ and the index of $(\Delta^{\text{Ch}})^*$ are equal to the index of the Laplacian $\Delta_\omega$ hence equal to $0$. Therefore, $\dim\ker(\Delta^{\text{Ch}})^*=\dim\ker(\Delta^{\text{Ch}})=1$. We need now the following
\begin{lemma}\label{lem:0}
  Let $f:X\to\mathds{R}$ be a smooth function. If $f\in\ker(\Delta^{\text{Ch}})^*$ and $f \not\equiv 0$, then $(f,1)_\omega=\int_X f\,dV_{\omega}\neq 0$.   
\end{lemma}
 \begin{proof}[Proof of \Cref{lem:0}]
     By contradiction, assume that $(f,1)_\omega=\int_X f\,dV_{\omega}= 0$. Since
     $\dim\ker(\Delta^{\text{Ch}})^*=1$, by hypothesis $f$ is a generator of $\ker(\Delta^{\text{Ch}})^*$. Therefore, the constant function $1$ is orthogonal to $\ker(\Delta^{\text{Ch}})^*$. By the orthogonal decomposition $$
     \mathcal{C}^\infty(X,\mathds{C})=
     \ker(\Delta^{\text{Ch}})^*\oplus\text{Im}(\Delta^{\text{Ch}})
     $$ one has that $1\in\text{Im}(\Delta^{\text{Ch}})$, namely there exists a function $h:X\to\mathds{C}$ such that $\Delta^{\text{Ch}}(h)=1>0$. By the maximum principle, it follows that $h$ is constant and so $\Delta^{\text{Ch}}(h)=0\neq 1$ giving a contradiction.
 \end{proof}

\begin{lemma}\label{lem:1}
 Let $f:X\to\mathds{R}$ be a smooth function. If $f$ is a generator of    $\ker(\Delta^{\text{Ch}})^*$ then $f\geq 0$ or $f\leq 0$ on $X$.
\end{lemma}
 \begin{proof}[Proof of \Cref{lem:1}]
Since $f$ is a generator of    $\ker(\Delta^{\text{Ch}})^*$ then $f\not\equiv 0$ and so by the previous result $(f,1)_\omega=\int_X f\,dV_{\omega}\neq 0$. Suppose that $(f,1)_\omega>0$. We will show that $f\geq 0$ on $X$. By contradiction assume that there exists a point $x_0\in X$ such that $f(x_0)<0$. By continuity of $f$, there exists an open neighborhood $U$ of $x_0$ such that $f<0$ on $U$. Consider the open set, containing $U$, given by
$$
V:=\left\lbrace x\in X \,:\,f(x)<0\right\rbrace\,.
$$
One can then construct a smooth function $\varphi:X\to\mathds{R}^+$ such that
$$
\int_V\varphi f\,\text{dV}_\omega+
\int_{X\setminus V}\varphi f\,\text{dV}_\omega=
\int_X\varphi f\,\text{dV}_\omega=0
$$
namely $(\varphi f,1)_\omega=0$. Therefore, with respect to the conformal metric $\varphi\omega$, $(\varphi^{1-n} f,1)_{\varphi\omega}=0$. Moreover, since $f\in\ker(\Delta^{\text{Ch}})^*$, then $\varphi^{1-n}f\in\ker(\Delta^{\text{Ch}})_{\varphi\omega}^*$ and $\varphi^{1-n}f\not\equiv 0$ but this contradicts Lemma \ref{lem:0}.
\end{proof}

Finally we recall the following
\begin{lemma}\label{lem:2}
Let $X$ be a connected smooth manifold and let $Q:\mathcal{C}^\infty(X,\mathds{R})\to \mathcal{C}^\infty(X,\mathds{R})$ be an elliptic differential operator of second order with real smooth coefficients. If $f\in \mathcal{C}^\infty(X,\mathds{R})$ is such that $Q(f)=0$ and $f\geq 0$ on $X$, then $f>0$ or $f\equiv 0$ on $X$.
\end{lemma}
We refer to {\itshape e.g.} \cite{popovici-book} for the proof.
As a consequence of the previous results one can prove the following
\begin{corollary}\label{cor:3}
    Let $f:X\to\mathds{R}$ be a smooth function such that $f\in\ker(\Delta^{\text{Ch}})^*$. Then $f> 0$ or $f< 0$ or $f\equiv 0$ on $X$.
\end{corollary}
\begin{proof}
If $f\not\equiv 0$ then $f$ generates $\ker(\Delta^{\text{Ch}})^*$ because it is $1$-dimensional. Then by Lemma \ref{lem:1}, $f\geq 0$ or $f\leq 0$ on $X$. Suppose that $f\geq 0$, then applying Lemma \ref{lem:2} with $Q=(\Delta^{\text{Ch}})^*$ one has $f>0$.
\end{proof}
We can finally conclude the proof of Theorem \ref{thm:gauduchon}. One needs to find a unique (up to constant) smooth positive function $f$ on $X$ such that $f\in\ker(\Delta^{\text{Ch}})^*$. Since $\ker(\Delta^{\text{Ch}})^*$ is $1$-dimensional the uniqueness up to constants follows immediately.
Now let $f\in \ker(\Delta^{\text{Ch}})^*$ be a generator, hence $f\not\equiv 0$ and so by Corollary \ref{cor:3}, 
$f> 0$ or $f< 0$ on $X$. Assuming $f> 0$ we are done indeed $\tilde\omega:=f^{\frac{1}{n-1}}\omega$ is a Gauduchon metric.
\end{proof}

\subsection{Classification of LCK threefolds}

For a compact holomorphic manifold \( X \) of complex dimension \( n \), the algebra of meromorphic functions \( \mathcal{M}_X \) is finitely generated, and its transcendence degree takes values in \( \{0, \dots, n\} \). This invariant, first established by Siegel, is known as the \emph{algebraic dimension}:
\[
a(X) := \mathrm{tr}\deg_{\mathds{C}} H^0(X, \mathcal{M}_X).
\]
Its geometric meaning becomes clear after the Algebraic Reduction Theorem, see {\itshape e.g.} \cite{ueno-LNM}. If a compact complex manifold \( X \) has algebraic dimension \( a(X) = k \), then there exists a (unique up to bimeromorphic equivalence) diagram
\[
\xymatrix{
X \ar@{<-->}[r] & X^* \ar[d]^f \\
& V
}
\]
where \( X^* \) is bimeromorphically equivalent to \( X \) of the same dimension, \( V \) is a projective manifold of dimension \( k \), and \( f \) is a holomorphic surjective map inducing an isomorphism between the fields of meromorphic functions of \( X^* \) and \( V \).
We recall that the \emph{Weak Factorization Theorem}, which also holds in the analytic category \cite{abramovich-karu-matsuki-wlodarczyk}, states that any bimeromorphic map can be factored as a finite sequence of blow-ups and blow-downs along smooth centers. It is conjectured that a \emph{Strong Factorization} also holds \cite{hironaka-AnnMath-1, hironaka-AnnMath-2, abramovich-karu-matsuki-wlodarczyk}, that is, any bimeromorphism can be expressed as
\[
\xymatrix{
& \hat X \ar[ld]_{\sigma} \ar[rd]^c & \\
X \ar@{<-->}[rr] && X^*
}
\]
where \(\sigma\) and \(c\) are compositions of blow-ups with smooth centers.

Let us survey the classification of compact holomorphic surfaces \( X \) based on their algebraic dimension.
\begin{description}
\item[$a(X)=2$:] By a theorem of Chow and Kodaira \cite{chow-kodaira}, a compact holomorphic surface has two algebraically independent meromorphic functions, {\itshape i.e.}, algebraic dimension \( a(X) = 2 \), if and only if it is algebraic.

\item[$a(X) = 1$:] If \( X \) is non-algebraic and \( a(X) = 1 \), then \( X \) is an \emph{elliptic surface}: there exists a holomorphic surjective map \( f\colon X \to C \) onto a non-singular complex curve \( C \), whose general fibre is a smooth elliptic curve. The classification of the possible singular fibres is due to Kodaira.
If moreover \( X \) is non-Kähler (note that it still admits a LCK structure), that is, \( b_1(X) \) is odd, then all smooth fibres are isomorphic, and the only singular fibres are either multiple fibres of the generic one or, if the surface is not minimal, trees of rational curves. In this case, the surface is called a \emph{quasi-bundle}; see \cite{brinzanescu, brinzanescu-book}.

\item[$a(X) = 0$:] If \( X \) has algebraic dimension zero, then it admits no non-constant meromorphic functions. In this case, \( X \) is necessarily non-algebraic. Surfaces of algebraic dimension zero include certain tori, certain K3 surfaces, and certain class VII surfaces.
\end{description}

In \cite{angella-parton-vuletescu}, we addressed the classification problem for threefolds under the assumption that \( X \) admits an LCK metric.
The case of maximal algebraic dimension, \( a(X) = 3 \), is well understood thanks to \cite{moishezon, deligne-griffiths-morgan-sullivan, kokarev-kotschick}: indeed, in this situation \( X \) is Mo\v{\i}\v{s}hezon, hence it satisfies the \(\partial \overline{\partial}\)-Lemma property, and is therefore Kähler. We then turn to the case in which the algebraic dimension is one less than maximal, namely \( a(X) = n - 1 = 2 \). The remaining case \( a(X) = 1 \) will be considered in \cite{angella-petcu-vuletescu}.

Under the Algebraic Reduction theorem and the Strong Factorization Conjecture, we consider the following diagram:
\[
\xymatrix{
& \hat{X} \ar[ld]_\sigma \ar[rd]^c & \\
X \ar@{<-->}[rr] && X^* \ar[d]^f \\
&& B
}
\]
where \( B \) is a projective surface, and \( f: X^* \to B \) is a surjective, proper, holomorphic map with connected fibres.  
It is known that the general fibre of \( f \) is a smooth elliptic curve \cite{kawai}.  
We further assume that \( f \) is flat (to avoid unexpected behavior in fibre degeneration) and that its singular locus \( S(f) \) is a simple normal crossing divisor.

One of the key tools in the argument is the following result, stated in the more general context of a \emph{weak LCK} (WLCK) structure. This means that $\omega \geq 0$ is assumed to be positive only outside a proper analytic subset \( N_\omega \), which may be non-empty. Such structures are necessary since, in general, metric properties do not behave well under blowups and blowdowns. On the other hand, if \(X\) is WLCK, then so is \(\hat{X}\), see \cite[Proposition 2.2]{angella-parton-vuletescu}.

\begin{theorem}[Lemma on fibrations, {\cite[Lemma 2.4]{angella-parton-vuletescu}}]
Let \( X \) and \( B \) be holomorphic manifolds with \(\dim X > \dim B\), and let \( f \colon X \to B \) be a surjective, proper holomorphic map with connected fibres. Suppose \(\omega\) is a WLCK structure on \(X\) with Lee form \(\theta\). If the cohomology class of the Lee form satisfies \( [\theta] = f^*[\alpha], \) for some class \([\alpha] \in H^1(B, \mathds{R})\), then \([\theta]=0\).
\end{theorem}

\begin{proof}
Up to a global conformal change, we may assume \(\theta = f^* \alpha\). Denote by \(S_f\) the union of the non-flat and singular loci, that is, the set of points \(b \in B\) for which the fibre \(f^{-1}(b)\) is either singular or has dimension strictly greater than \(k := \dim X - \dim B > 0\). It is known that \(S_f\) is an analytic subspace. Moreover, the restriction $f \colon X \setminus f^{-1}(S_f) \to B \setminus S_f$ is a submersion. 
We also consider the open subset \(B' \subset B \setminus S_f\) where, for any \(b \in B'\),
\[
v(b) := (f_* \omega^k)(b) = \int_{f^{-1}(b)} \omega^k\vert_{f^{-1}(b)} = \mathrm{vol}_\omega(f^{-1}(b)) > 0.
\]
By the ``weak'' assumption on the structure, the complement \( B \setminus B' \) is a proper analytic subset of \( B \).

By the WLCK condition \( d\omega = f^*\alpha \wedge \omega \) and using the projection formula, we compute
\[
d(f_* \omega^k) = f_*(d\omega^k) = k f_*(f^*\alpha \wedge \omega^k) = k \alpha \wedge f_* \omega^k,
\]
which we rewrite as
\[
d v = k \alpha \wedge v.
\]
This implies that on \( B' \),
\[
\alpha_{\vert B'} = \frac{1}{k} \frac{d v}{v} = \frac{1}{k} d \log v,
\]
and therefore the cohomology class satisfies \( [\alpha] \vert_{B'} = 0 \).

The map \( H^1(B, \mathds{R}) \to H^1(B', \mathds{R}) \) induced by the inclusion \( B' \hookrightarrow B \) is injective. This follows, for instance, from the long exact sequence in cohomology for the pair \((B, B')\), together with the fact that $H^1(B, B', \mathds{R}) = 0$, which holds by excision, homotopy invariance, and the Thom isomorphism, and to the real codimension of \( B \setminus B' \) in \( B \) being at least $2$, see, for example, \cite[Theorem II.5.3]{suwa}. Therefore, we conclude that \([ \alpha ] = 0\) in \( H^1(B, \mathds{R}) \), proving the statement.
\end{proof}

The second tool is the following generalization of the Vaisman lemma to the WLCK context.

\begin{theorem}[{Vaisman Lemma, \cite[Lemma 2.5]{angella-parton-vuletescu}, after \cite{vaisman}}]
Let \(X\) be a compact holomorphic manifold of dimension \(\dim X = n > 1\) endowed with a WLCK structure \(\omega\) with Lee form \(\theta\). If \(X\) admits Kähler metrics, then \([\theta] = 0\).
\end{theorem}

\begin{proof}
Fix a Kähler metric \( g_0 \) on \( X \), and let \( \theta_0 \) be the harmonic representative of \( [\theta] \) with respect to \( g_0 \). Up to a global conformal change, we may assume the WLCK structure \( \omega_0 \) has Lee form \( \theta_0 \).
Decompose \( \theta_0 = \alpha + \overline{\alpha} \) into its pure-type components, where \( \alpha \in \wedge^{1,0} X \). Multiplying the equation $d\omega_0^{n-1} = (n-1) \theta_0 \wedge \omega_0^{n-1}$ by \( \alpha \), we obtain
\[
\alpha \wedge d\omega_0^{n-1} = (n-1) \alpha \wedge \overline{\alpha} \wedge \omega_0^{n-1}.
\]
By the strong Hodge decomposition, \( \alpha \) is itself harmonic, and in particular closed. Therefore,
\[
0 = \int_X \alpha \wedge d\omega_0^{n-1} = (n-1) \|\alpha\|^2_{\omega_0\vert_{X \setminus N_\omega}}.
\]
This implies \( \alpha = 0 \), and hence \( [\theta] = 0 \), concluding the proof.
\end{proof}

We are now ready to state the main result. In the following, we assume:
\begin{itemize}
\item the Strong Factorization Conjecture;
\item the algebraic reduction is flat and its singular locus is a simple normal crossing divisor (note that applying Hironaka flattening may introduce heavy singularities);
\item the canonical bundle \( K_{X^*} \) is nef, namely $K_{X^*}\cdot C\geq 0$ for any curve $C$ in $X^*$; we recall that the existence of minimal models for compact elliptic threefolds that are not uniruled  was established by Grassi \cite{grassi} in the projective case, and more generally by H\"oring and Peternell  \cite{horing-peternell} for K\"ahler threefolds.
\end{itemize}

\begin{theorem}[{\cite{angella-parton-vuletescu}}]
Let \( X \) be a compact holomorphic threefold with algebraic dimension \( a(X) = 2 \), endowed with an LCK metric whose Lee class is non-zero.
Under the technical assumptions stated above, \( X \) is a blown-up elliptic quasi-bundle over a projective surface.
\end{theorem}

\begin{proof}[Sketch of the proof]
We first study the general fibres and prove that they are all isomorphic. By the continuity of the \( j \)-invariant, it follows that all smooth fibres are isomorphic; denote this common elliptic curve by \( \mathcal{E} \).

Take a general smooth divisor \( H \subset B \). By applying twice the Bertini theorem, we can assume that both \( H \) and  \(S_H := (f \circ c)^{-1}(H)\) are smooth:
\[
\xymatrix{
& \hat{X} \ar[ld]_\sigma \ar[rd]^c & & S_H \ar@{_{(}->}[ll] \ar[dd]^{f \circ c} \\
X \ar@{<-->}[rr] && X^* \ar[d]^f & \\
&& B & H \ar@{_{(}->}[l]
}
\]
We can also assume that \( S_H \) is not contained in the exceptional divisor of \( c \), which assures that \( \sigma^*\omega \) induces a WLCK structure on \( S_H \).

If \( S_H \) were Kähler, then \( [\sigma^* \theta] \vert_{S_H} = 0 \) by the Vaisman Lemma. By standard cohomological arguments, it follows that there exists an analytic subset \( \mathcal{Z} \subseteq B \) such that
\[
[\sigma^* \theta] \vert_{\hat{X} \setminus (f \circ c)^{-1}(\mathcal{Z})} = (f \circ c)^*[\alpha]
\]
for some \( [\alpha] \in H^1(B \setminus \mathcal{Z}) \). By the Lemma on fibrations, it follows that \( [\sigma^* \theta] \vert_{\hat{X} \setminus (f \circ c)^{-1}(\mathcal{Z})} = 0 \), and therefore \( [\theta] = 0 \), which contradicts the assumption that the Lee class is non-trivial.  
Therefore, \( S_H \) is not Kähler. We can now apply the Br\^{i}nzanescu theorem on non-K\"ahler surfaces, see {\itshape e.g.} \cite[Proposition 3.17]{brinzanescu-book}, to conclude that the general fibres of \( S_H \to H \) are isomorphic. Therefore, the general fibres of \( f \) are isomorphic.

We now consider singular fibres. 
First, by applying Bertini's theorem and Br\^inz\v{a}nescu's result for surfaces, we prove that the general singular fibres over \( p \in S(f) \) are just multiple fibres of the form \( m\mathcal{E} \), with no trees of rational curves appearing due to the minimality assumption on \( X^* \).

Under the assumption that \( S(f) \) is a simple normal crossing divisor, the singular fibres that are not of the form \( m\mathcal{E} \) occur only at special points: either smooth points of a single branch, or nodes at the intersection of branches of different types, say \( m\mathcal{E} \) and \( q\mathcal{E} \). Up to a suitable Galois covering, these singular fibres form a finite set.

Finally, assuming \( f \) flat, we observe that singular fibres, endowed with the reduced structure, cannot be isolated, see \cite[Lemma 2.9]{angella-parton-vuletescu}. Naively, the reason is that a fibre with singular reduction would induce a non-trivial local monodromy around it, but this is impossible, since the punctured \( 2 \)-dimensional ball is simply connected.

The final step in the study of the geometric structure is to show that the map \(\sigma\) can be taken to be the identity and that the blow-ups in \(c\) are ``special'', in the sense that each blow-up is either at a point or along a curve contained in an exceptional divisor produced in an earlier step, see \cite[Theorem 2.12]{angella-parton-vuletescu}.
\end{proof}

\ifdefined\frommain
\else
    \bibliographystyle{alpha}
    \bibliography{biblio}
\fi

\subsection{Problems}

\begin{exercise}
Prove that, on a Hermitian manifold $(X,J,g,\omega)$, the Levi-Civita connection $D$ and the Chern connection $\nabla$ coincide if and only if the metric is K\"ahler.
\end{exercise}

\begin{proof}[Hint]
Prove the identities $3d\omega(V,W,Z) = g((D_VJ)W,Z) + g((D_WJ)Z,V) + g((D_ZJ)V,W)$ and $2g((D_VJ)W,Z) = 3d\omega(V,W,Z) - 3d\omega(V,JW,JZ)$.
\end{proof}

\begin{exercise}[Hodge decomposition]
Let \( X \) be a compact Kähler manifold. Prove that there exists a decomposition
\[
H^k_{\mathrm{dR}}(X, \mathds{C}) \simeq \bigoplus_{p+q=k} H^{p,q}_{\overline{\partial}}(X),
\quad
\text{ with }
\quad
H^{p,q}_{\overline{\partial}}(X) \simeq \overline{H^{q,p}_{\overline{\partial}}(X)}.
\]
\end{exercise}

\begin{proof}[Hint]
Let \( L \colon \Omega^k(X) \to \Omega^{k+2}(X) \) denote the Lefschetz operator defined by \( L = \omega \wedge \_ \), and let \( \Lambda \) its adjoint. Together with the counting operator \( H \), defined on \(\Omega^k(X)\) by \( H|_{\Omega^k(X)} = (n-k) \cdot \mathrm{id}\), where \( n = \dim_{\mathds{C}} X \), these operators give rise to an \(\mathfrak{sl}(2,\mathds{R})\)-representation on the space of differential forms.

Using this structure, one can verify the following Kähler identities:
\begin{align*}
[\overline{\partial}, L] &= 0, & \quad [\partial, L] &= 0, \\
[\overline{\partial}^*, L] &= \sqrt{-1}\, \partial, & \quad [\partial^*, L] &= -\sqrt{-1}\, \overline{\partial}, \\
[\overline{\partial}, \Lambda] &= \sqrt{-1}\, \partial^*, & \quad [\partial, \Lambda] &= -\sqrt{-1}\, \overline{\partial}^*, \\
[\overline{\partial}^*, \Lambda] &= 0, & \quad [\partial^*, \Lambda] &= 0,
\end{align*}
on each indecomposable component of the Lefschetz decomposition.

Another way to prove the Kähler identities is to notice that every Kähler metric osculates to order $2$ the standard Hermitian metric on \(\mathds{C}^n\). Therefore, one can reduce the proof to the corresponding identities on \(\mathds{C}^n\) equipped with the standard Kähler structure, which are known as the Akizuki–Nakano identities, see \cite{demailly-agbook}.

Finally, these identities imply that the Laplacians coincide up to a constant factor:
\[
\Delta = 2 \Box = 2 \overline{\Box}.
\]
The result then follows directly from Hodge theory.
\end{proof}

\begin{proof}[Remark]
Note that the Hodge decomposition above appears to depend on the choice of metric, since it is stated in terms of harmonic forms. However, for compact Kähler manifolds, the decomposition actually descends to cohomology independently of the metric, thanks to the \(\partial\overline{\partial}\)-Lemma property, which ensures a canonical pure-type decomposition in cohomology.
\end{proof}

\begin{exercise}[{\cite{weil, deligne-griffiths-morgan-sullivan}}]
Prove that compact K\"ahler manifolds satisfy the $\partial\overline\partial$-Lemma property.
\end{exercise}

\begin{proof}
Let \(\alpha \in \Omega^{p,q}(X)\) be a \(\partial\)-closed, \(\overline{\partial}\)-closed, and \(d\)-exact form on \(X\). In particular, by the Hodge theory for the de Rham cohomology, the form \(\alpha\) is orthogonal to the space of \(\Delta\)-harmonic forms. 
Note that, by the Kähler identities, the space of \(\Delta\)-harmonic forms coincides with the space of \(\Box\)-harmonic forms and with the space of \(\overline{\Box}\)-harmonic forms. Since \(\alpha\) is \(\partial\)-closed and orthogonal to the space of \(\Box\)-harmonic forms, thanks again to the Hodge theory for the conjugate Dolbeault cohomology, we have that there exists \(\gamma \in \Omega^{p-1,q}(X)\) such that $\alpha = \partial \gamma$.
Applying the Hodge theory for the Dolbeault cohomology to \(\gamma\), we obtain a \(\overline{\Box}\)-harmonic form \(h_\gamma\), a form \(\beta \in \Omega^{p-1,q-1}(X)\), and a form \(\eta \in \Omega^{p-1,q+1}(X)\) such that \( \gamma = h_\gamma + \overline{\partial}\beta + \overline{\partial}^*\eta \).
By the Kähler identities, we have \([\partial, \overline{\partial}^*] = 0\), and \(h_\gamma\) is also \(\Box\)-harmonic, in particular \( \partial \)-closed. Hence, \( \alpha = \partial \gamma = \partial h_\gamma + \partial \overline{\partial} \beta + \partial \overline{\partial}^* \eta = \partial \overline{\partial} \beta + \overline{\partial}^* \partial \eta \).
It suffices to prove that \(\overline{\partial}^* \partial \eta = 0\). Indeed, since \(\alpha\) is \(\overline{\partial}\)-closed, we get \( \overline{\partial} \overline{\partial}^* \partial \eta = 0 \).
Therefore,
\[
\|\overline{\partial}^* \partial \eta\|^2 = \langle \overline{\partial}^* \partial \eta, \overline{\partial}^* \partial \eta \rangle = \langle \overline{\partial} \overline{\partial}^* \partial \eta, \partial \eta \rangle = 0,
\]
which implies \(\overline{\partial}^* \partial \eta = 0\). 
Hence, \(\alpha = \partial \overline{\partial} \beta\) is \(\partial \overline{\partial}\)-exact, as claimed.
\end{proof}

\begin{exercise}
Prove that, on a compact K\"ahler manifold, the even Betti numbers are positive and the odd Betti numbers are even.
\end{exercise}

\begin{proof}[Hint]
For the even Betti numbers, let \(\omega\) denote the symplectic form associated to the K\"ahler metric. Note that the cohomology class \([\omega]\) is non-trivial, as are its non-trivial powers \([\omega^k]\).
For the odd Betti numbers, use the Hodge decomposition.
\end{proof}

\begin{exercise}
Prove that the Hopf manifold does not admit any Kähler metric. More generally, show that the differentiable manifold \(S^1 \times S^{2n-1}\) does not admit any complex structure compatible with a Kähler metric.
\end{exercise}

\begin{proof}[Hint]
Use the K\"unneth formula to compute $b_2(S^1 \times S^{2n-1})=0$.
\end{proof}

\begin{exercise}
Let $X$ be a compact holomorphic manifold of complex dimension $n > 2$, endowed with a Hermitian metric $\omega$.
\begin{itemize}
\item Prove that there always exists a $1$-form $\theta$ such that
\[ d\omega^{n-1} = \theta \wedge \omega^{n-1}, \]
called the (balanced) \emph{Lee form}, or torsion $1$-form. It corresponds to the trace of the torsion of the Chern connection.
\item Prove that $\omega$ is Gauduchon if and only if $d^*\theta = 0$.
\item Show that the Chern Laplacian \( \Delta^{\mathrm{Ch}} f := 2 \mathrm{tr}_\omega (\sqrt{-1} \partial \overline{\partial} f) \) and the Hodge–de Rham Laplacian $\Delta$ are related by
\[ \Delta^{\mathrm{Ch}} f = \Delta f + g(df, \theta). \]
\end{itemize}
\end{exercise}

\begin{exercise}[{\cite{alexandrov-ivanov}}]
Prove that, on a compact holomorphic manifold, a Hermitian metric that is both pluriclosed and balanced, must be Kähler.
\end{exercise}

\begin{proof}[Hint]
Following \cite{popovici-book}, show that, if $\omega$ is a pluriclosed metric, then $*\partial\omega\in\ker\partial^*$ and that, if $\omega$ is balanced, then $\omega^{n-2}\wedge\partial\omega=0$, namely $\partial\omega$ is a primitive form. In particular, if $\partial\omega$ is primitive, then $*\partial\omega\in\text{Im}\,\partial$. Since $\ker\partial^*\cap\text{Im}\,\partial=\left\lbrace 0\right\rbrace$, one concludes that $\partial\omega=0$, that is $\omega$ is K\"ahler.
\end{proof}

\begin{exercise}
Prove that the Iwasawa manifold \( X = \sfrac{{\rm Heis}(3, \mathds{C})}{\mathds{Z}[\sqrt{-1}]} \) admits a balanced Hermitian metric. Show that it does not admit any pluriclosed metrics.
\end{exercise}

\begin{proof}[Hint]
First, prove that \( X \) does not admit any \emph{invariant} pluriclosed metrics, meaning a metric induced by a Hermitian metric on \( \mathrm{Heis}(3, \mathds{C}) \) that is invariant under left-translations. (See also \cite{fino-parton-salamon} for a classification of six-dimensional nilmanifolds admitting pluriclosed metrics.) Then, apply the Belgun symmetrization trick \cite[Theorem 7]{belgun} to conclude that there are no pluriclosed metrics at all on \( X \).
\end{proof}

\begin{proof}[Remark]
The previous result is not specific for the Iwasawa manifold. In fact, more in general, it was shown in \cite{fino-vezzoni-PAMS} that a (2-step) nilmanifold endowed with an invariant complex structure admitting both a pluriclosed metric and a balanced metric is then a torus. Notice that the nilpotency step assumption is not necessary, due to \cite{arroyo-nicolini}. Alternatively, by \cite{egidi} one can show that the Iwasawa manifold does not admit any pluriclosed metric by showing that there exists a non-zero current $T$ of bi-degree $(n-1,n-1)=(2,2)$ such that $T\geq 0$ and $T$ is $\partial\overline{\partial}$-exact. More precisely, the $(2,2)$-form $\partial\overline{\partial}(\varphi^{3}\wedge\bar\varphi^{3}$ gives rise to such a current.
\end{proof}

\begin{exercise}[{\cite{guan-li, chiose-Toulouse}}]
Let \( X \) be a compact holomorphic manifold of complex dimension \( n \), endowed with a Hermitian metric \( \omega \). Prove that the following are equivalent.
\begin{enumerate}
\item The total volume is invariant under \( \sqrt{-1}\partial\overline\partial \)-deformations, namely, for any smooth real function \( \varphi \in \mathcal{C}^\infty(X, \mathds{R}) \) such that \( \omega + \sqrt{-1}\partial\overline\partial \varphi > 0 \), it holds
\[
\int_X \left( \omega + \sqrt{-1}\partial\overline\partial \varphi \right)^n = \int_X \omega^n.
\]
\item The metric \( \omega \) satisfies the so-called \emph{Guan-Li condition}:
\[
\sqrt{-1}\,\partial\overline\partial\omega = 0 \quad \text{and} \quad \sqrt{-1}\,\partial\omega\wedge\overline\partial\omega = 0.
\]
\item The metric \( \omega \) satisfies
\[
\sqrt{-1}\,\partial\overline\partial\omega^k = 0 \quad \text{for all } 1 \leq k \leq n-1.
\]
\end{enumerate}
\end{exercise}

\begin{exercise}
Compute the Chern curvature of the Hopf surface \( X = \sfrac{\mathds{C}^2 \setminus \{0\}}{\mathds{Z}} \), where \( \mathds{Z} \) acts via \( (z_1, z_2) \mapsto \left( \frac{1}{2} z_1, \frac{1}{2} z_2 \right) \), endowed with the Hermitian metric
\[
\omega = \frac{\sqrt{-1} \, dz_1 \wedge d\bar{z}_1 + \sqrt{-1} \, dz_2 \wedge d\bar{z}_2}{|z_1|^2 + |z_2|^2} .
\]
In particular, show that the second Chern-Ricci form satisfies
\[
\mathrm{Ric}^{(2)}(\omega) = 2\omega.
\]
\end{exercise}

\begin{proof}[Hint]
For explicit computations, see {\itshape e.g.} \cite{angella-sferruzza}.
\end{proof}

\ifdefined\frommain
\else
    \bibliographystyle{alpha}
    \bibliography{biblio}
\fi

\cleardoublepage

\section{Hermitian metrics of constant Chern scalar curvature}\label{sec:chern-yamabe}

In the following section, we consider a version of the Yamabe problem in the Hermitian setting: we look for Hermitian metrics in a fixed conformal class with constant scalar curvature with respect to the Chern connection, following \cite{angella-calamai-spotti}. Motivated by the role of scalar curvature as a moment map in the Donaldson-Fujiki picture for Kähler geometry, we also propose a new notion of canonical metric in the LCK setting, see \cite{angella-calamai-pediconi-spotti, angella-pediconi-scarpa-spotti-windare}.

\subsection{Curvature of the Chern connection}

Let \( X \) be a compact holomorphic manifold of complex dimension \( n \). Given a fixed Hermitian structure \( h \), we denote by $\nabla^{Ch}$ the Chern connection and with $R^{Ch}$ the Chern curvature tensor defined by $R^{Ch}(x,y)=\nabla^{Ch}_x\nabla^{Ch}_y-\nabla^{Ch}_y\nabla^{Ch}_x-\nabla^{Ch}_{[x,y]}$. Then, $R^{Ch}\in\wedge^2(X,\text{End}(TX))\simeq\wedge^2(X,TX\otimes \overline{TX})$. In local holomorphic coordinates $\left\lbrace z^j\right\rbrace$, denoting with $h_{i\bar j}$ the Hermitian matrix representing $h$, the curvature has the following expression
$$
R^{Ch}=R^{Ch}_{i\bar jk\bar l}\sqrt{-1}dz^i\wedge d\bar z^j\otimes\sqrt{-1}dz^k\wedge d\bar z^l
$$
where
$$
R^{Ch}_{i\bar jk\bar l}=-\frac{\partial^2h_{k\bar l}}{\partial z^i\partial \bar z^j}+h^{p\bar q}\frac{\partial h_{k\bar q}}{\partial z^i}\frac{\partial h_{p\bar l}}{\partial\bar z^j}\,.
$$
Since the curvature tensor does not have the Bianchi symmetry, as already recalled in the previous section, there are three different ways to contract $R^{Ch}$.
\begin{itemize}
\item The \emph{first Chern-Ricci form} is defined as
    \begin{align*}
    \text{Ric}^{(1)}(\omega)&=
    \text{tr}R^{Ch}=\text{tr}R^{Ch\,\bullet}_{i\bar j\bullet}\sqrt{-1}dz^i\wedge d\bar z^j\\
    &=h^{k\bar l}R^{Ch}_{i\bar jk\bar l}\sqrt{-1}dz^i\wedge d\bar z^j=
    -\frac{\partial^2\text{log}\,\text{det}(h_{k\bar l})}{\partial z^i\partial \bar z^j}dz^i\wedge d\bar z^j
    \end{align*}
It is a closed $(1,1)$-form representing the first Bott-Chern class $c_1^{BC}(X)\in H^{1,1}_{BC}(X,\mathds{R}):=\left\lbrace d\text{-closed real } (1,1)\text{-forms}\right\rbrace/
\left\lbrace \sqrt{-1}\partial\overline\partial\psi,\,\psi\in\mathcal{C}^\infty(X,\mathds{R})\right\rbrace$.
\item The \emph{second Chern-Ricci form} is defined as
    \begin{align*}
    \text{Ric}^{(2)}(\omega)&=
    \text{tr}_gR^{Ch}_{\bullet\bullet k\bar l}\,\sqrt{-1}dz^k\wedge d\bar z^l\\
    &=h^{i\bar j}R^{Ch}_{i\bar jk\bar l}\,\sqrt{-1}dz^k\wedge d\bar z^l
    \end{align*}
It is a $(1,1)$-form that is not $d$-closed in general.
\item The \emph{third Chern-Ricci form} is defined as
    \begin{align*}
    \text{Ric}^{(3)}(\omega)_{k\bar j}&=
    \text{tr}_gR^{Ch}_{\bullet\bar j k\bullet}\\
    &=h^{i\bar l}R^{Ch}_{i\bar jk\bar l}\,.
    \end{align*}
\end{itemize}
In the cases when the Chern connection satisfies the first Bianchi identity, then the Ricci curvatures coincide. Hermitian manifolds satisfying this property are called \emph{Chern K\"ahler-like} \cite{yang-zheng, angella-otal-ugarte-villacampa}. Notice that, to our knowledge, there are no examples of compact Chern K\"ahler-like manifolds that are not Chern-flat.

The \emph{Chern scalar curvature} takes then the following expression
$$
s^{Ch}(\omega)=\text{tr}_g\,{\rm Ric}^{(1)}=\text{tr}\,{\rm Ric}^{(2)}
=h^{i\bar j}h^{k\bar l}R^{Ch}_{i\bar j k\bar l}\,.
$$
Consider now a conformal change of the Hermitian metric $\omega_f=e^f\omega$, with
$f\in\mathcal{C}^\infty(X,\mathds{R})$. Denote by $R^{Ch}_f$ its Chern curvature. Then, the curvature changes as
$$
(R^{Ch}_f)_{i\bar j k\bar l}=e^f(R^{Ch}_{i\bar jk\bar l}-h_{k\bar l}\partial^2_{i\bar j}f)\,.
$$
Hence, for the Ricci forms we have
$$
\text{Ric}^{(1)}(\omega_f)=\text{Ric}^{(1)}(\omega)-n\sqrt{-1}\partial\overline{\partial}f\,,\quad
\text{Ric}^{(2)}(\omega_f)=\text{Ric}^{(2)}(\omega)+\frac{1}{2}\omega\Delta^{Ch}_{\omega}f,
$$
where $\Delta^{Ch}$ is the Chern Laplacian with local expression $\Delta^{Ch}_\omega f=-2 h^{i\bar j}\partial^2_{i\bar j}f$.\\
As a consequence, the scalar curvature changes as
$$
s^{Ch}(\omega_f)=e^{-f}\left(s^{Ch}(\omega)+\frac{n}{2}\Delta^{Ch}_\omega f\right).
$$

\subsection{Chern-Yamabe problem}

Let \( X \) be a compact holomorphic manifold of complex dimension \( n \). Given a fixed Hermitian metric \( \omega \), we denote by \( s^{Ch}(\omega) \) its scalar curvature with respect to the Chern connection, and by
\[ \{ \omega \} := \left\{ \exp(f) \omega \,\middle|\, f \in \mathcal{C}^\infty(X, \mathds{R}) \right\} \]
its conformal class, which consists entirely of Hermitian metrics.

Motivated by the classical Yamabe problem in Riemannian geometry, we ask whether the conformal class \( \{ \omega \} \) contains one or more Hermitian metrics with constant scalar curvature with respect to the Chern connection.
More precisely, we aim to study the moduli space
\[
\mathcal{C}h\mathcal{Y}a(X,\{\omega\}) := 
\frac{\left\{ \omega' \in \{\omega\} \;\middle|\; s^{Ch}(\omega') \text{ is constant} \right\}}{ \mathcal{HC}onf(X,\{\omega\}) \times \mathds{R}^{>0} },
\]
of Hermitian metrics with constant Chern scalar curvature, modulo the action of the group \( \mathcal{HC}onf(X,\{\omega\}) \) of biholomorphisms preserving the conformal class \( \{\omega\} \), and modulo the scaling action of \( \mathds{R}^{>0} \) by homotheties.

Let us observe that this problem is fundamentally different from the classical Yamabe problem. Indeed, as shown in \cite[Corollary 4.5]{liu-yang}, on a compact holomorphic manifold, if the (average of the) Chern scalar curvature of a Hermitian metric coincides with the (average of the) scalar curvature of the associated Riemannian metric, then the metric must be Kähler. For the same reason, the Chern–Yamabe problem also differs from the Yamabe-type problem for almost Hermitian manifolds studied by del Rio and Simanca in \cite{delrio-simanca}.

Our main goal is to investigate the properties of the moduli space of Chern scalar curvature metrics \( \mathcal{C}h\mathcal{Y}a(X,\{\omega\}) \). In particular, we ask whether it is always non-empty (a question we refer to as the {\em Chern-Yamabe problem}), when it reduces to a single point, and whether it can be compact under suitable assumptions.

As we will see later, under suitable normalizations, the expected constant curvature is a cohomological invariant, depending solely on the holomorphic manifold and the conformal class, known as the Gauduchon degree and denoted by $\Gamma_X(\{\omega\})$. We will prove the following theorem, conjecturing that the same result holds without the assumption that the Gauduchon degree is non-positive.

\begin{theorem}[{\cite{angella-calamai-spotti}}]
Let \(X\) be a compact holomorphic manifold endowed with a conformal class \(\{\omega\}\) of Hermitian metrics.
If the Gauduchon degree satisfies \(\Gamma_X(\{\omega\}) \leq 0\), then the Chern–Yamabe moduli space consists of a single point.
In particular, this result holds for any conformal class on a manifold with non-negative Kodaira dimension.
\end{theorem}

\subsection{Liouville-type equation}

We translate the Chern-Yamabe problem into a nonlinear PDE, in fact, a semilinear one, similar to the Liouville equation.

We first notice that, under a conformal change of the form \( e^{\sfrac{2f}{n}} \omega \), the Chern scalar curvature transforms according to the formula:
\[
s^{Ch}\left( e^{\sfrac{2f}{n}} \omega \right) = e^{-\sfrac{2f}{n}} \left( s^{Ch}(\omega) + \Delta^{Ch} f \right).
\]
We recall that the Chern Laplacian operator \( \Delta^{Ch} \), associated to \( \omega \), is given by
\[
\Delta^{Ch} f = -2\sqrt{-1} \, \mathrm{tr}_\omega(\partial\overline{\partial} f) = \Delta f + g(df, \theta),
\]
where \( \theta \) is the so-called \emph{Lee form} (or \emph{torsion 1-form}), defined by the identity \( d\omega^{n-1} = \theta \wedge \omega^{n-1} \).
Two important special cases are of interest.
\begin{itemize}
\item If \( \omega \) is balanced, that is, \( d\omega^{n-1} = 0 \), (equivalently, \( \theta = 0 \),) then the Chern Laplacian \( \Delta^{Ch} \) reduces to the usual Hodge–de Rham Laplacian \( \Delta \) on functions.

\item If \( \omega \) is Gauduchon, meaning \( \partial\overline{\partial} \omega^{n-1} = 0 \), (equivalently, \( d^*_\omega \theta = 0 \),) then the Green-type identity
\[
\int_X \Delta^{Ch} f \, \omega^n = 0
\]
holds for any smooth function \( f \).
\end{itemize}

The Chern–Yamabe problem is then reduced to finding a smooth function \( f \in \mathcal{C}^\infty(X, \mathds{R}) \) and a constant \( \lambda \in \mathds{R} \) such that
\begin{equation}\tag{ChYa}\label{eq:ChYa}
s^{Ch}(\omega) + \Delta^{Ch} f = \lambda \cdot e^{\sfrac{2f}{n}} ,
\end{equation}
which corresponds to \( s^{Ch}(\exp({\sfrac{2f}{n}})\omega) = \lambda \).

We now give a geometric interpretation of the expected constant Chern scalar curvature \(\lambda\).  
To do so, we fix a suitable reference metric in the conformal class.  


More precisely, we fix the Gauduchon representative \(\eta\) of volume \(1\) in the conformal class \(\{\omega\}\). All geometric quantities are then referred to this choice.
Up to homotheties, we also normalize the solution \(f\) by imposing the condition
\[
\frac{1}{n!} \int_X \exp\left(\sfrac{2f}{n}\right) \eta^n = 1.
\]
Note that, by the Jensen inequality, the volume of any metric in the normalized slice is always greater than or equal to $1$. It follows that, when considering the more geometric slice of metrics with unit volume, the expected constant Chern scalar curvature has absolute value greater than or equal to the Gauduchon degree.

By integrating \eqref{eq:ChYa}, we find that the expected constant Chern scalar curvature \( \lambda \) is a cohomological invariant, completely determined by \( X \) and the conformal class \( \{\omega\} \):
\[
\lambda = \int_X s^{Ch}(\eta)\, \eta^n = \frac{1}{(n-1)!} \int_X c_1^{BC}(K_X^{-1}) \wedge [\eta^{n-1}] ,
\]
where \( c_1^{BC}(K_X^{-1}) \in H^{1,1}_{BC}(X, \mathds{R}) \) maps to the first Chern class of \( X \) in \( H^2(X, \mathds{R}) \), and \( [\eta^{n-1}] \in H^{n-1,n-1}_A(X, \mathds{R}) \) is well defined thanks to \( \eta \) being Gauduchon.
This quantity is known as the \emph{Gauduchon degree} \( \Gamma_X(\{\omega\}) \) of the conformal class \( \{ \omega \} \). It is equal to the volume of the divisor associated with any meromorphic section of the anti-canonical line bundle $K_X^{-1}$, measured with respect to the Gauduchon metric $\eta$, see \cite{gauduchon-cras-2, gauduchon-mathann}.
Recall that the Kodaira dimension measures the growth of sections of the pluricanonical bundles. When \(\mathrm{Kod}(X) \geq 0\), the existence of holomorphic sections of the pluricanonical bundle implies that the anti-canonical bundle admits meromorphic sections with negative volume. As a consequence, the Gauduchon degree satisfies \( \Gamma_X(\{\omega\}) \leq 0 \) for any conformal class \(\{\omega\}\). More precisely, \(\Gamma_X(\{\omega\}) < 0\) unless the canonical bundle \(K_X\) is holomorphically torsion, \emph{i.e.}, there exists an integer \(\ell\) such that \( K_X^{\otimes \ell} = \mathcal{O}_X \), in which case \(\Gamma_X(\{\omega\}) = 0\) for all conformal classes.

\subsection{The case of expected zero Chern scalar curvature}

We first consider the case \( \Gamma_X(\{\omega\}) = 0 \), meaning that the expected Chern scalar curvature is identically zero. In this case, the Chern–Yamabe equation reduces to a linear pde:
\[
\Delta^{Ch}(f) = -s^{Ch}(\omega).
\]
To show that it is solvable, it suffices to prove that  \( -s^{Ch}(\omega) \in \operatorname{im}(\Delta^{Ch}) = \ker((\Delta^{Ch})^*)^\perp \). In this case, the solution is clearly unique up to an additive constant, since \(\ker \Delta^{Ch}\) consists of constant functions.

Note that the Gauduchon condition on \( \eta \) guarantees that the kernel of the adjoint \( (\Delta^{Ch})^* \) consists only of constant functions (cf. the proof of Theorem \ref{thm:gauduchon}). Indeed, a straightforward computation shows that
\[
(\Delta^{Ch})^*f = \Delta f - g(df, \theta) .
\]
If \( f \in \ker(\Delta^{Ch})^* \), then
\[
0 = \int_X f \, (\Delta^{Ch})^* f \, \eta^n = \int_X \left( |df|^2 - \frac{1}{2} g(d(f^2), \theta) \right) \eta^n = \int_X |df|^2 \, \eta^n,
\]
since \( d^* \theta = 0 \) for a Gauduchon metric.

Therefore, the condition \( \Gamma_X(\{\omega\}) = 0 \) ensures that \( -s^{Ch} \) lies in the orthogonal complement of \( \ker((\Delta_\eta^{Ch})^*) = \mathds C\), thus making the linear pde solvable.

\subsection{The case of expected negative Chern scalar curvature}

We now consider the case \( \Gamma_X(\{\omega\}) < 0 \), meaning that the expected constant Chern scalar curvature is negative. In this setting, we apply conformal techniques and standard elliptic theory, using the continuity method, to prove the existence of a unique solution.

First of all, by exploiting conformal arguments, we show that we can choose a reference metric \( \omega \in \{\eta\} \) such that \( s^{Ch}(\omega) < 0 \) everywhere. Indeed, as before, we can solve the linear pde \( \Delta^{Ch} f = -s^{Ch}(\eta) + \int_X s^{Ch}(\eta)\, \eta^n \). Then the conformal metric \( \exp\left(\sfrac{2f}{n}\right) \eta \) has Chern scalar curvature \( s^{Ch}\left( \exp\left( \sfrac{2f}{n} \right) \eta \right) = \exp\left( -\sfrac{2f}{n} \right) \Gamma_X(\{\omega\}) < 0 \).
From now on, we fix this metric as our reference representative in \(\{\omega\}\), and all geometric quantities will be understood with respect to it.

We then apply the continuity method to establish the existence of a metric in the conformal class \(\{\omega\}\) with constant Chern scalar curvature, initially in the regularity class \( \mathcal{C}^{2,\alpha} \). Smoothness of the solution then follows from standard elliptic regularity theory.

We consider the following continuity path, where \( \lambda := \Gamma_X(\{\omega\}) \):
\[
{\rm ChYa} \colon [0,1] \times \mathcal{C}^{2,\alpha}(X, \mathds{R}) \to \mathcal{C}^{0,\alpha}(X, \mathds{R}),
\]
\[
{\rm ChYa}(t, f) := \Delta^{Ch} f + t\, s^{Ch}(\omega) - \lambda \exp\left( \sfrac{2f}{n} \right) + \lambda\, (1 - t) .
\]
Let us define
\[
T := \left\{ t \in [0,1] : \exists f_t \in \mathcal{C}^{2,\alpha}(X, \mathds{R}) \text{ such that } {\rm ChYa}(t, f_t) = 0 \right\} .
\]
Clearly, $T$ is non-empty, since ${\rm ChYa}(0, 0) = 0$.

We prove that \(T\) is open. The implicit function theorem ensures that \(T\) is open provided that the linearization of \(\mathrm{ChYa}\) with respect to the second variable is bijective. This linearization at $(t_0,f_{t_0})$ is given by the operator
\[ D \colon \mathcal{C}^{2,\alpha}(X, \mathds{R}) \to \mathcal{C}^{0,\alpha}(X, \mathds{R}), \]
\[ Dv = \Delta^{\mathrm{Ch}} v - \sfrac{2\lambda}{n} \exp\left(\sfrac{2 f_{t_0}}{n}\right) \cdot v. \]
Let \(v \in \ker D\). Since \(D\) is an elliptic operator, standard regularity theory allows us to apply the classical maximum principle. At a maximum point \(p\) of \(v\), we have
\[
- \sfrac{2\lambda}{n} \exp\left(\sfrac{2 f_{t_0}(p)}{n}\right) \cdot v(p) \leq 0,
\]
which implies \(v(p) \leq 0\). Similarly, at a minimum point \(q\), one obtains \(v(q) \geq 0\). Therefore, \(v\) must be identically zero, {\itshape i.e.}, \(\ker D = \{0\}\).
To prove that \(D\) is surjective, we observe that the index of the linear operator \(D\) coincides with the index of the Laplacian, since \(D\) differs from the Laplacian by a compact operator. Therefore, the injectivity of \(D\) directly implies its surjectivity.

Finally, we prove that \(T\) is closed. The argument relies on the following uniform \(L^\infty\) estimates for the solutions. Let \((t_n)_n \subseteq T\) be a sequence in \(T\) converging to \(t_\infty\), and let \(f_{t_n} \in \mathcal{C}^{2,\alpha}(X, \mathds{R})\) be such that \( \mathrm{ChYa}(t_n, f_{t_n}) = 0 \).
We want to show that there exists a uniform constant \(C > 0\), depending only on \(X\), \(\omega\), and \(\lambda\), such that for all \(n\),
\[
\|f_{t_n}\|_{L^\infty} \leq C.
\]
This follows again from the maximum principle. Indeed, evaluating the equation \( \mathrm{ChYa}(t_n, f_{t_n}) = 0 \) at a point \(p \in X\) where \(f_{t_n}\) attains its maximum, we obtain
\[
-\lambda \exp\left( \sfrac{2 f_{t_n}(p)}{n} \right) \leq -t_n\, s^{\mathrm{Ch}}(\omega)(p) - \lambda (1 - t_n) \leq -\max_X s^{\mathrm{Ch}}(\omega) - \lambda,
\]
where we used the fact that \(s^{\mathrm{Ch}}(\omega) < 0\) everywhere on \(X\).
Similarly, evaluating the equation at a point \(q \in X\) where \(f_{t_n}\) attains its minimum, we obtain
\[
- \lambda \exp\left( \sfrac{2 f_{t_n}(q)}{n} \right) \geq \min\left\{ \min_X (-s^{\mathrm{Ch}}(\omega)),\ -\lambda \right\} > 0,
\]
which yields a uniform lower bound for \(f_{t_n}\) and thus completes the proof of the claim.

Consider now the equation
\[ \Delta^{\mathrm{Ch}} f_{t_n} + t_n\, s^{\mathrm{Ch}}(\omega) + \lambda(1 - t_n) = \lambda \exp\left( \sfrac{2 f_{t_n}}{n} \right), \]
where the left-hand side can be interpreted as an elliptic operator acting on \(f_{t_n}\). The uniform \(L^\infty\) estimates obtained above imply that the right-hand side is uniformly bounded in \(L^\infty\). By iterating the Calderón–Zygmund inequality and applying the Sobolev embedding theorem, it follows that \(f_{t_n}\) is uniformly bounded in \(\mathcal{C}^3(X, \mathds{R})\). Then, by the Ascoli–Arzelà theorem, there exists a subsequence converging in \(\mathcal{C}^{2,\alpha}(X)\), whose limit \(f_\infty\) solves \(\mathrm{ChYa}(t_\infty, f_\infty) = 0\).

The arguments above ensure that \(T = [0,1]\). In particular, the equation \(\mathrm{ChYa}(1, f) = 0\), corresponding to \eqref{eq:ChYa}, admits a solution \(f \in \mathcal{C}^{2,\alpha}(X, \mathds{R})\). It remains to prove that this solution is in fact smooth. This follows from a standard bootstrap argument based on Schauder estimates.

It remains to prove uniqueness. Suppose that \( \omega_1 = \exp\left(\sfrac{2f_1}{n}\right) \omega \in \{\omega\} \) and \( \omega_2 = \exp\left(\sfrac{2f_2}{n}\right) \omega \in \{\omega\} \) are two metrics with constant Chern scalar curvature \( \lambda_1 < 0 \) and \( \lambda_2 < 0 \), respectively.
Consider the difference \( f_1 - f_2 \), which satisfies the equation
\[
\Delta^{\mathrm{Ch}}(f_1 - f_2) = \lambda_1 \exp\left( \sfrac{2f_1}{n} \right) - \lambda_2 \exp\left( \sfrac{2f_2}{n} \right).
\]
Evaluating this equation at a maximum point and at a minimum point of \( f_1 - f_2 \), we deduce
\[
f_1 - f_2 = \frac{n}{2} \log\left( \sfrac{\lambda_2}{\lambda_1} \right).
\]
Therefore, the functions differ by an additive constant. Imposing the chosen normalization condition, we conclude that \( f_1 = f_2 \), proving uniqueness.

\subsection{The case of expected positive Chern scalar curvature}

In the previous arguments, the non-positivity of the curvature played a fundamental role in the application of the maximum principle. For this reason, the last case, \( \Gamma_X(\{\omega\}) > 0 \), appears to be more challenging, and it remains an open problem whether the moduli space \(\mathcal{C}h\mathcal{Y}a(X, \{\omega\})\) is non-empty. We collect here some preliminary remarks concerning this situation.

An important example is the Hopf surface, whose classical Hermitian metric has the remarkable property that the second Chern-Ricci form satisfies \( \mathrm{Ric}^{(2)} = 2\omega \), see \cite{gauduchon-ivanov}, then the Chern scalar curvature is constant and equal to \(4\). Explicit computations can be found in \cite{angella-calamai-spotti-2}. Further examples can be constructed via an implicit function theorem argument, starting from a metric whose Chern scalar curvature is sufficiently small in the \( \mathcal{C}^{0,\alpha} \)-norm.

Differently from the classical Yamabe problem, it is important to note that the Chern–Yamabe problem is not variational in general: the Chern–Yamabe equation can be interpreted as the Euler–Lagrange equation of an associated functional if and only if the conformal class contains a balanced representative. However, it remains unclear whether this functional is bounded from below.

In the positive case, uniqueness of constant Chern scalar curvature metrics within a fixed conformal class generally fails. This can be shown by adapting the argument in \cite{delima-piccione-zedda} and applying a version of the Krasnosel’skii Bifurcation Theorem. Nevertheless, we still expect compactness of the moduli space \(\mathcal{C}h\mathcal{Y}a(X, \{\omega\})\), provided the volumes are uniformly bounded.

As an additional strategy to address the existence problem, we introduced the Chern–Yamabe flow. Further developments and contributions in this direction have been made in~\cite{lejmi-maalaoui, calamai-zou, yu}.

\subsection{Moment map picture}

Constant scalar curvature metrics play a central role in Kähler geometry, as the scalar curvature naturally arises as a moment map for an infinite-dimensional Hamiltonian action \cite{fujiki, donaldson}. More precisely, the problem of finding constant scalar curvature Kähler metrics in a fixed Kähler class can be reinterpreted as the search for a zero of the moment map within the orbit of the complex structure under the complexified group action. Inspired by the Hilbert-Mumford criterion in finite-dimensional Geometric Invariant Theory, this perspective leads to the notion of \(K\)-stability, see {\itshape e.g.} \cite{szekelyhidi}.

We briefly recall the geometric setting of the Fujiki-Donaldson picture in K\"ahler geometry. Let \(X\) be a differentiable manifold endowed with a symplectic form \(\omega\) admitting Kähler structures. The space \(\mathcal{J}_{\rm alm}(\omega)\) of \(\omega\)-compatible almost-complex structures carries a natural structure of infinite-dimensional Kähler manifold \((\mathbb J, \mathbb\Omega)\), invariant under the action of the symplectomorphism group \({\rm Aut}(X, \omega)\). This action preserves the analytic subset \(\mathcal{J}(\omega)\) of integrable ({\itshape i.e.} K\"ahler) structures.
The map \( s \colon \mathcal{J}_{\mathrm{alm}}(\omega) \to \mathcal{C}^{\infty}(X, \mathds{R}) \), which assigns to each \(J\) the scalar curvature of the metric \(g := \omega(\_, J \_)\), is smooth, \({\rm Aut}(X, \omega)\)-equivariant, and satisfies the moment map equation
\[
\int_M {\rm d}(\mathrm{scal})|_J(v) \, h_Y \, \omega^n = -\tfrac{1}{2} \, \mathbb\Omega_J(Y^*_J, v),
\]
for all \(J \in \mathcal{J}(\omega)\), \(v \in T_J \mathcal{J}_{\mathrm{alm}}(\omega)\), and \(Y \in \mathfrak{ham}(X, \omega)\). Here, \(\mathfrak{ham}(X, \omega)\) denotes the Lie subalgebra of infinitesimal automorphisms \(Y\) satisfying \(Y \lrcorner \omega = \mathrm{d} h_Y\) for some smooth function \(h_Y \in \mathcal{C}^\infty(M, \mathds{R})\), unique up to an additive constant. The fundamental vector field associated to \(Y\) is denoted by \(Y^*\). 
Hence, \(s\) is called a \emph{moment map} for the action of the Hamiltonian diffeomorphism group \({\rm Ham}(M, \omega)\) on \(\mathcal{J}(\omega)\).

Motivated by the desire to gain a deeper understanding of the role played by the Chern scalar curvature in non-Kähler geometry, in \cite{angella-calamai-pediconi-spotti} we extended the classical Donaldson–Fujiki framework to the broader setting of locally conformally Kähler geometry.
In this context, we consider special conformal vector fields, satisfying \( d_\theta(Y \lrcorner\, \omega) = 0 \), and \emph{twisted Hamiltonian vector fields}, that is, vector fields \( Y \) for which \( Y \lrcorner\, \omega = d_\theta h \), where the twisted differential is defined as \( d_\theta := d - \theta \).
A further technical subtlety arises from the fact that the automorphism group \( {\rm Aut}^*(X, \{\omega\}) \) preserves only the conformal class \( \{\omega\} \), so that the Riemannian scalar curvature is not \( {\rm Aut}^*(X, \{\omega\}) \)-equivariant. Additionally, an extra term involving \( D(\theta^{\sharp}) \) appears in the computations, which we do not know how to handle. These difficulties force us to introduce an additional symmetry assumption: we restrict to the subspace of structures invariant under the torus \( T \), defined as the closure of the flow generated by \( \theta^{\sharp_\omega} \), the symplectic dual of the Lee form.
Our main result can then be stated as follows.

\begin{theorem}[{\cite{angella-calamai-pediconi-spotti}}]
Let \( X \) be a compact holomorphic manifold of complex dimension \( n \), endowed with a locally conformally Kähler structure \( \omega \) with Lee form \( \theta \). 
The action of \( T \)-invariant twisted Hamiltonian diffeomorphisms on the space \( \mathcal{J}(\omega)^T \) of \( T \)-invariant compatible complex structures is Hamiltonian, with moment map
\[
\mu = s^{\mathrm{Ch}}(\omega) + n\, d^*\theta.
\]
That is, for every \( J \in \mathcal{J}(\omega)^T \), \( A \in T_J \mathcal{J}(\omega)^T \), and twisted Hamiltonian vector field \( Y \) with twisted Hamiltonian potential \( h \) and fundamental vector field \( Y_J^* \), the moment map equation holds:
\[
\int_X \left. d\mu \right|_J(A)\, h\, \omega^n = -\frac{1}{2} \, \mathbb\Omega_J(Y_J^*, A) .
\]
\end{theorem}

These considerations suggest that, rather than requiring constant Chern scalar curvature, a more natural notion of canonical metric might be characterized by the constancy of the moment map \( \mu \). This class includes, as special cases, both constant scalar curvature Kähler metrics and Gauduchon metrics with constant scalar curvature. See \cite{angella-pediconi-scarpa-spotti-windare} for further details.

\ifdefined\frommain
\else
    \bibliographystyle{alpha}
    \bibliography{biblio}
\fi

\subsection{Problems}

\begin{exercise}
Let \( X \) be a holomorphic manifold endowed with a Hermitian metric \( \omega \), and let \( f \in \mathcal{C}^\infty(X, \mathds{R}) \) be a smooth real-valued function.
Prove the following formula for the conformal change of the Chern scalar curvature:
\[
s^{Ch}\left(e^{\sfrac{2f}{n}}\omega\right) = e^{-\sfrac{2f}{n}} \left( s^{Ch}(\omega) + \Delta^{Ch}_\omega f \right),
\]
where \( \Delta^{Ch}_\omega f := 2\, \mathrm{tr}_\omega(\sqrt{-1}\,\partial\overline{\partial} f) \) is the Chern Laplacian of \( f \) with respect to the metric \( \omega \).
\end{exercise}

\begin{exercise}[{\cite{angella-calamai-spotti-2}}]
Let \( X \) be a compact holomorphic manifold, and let \( \omega \) be a \emph{(weak) first Chern-Einstein} Hermitian metric on \( X \), namely
\[
\mathrm{Ric}^{(1)}(\omega) = \lambda\, \omega
\]
for some smooth function \( \lambda \neq 0 \). Show that \( \omega \) is conformal to a Kähler metric.
\end{exercise}

\begin{proof}[Hint]
Consider the unique Gauduchon metric $\eta=e^f\omega$ with total volume one in the conformal class of $\omega$ and prove that this metric is K\"ahler. To do so, take the function $\tilde\lambda=\lambda e^{-f}$ and notice that the first Chern-Einstein condition implies $\partial(\tilde\lambda\eta)=0$. To prove that $\eta$ is K\"ahler it is then enough to prove that $\tilde\lambda\neq 0$ is constant. The argument now reduces to show that the form $\tilde\lambda^{n-1}\sqrt{-1}\partial\overline\partial\tilde\lambda^{n-1}\wedge\eta^{n-1}$ is $d$-exact and show that the following equalities hold
$$
0=\int_X \tilde\lambda^{n-1}\sqrt{-1}\partial\overline\partial\tilde\lambda^{n-1}\wedge\eta^{n-1}=
\frac{1}{2n}\int_X |d\tilde\lambda^{n-1}|^2\eta^n\,.
$$
\end{proof}

\begin{exercise}[{non-Kähler Calabi–Yau, \cite{tosatti}}]
Let \( X \) be a compact holomorphic $n$-dimensional manifold with \( c_1^{\mathrm{BC}}(X) = 0 \). Show that, for any Hermitian metric \( \omega \), there exists a conformal representative \( \omega' = e^f \omega \), for some smooth function \( f \), such that
\[
\mathrm{Ric}^{(1)}(\omega') = 0.
\]
The function \( f \) is unique up to an additive constant.
\end{exercise}

\begin{proof}[Hint]
Since $c_1^{BC}(X)=0$, then $\mathrm{Ric}^{(1)}(\omega)=\sqrt{-1}\partial\overline{\partial}f$ for some smooth function $f$. Show that $\omega'=e^{\frac{f}{n}}\omega$ is first Chern-Ricci flat. The uniqueness up to constants comes from the uniqueness up to constants of the Ricci potential.    
\end{proof}

\begin{proof}[Remark]
By \cite{tosatti-weinkove}, the solution can also be expressed as \( \omega' = \omega + \sqrt{-1} \, \partial \overline{\partial} \varphi \) for some smooth function \( \varphi \), unique up to an additive constant.
\end{proof}

\begin{exercise}[{\cite{gauduchon-cras-2, angella-calamai-spotti-2}}]
Let \( X \) be a compact holomorphic manifold endowed with a Hermitian metric \( \omega \) such that 
\[
\mathrm{Ric}^{(2)}(\omega) = \lambda \omega
\]
for some smooth function \( \lambda \). 
\begin{itemize}
\item Show that the property is invariant for conformal changes.
\item Show that, up to a conformal change, one can assume \( \lambda \) to have constant sign, equal to the sign of the Gauduchon degree of the conformal class \( \{ \omega \} \).
\item Assuming a positive solution to the Chern–Yamabe problem, prove that one can choose a representative in the conformal class with \( \lambda \) constant.
\end{itemize}
\end{exercise}

\begin{proof}[Hint]
Show that, under a conformal change, the second Chern-Ricci curvature transforms as \( \mathrm{Ric}^{(2)}(e^f\, \omega) = \mathrm{Ric}^{(2)}(\omega) + \frac{1}{2}(\Delta^{\mathrm{Ch}}_\omega f)\, \omega \).
It then follows that one can use a conformal change to let the Einstein factor have definite sign. To do so, choose the unique Gauduchon metric $\eta$ in the conformal class of $\omega$, with total volume one, and show that the following equation has a smooth solution $f\in\mathcal{C}^\infty(X,\mathbb{R})$,
$$
\Delta^{Ch}f=-s^{Ch}(\eta)+\int_X s^{Ch}(\eta)d\mu_{\eta}\,.
$$
Then the metric $e^{\frac{2f}{n}}\eta$ is second Chern-Einstein and its Einstein factor has definite sign, depending on the sign of its Chern scalar curvature, that equals the sign of the Gauduchon degree of the conformal class \( \{ \omega \} \).
Finally, show that the problem of finding second Chern-Einstein metrics in \( \{ \omega \} \) with $\lambda$ constant reduces to solve the Liouville type equation
$$
\Delta^{Ch}_\eta f+s^{Ch}(\eta)=\lambda e^{\frac{2f}{n}}
$$
with $\eta$ Gauduchon metric in \( \{ \omega \} \).
\end{proof}

\begin{exercise}[{\cite{gauduchon-ivanov}}]
Prove that the only compact complex non-Kähler surface admitting a Hermitian metric \( \omega \) such that \( \mathrm{Ric}^{(2)}(\omega) = f\, \omega \) for some smooth function \( f \), is the Hopf surface.
\end{exercise}

\begin{proof}[Remark]
On compact holomorphic manifolds, the existence of Hermitian metrics satisfying \( \mathrm{Ric}^{(2)}(\omega) = f\, \omega \) for some smooth function \( f \) is subject to several obstructions, see \cite{teleman, gauduchon-cras-2, gauduchon-bsmf, yang}. In particular:
the sign of the Einstein factor \( f \) is constrained by the pseudo-effectiveness or unitary flatness of \( K_X \) and \( K_X^{-1} \);
second-Chern-Ricci-flat metrics can only exist if \( \mathrm{Kod}(X) \leq 0 \);
manifolds having second-Chern-Einstein metrics with positive (respectively, negative) curvature do not admit non-trivial holomorphic \( p \)-forms (respectively, holomorphic \( p \)-vector fields), for any \( p \geq 1 \).
Moreover, such metrics are weakly \( g \)-Hermitian–Einstein, see {\itshape e.g.} \cite{kobayashi-Nagoya, lubke-teleman-book, kobayashi-book}, and thus satisfy further geometric constraints:
the Bogomolov–Lübke inequality holds \cite{bogomolov, lubke};
by the Kobayashi–Hitchin correspondence, the holomorphic tangent bundle is \( g \)-semi-stable \cite{kobayashi, lubke-2}.
\end{proof}

\ifdefined\frommain
\else
    \bibliographystyle{alpha}
    \bibliography{biblio}
\fi

\cleardoublepage

\section{Geometric flows in non-K\"ahler geometry}\label{sec:flows}

In this section, we present the basic theory of the Chern-Ricci flow \cite{gill, tosatti-weinkove-CRF, tosatti-weinkove-surveyCRF}, with a focus on its applications to the study of compact complex surfaces \cite{tosatti-weinkove-surfacesCRF}. In particular, we discuss the long-time behavior of the flow on Inoue surfaces, following \cite{fang-tosatti-weinkove-zheng, angella-tosatti}.

\subsection{Geometric flows of Hermitian metrics}

Over the past few decades, the Ricci flow, originally introduced by Hamilton \cite{hamilton}, has emerged as a powerful tool in differential geometry. It played a central role in the resolution of the Poincaré and Geometrization conjectures, through the groundbreaking work of Perelman \cite{perelman-1, perelman-2, perelman-3}.

In the Hermitian setting, when the initial metric is Kähler, the flow preserves the Kähler condition \cite{hamilton, shi}; in this case, it is referred to as the \emph{Kähler-Ricci flow}.
In his seminal work, Cao \cite{cao} reproved the Aubin-Yau theorem using the Kähler-Ricci flow by showing that, on a compact Kähler manifold \( X \) with \( c_1(X) = 0 \) (respectively, \( c_1(X) < 0 \)), the Kähler-Ricci flow (respectively, the normalized K\"ahler-Ricci flow) starting from any Kähler metric \( \omega_0 \) converges smoothly to the unique Kähler-Einstein metric in the class \( [\omega_0] \) (respectively, in the class \( -c_1(X) \)).

On the other hand, if the initial metric is just Hermitian, the Ricci flow generally fails to preserve Hermitianity, except when the curvature of its Levi-Civita connection satisfies all Kähler symmetries, see \cite{angella-otal-ugarte-villacampa}. To address this issue, various geometric flows have been introduced to deal with the Hermitian non-K\"ahler context. These include the \emph{Chern-Ricci flow} \cite{gill, tosatti-weinkove-CRF}, the family of \emph{Hermitian curvature flows} introduced by Streets and Tian \cite{streets-tian}, such as the \emph{pluriclosed flow} \cite{streets-tian-IMRN} and the flow studied by Ustinovskiy \cite{ustinovskiy}, the \emph{anomaly flows} \cite{fei-phong-picard-zhang-CambJMath, phong-picard-zhang-JDG, phong-picard-zhang-MathZ, phong-picard-zhang-CAG, picard-LNM}, the \emph{generalized Ricci flow} \cite{garciafernandez-streets}, flows coupling the Ricci flow with a flow evolving the complex structure, and many others.

\subsection{Introduction to the Chern-Ricci flow}

In this section, we focus on the \emph{Chern-Ricci flow}. It was introduced by Gill \cite{gill} in the study of compact non-Kähler Calabi–Yau manifolds, that is, compact holomorphic manifolds \( X \) with vanishing first Bott–Chern class \( c_1^{BC}(X) = 0 \). Then it was later studied extensively by Tosatti and Weinkove \cite{tosatti-weinkove-CRF, tosatti-weinkove-surfacesCRF, tosatti-weinkove-surveyCRF}. One expects the behaviour of the Chern-Ricci flow to reflect the underlying holomorphic structure of \( X \), making it a powerful tool in the classification theory of compact non-Kähler holomorphic surfaces. A key feature of the Chern-Ricci flow is that, like the Kähler-Ricci flow, it reduces to a scalar parabolic equation, and thus inherits many interesting analytic properties.

Let \( X \) be a compact holomorphic manifold, of complex dimension \( n \), and let \( \omega_0 \) be a fixed Hermitian metric. The \emph{Chern-Ricci flow} evolves the initial metric according to the equation
\begin{equation}\tag{CRF}\label{eq:crf}
\frac{\partial \omega(t)}{\partial t} = -\mathrm{Ric}^{Ch}(\omega(t)), \quad \omega(0) = \omega_0,
\end{equation}
for \( 0 \leq t < T \), where \( 0 < T \leq +\infty \) will denote the maximal existence time, see below.
Here, \( \mathrm{Ric}^{Ch} \stackrel{\mathrm{loc}}{=} -\sqrt{-1} \, \partial \overline{\partial} \log \det \omega \) denotes the first Chern-Ricci curvature, that is, the tensor \( \mathrm{Ric}^{(1)} \) associated to the Chern connection.

We can reduce the Chern-Ricci flow to a parabolic complex Monge-Ampère equation.

\begin{theorem}[{\cite{gill, tosatti-weinkove-CRF}}]
Let \( X \) be a compact holomorphic manifold endowed with a Hermitian metric \( \omega_0 \). Then the Chern-Ricci flow \eqref{eq:crf} admits a solution \( (\omega(t))_{0 \leq t < T} \) if and only if there exists a family of smooth functions \( (\varphi(t))_{0 \leq t < T} \) solving the following parabolic equation for scalar functions:
\begin{equation}\tag{P$\mathds C$MA}\label{eq:pcma}
\left\{
\begin{array}{l}
\frac{\partial \varphi}{\partial t} = \log \frac{(\omega_0 - t\, \mathrm{Ric}^{Ch}(\omega_0) + \sqrt{-1} \partial \overline{\partial} \varphi)^n}{\omega_0^n}, \\[1.2ex]
\omega_0 - t\, \mathrm{Ric}^{Ch}(\omega_0) + \sqrt{-1} \partial \overline{\partial} \varphi > 0, \\[0.5ex]
\varphi(0) = 0.
\end{array}
\right.
\end{equation}
Moreover, the two solutions are related by
\[
\omega(t) = \omega_0 - t\, \mathrm{Ric}^{Ch}(\omega_0) + \sqrt{-1} \partial \overline{\partial} \varphi(t).
\]
\end{theorem}

By standard parabolic theory, assuming \( X \) is compact, the Chern-Ricci flow admits a unique solution on a maximal time interval \( [0, T) \), for some \( 0 < T \leq +\infty \). The maximal existence time is characterized by \cite{tosatti-weinkove-CRF} as
\[
T = \sup \left\{ t > 0 : \exists \psi \in \mathcal{C}^\infty(X, \mathds{R}) \text{ such that } \omega_0 - t\, \mathrm{Ric}^{Ch}(\omega_0) + \sqrt{-1} \partial \overline{\partial} \psi > 0 \right\}.
\]
This parallels the result in the Kähler case, the only difference being the presence of torsion terms, which complicate the analytic estimates.
In particular, note that the condition \( T = +\infty \) is independent of the initial metric \( \omega_0 \), and it holds if and only if the canonical bundle \( K_X \) is \emph{nef} (this means that, for every \( \varepsilon > 0 \), there exists a smooth function \( \psi_\varepsilon \in C^\infty(X, \mathds{R}) \) such that \( - \mathrm{Ric}^{Ch}(\omega_0) + \sqrt{-1} \partial \overline{\partial} \psi_\varepsilon > -\varepsilon \omega_0 \)).

\subsection{Chern-Ricci flow for non-K\"ahler Calabi-Yau manifolds}

To motivate the study of the Chern-Ricci flow, we recall a fundamental result by \cite{gill} concerning non-Kähler Calabi–Yau manifolds \cite{tosatti}.

Recall that a compact holomorphic manifold $X$, of complex dimension $n$, is called {\em non-K\"ahler Calabi-Yau} if its first Bott–Chern class vanishes in Bott-Chern cohomology, \( c_1^{BC}(X) = 0 \) in \( H^{1,1}_{BC}(X, \mathds{R}) \).
In the Kähler case, this condition is equivalent to \( c_1(X) = 0 \) in de Rham cohomology \( H^2_{dR}(X, \mathds{R}) \), and the celebrated Calabi–Yau theorem guarantees the existence of a Ricci-flat Kähler metric.

In the non-Kähler setting, the following conditions are equivalent, see \cite{tosatti}:
\begin{itemize}
    \item \( c_1^{BC}(X) = 0 \) in \( H^{1,1}_{BC}(X, \mathds{R}) \);
    \item there exists a Hermitian metric \( \omega \) on \( X \) such that \( \mathrm{Ric}^{Ch}(\omega) = 0 \). Such a Chern-Ricci-flat metric can be obtained from a fixed Hermitian metric \( \omega_0 \) through various ansätze \cite{tosatti, szekelyhidi-tosatti-weinkove}: a conformal ansatz, \( \omega = e^\varphi \omega_0 \); a Calabi-type ansatz, \( \omega = \omega_0 + \sqrt{-1} \partial \overline{\partial} \varphi \); a Fu–Wang–Wu-type ansatz, \( \omega^{n-1} = \omega_0^{n-1} + \sqrt{-1} \partial \overline{\partial} \varphi \wedge \omega_1^{n-2} \), where \( \omega_1 \) is another fixed Hermitian metric on \( X \);
    \item the canonical bundle \( K_X \) is unitary flat, {\itshape i.e.}, it admits a Hermitian metric with vanishing curvature;
    \item there exists a Hermitian metric \( \omega \) whose Chern connection has restricted holonomy contained in \( {\rm SU}(n) \).
\end{itemize}
We observe that the condition that \( K_X \) is holomorphically torsion, namely, \( K_X^{\otimes \ell} \simeq \mathcal{O}_X \) for some \( \ell \geq 1 \), implies \( c_1^{BC}(X) = 0 \), but it is generally strictly weaker, see \cite{nakamura, magnusson}. Conversely, there exist holomorphic manifolds for which the de Rham class \( c_1(X) = 0 \) vanishes in \( H^2_{dR}(X, \mathds{R}) \), but \( c_1^{BC}(X) \neq 0 \): examples include Hopf surfaces and the hypothetical complex structure on the six-sphere \( {S}^6 \).

Similarly to Cao's theorem \cite{cao}, which gives a parabolic proof of the Calabi–Yau theorem for Kähler manifolds with \( c_1(X) = 0 \) via the Kähler-Ricci flow, an analogous result holds in the non-Kähler setting for compact non-Kähler Calabi–Yau manifolds.

\begin{theorem}[{\cite{gill}}]
Let \( X \) be a compact non-Kähler Calabi–Yau manifold. Then, for any initial Hermitian metric \( \omega_0 \), the Chern-Ricci flow \eqref{eq:crf} admits a long-time solution \( \omega(t) \) defined for all \( t \in [0, +\infty) \), which converges smoothly as \( t \to +\infty \) to a Hermitian metric \( \omega_\infty \) satisfying
\[
{\rm Ric}^{Ch}(\omega_\infty) = 0.
\]
More precisely, \( \omega_\infty \) is the unique Chern-Ricci flat metric of the form \( \omega_\infty = \omega_0 + \sqrt{-1} \, \partial \overline{\partial} \varphi \) for some smooth real function \( \varphi \).
\end{theorem}

\subsection{Chern-Ricci flow on compact complex surfaces}

We now focus on complex dimension \( 2 \), that is, on compact  holomorphic surfaces. As a first observation, we note that the Gauduchon condition, which in this setting is equivalent to the pluriclosed condition, is preserved along the Chern-Ricci flow, since the Chern-Ricci form \( {\rm Ric}^{Ch} \) is \( d \)-closed.
When the initial metric \( \omega_0 \) is Gauduchon, the maximal existence time of the flow can be also characterized as follows, thanks to Buchdahl’s Nakai–Moishezon criterion \cite{buchdahl}:
\[
T = \sup \left\{ \tau \geq 0 :  
\begin{array}{lc}
&\text{for all } t \in [0, \tau], \text{ the form } \alpha_t := \omega_0 - t\, \mathrm{Ric}^{Ch}(\omega_0) \text{ satisfies} \\
&\int_X \alpha_t^2 > 0 \text{ and } \int_D \alpha_t > 0 \\
& \text{ for every irreducible effective divisor } D \text{ with } D^2 < 0
\end{array}
\right\}.
\]
In other words, the Chern-Ricci flow exists as long as the volume of \( X \) with respect to \( \omega(t) \) remains positive and the volume of every irreducible curve with negative self-intersection does not collapse to zero.

We recall that compact holomorphic surfaces can be classified according to their Kodaira dimension, which takes the values \( -\infty \), \( 0 \), \( 1 \), or \( 2 \). For a detailed exposition of the Enriques–Kodaira–Siu classification, we refer the reader to the comprehensive book \cite{barth-hulek-peters-vandeven}.

A first fundamental distinction is whether the surface contains a \((-1)\)-curve, that is, a smooth rational curve \( D \simeq \mathds CP^1 \) with self-intersection \( D^2 = -1 \). If it does not, the surface is said to be \emph{minimal}.
By the Castelnuovo contraction theorem, every \((-1)\)-curve can be contracted. Consequently, any compact holomorphic surface can be obtained from a minimal one by blowing up finitely many points.

If \( X \) is {\bfseries non-minimal}, then there exists a \( (-1) \)-curve \( D \). By the adjunction formula, \( K_D = (K_X + D)|_D \), and since \( K_D \cdot D = 2g(D) - 2 = -2 \), it follows that \( K_X \cdot D = -1 < 0 \). Therefore, along the Chern-Ricci flow starting at any Gauduchon metric \( \omega_0 \), we have \( \int_D \omega(t) = \int_D \omega_0 + t \, K_X \cdot D = \int_D \omega_0 - t \), which becomes zero in finite time. Hence, the maximal existence time of the flow must be finite, \( T < +\infty \).
Conversely, if the initial metric $\omega_0$ is Gauduchon, the maximal existence time satisfies \( T < +\infty \), and the flow is non-collapsing, meaning that \( \mathrm{Vol}(X, \omega(t)) \geq c > 0 \) as \( t \to T^- \), then \( X \) contains \((-1)\)-curves, see \cite{tosatti-weinkove-CRF}.
In this case, denote the $(-1)$-curves by \( E_1, \dots, E_k \). Then, according to \cite{tosatti-weinkove-CRF}, there exists a smooth Gauduchon metric \( \omega_T \) on \( X \setminus \bigcup_i E_i \) such that \( \omega(t) \to \omega_T \) in \( \mathcal{C}^\infty_{\mathrm{loc}}(X \setminus \bigcup_i E_i) \) as \( t \to T^- \). Furthermore, the metric spaces \( (X, \omega(t)) \) converge in the Gromov-Hausdorff topology to a compact metric space \( (N, d_N) \) that is homeomorphic to the complex space obtained by contracting the curves \( E_1, \dots, E_k \) in \( X \); denote this map by \( \pi \colon X \to N \), with \( E_i = \pi^{-1}(p_i) \).
Independently, \cite{to-MathAnn, nie-MRL} proved that the Chern-Ricci flow can be continued in a weak sense on the limit space \( N \), with backward convergence in time. The argument above by Tosatti and Weinkove is proved under the additional assumption that \( d\omega_0 = \pi^*(d\beta) \) for some smooth real $(1,1)$-form \( \beta \) on \( N \), which in particular implies that \( d\omega_0 = 0 \) on the exceptional divisors of \( \pi \), see condition (*) in \cite{tosatti-weinkove-surfacesCRF}.
Once this condition is removed, what remains is to show that \( (N, d_N) \) is the metric completion of \( (N \setminus \{p_i\}, \omega_T) \). These two remaining steps would complete the proof that the Chern-Ricci flow contracts the \( (-1) \)-curves in the Gromov-Hausdorff topology, while converging smoothly outside the exceptional divisors.
This result would demonstrate that the Chern-Ricci flow performs a canonical surgical contraction, generalizing the results of \cite{song-tian} for Kähler surfaces, and fitting into the broader framework of the analytic Minimal Model Program.

In what follows, we will assume that the surface is {\bfseries minimal}.
Surfaces \( S \) of {\bfseries general type}, that is, those with \( \mathrm{Kod}(S) = 2 \), are always algebraic. 
Restricting attention to non-Kähler surfaces, which are topologically characterized by having odd first Betti number \cite{lamari, buchdahl}, only the following cases remain.

\begin{itemize}
\item For \( {\rm Kod}(X) = 1 \), we have \textbf{minimal non-K\"ahler properly elliptic surfaces}, which are described, up to finite covers, as elliptic bundles, see {\itshape e.g.} \cite{brinzanescu, brinzanescu-book}. In this case, the Chern-Ricci flow starting from any Gauduchon metric exists for all \( t \geq 0 \), and exhibits the following collapsing behavior: the normalized metrics \( \sfrac{\omega(t)}{t} \) converge, in the Gromov–Hausdorff sense, to a Riemann surface equipped with the distance function induced by an orbifold Kähler–Einstein metric, see \cite{tosatti-weinkove-MathAnn}.

\item For \( \mathrm{Kod}(X) = 0 \), we have \textbf{Kodaira surfaces}, which satisfy \( c_1^{BC}(X) = 0 \), meaning they are non-K\"ahler Calabi–Yau manifolds. This case was studied in detail by \cite{gill}, who showed that the Chern-Ricci flow exists for all time and converges to the unique Chern-Ricci flat metric in the \( \sqrt{-1}\partial\overline{\partial} \)-class of the initial metric.

\item Lastly, the {\bfseries class VII} refers to compact holomorphic surfaces satisfying \( \mathrm{Kod}(X) = -\infty \) and \( b_1(X) = 1 \). 
Those with \( b_2(X) = 0 \) are known to be either Hopf surfaces or Inoue–Bombieri surfaces \cite{inoue, bombieri}, as established in \cite{bogomolov, li-yau-zheng, teleman-IJM}. 
The first known examples with \( b_2(X) > 0 \) were introduced in \cite{inoue-1, inoue-2}, and subsequently extended by Kato \cite{kato}, who developed a construction based on iteratively blowing up the standard ball in \( \mathds{C}^2 \), followed by holomorphic surgery.
These surfaces are characterized by the existence of a \emph{global spherical shell}, that is, an open subset \( U \subset X \) biholomorphic to a neighborhood of the \( 3 \)-sphere \( S^3 \) in \( \mathds{C}^2 \), such that the complement \( X \setminus U \) remains connected. Moreover, they can be viewed as degenerations of blown-up primary Hopf surfaces \cite{nakamura-Sugaku} and admit singular holomorphic foliations \cite{dloussky-oeljeklaus}. Each Kato surface \( X \) contains exactly \( b_2(X) \) rational curves, and conversely, every minimal compact complex surface in class VII possessing precisely \( b_2(X) > 0 \) rational curves is a Kato surface \cite{dloussky-oeljeklaus-toma}. It is conjectured that no other examples exist in class VII beyond these, a statement known as the Global Spherical Shell conjecture.

For classical Hopf surfaces of the form \(\sfrac{\mathds{C}^2 \setminus \{0\}}{\mathds{Z}}\), where \(\mathds{Z}\) is generated by the map \(\lambda \cdot {\rm id}\) with \(\lambda \in \mathds{C}\) satisfying \(0 < |\lambda| < 1\), the solution to the Chern-Ricci flow starting from the standard metric admits an explicit expression \cite{tosatti-weinkove-surfacesCRF}. More precisely, the flow collapses in finite time and converges smoothly to a non-negative \((1,1)\)-form whose kernel defines a smooth distribution, whose iterated Lie brackets generate the full tangent space of the Sasaki \(S^3\). In this way, the limiting behavior of the Chern-Ricci flow encodes both the differential and complex geometric structures inherent to the Hopf surface. Further insights can be found in \cite{edwards}, where results support the conjecture that the Gromov-Hausdorff limit is isometric to a round \(S^1\).
The Chern-Ricci flow on Inoue-Bombieri surfaces has been extensively studied in \cite{tosatti-weinkove-surfacesCRF, lauret, lauret-valencia, fang-tosatti-weinkove-zheng, angella-tosatti}, which will be discussed in detail in the following subsection.
Finally, the ``optimistic conjecture'' in \cite[Section 4.6]{fang-tosatti-weinkove-zheng} seeks to recover a global spherical shell through the asymptotic behavior of the Chern-Ricci flow.
\end{itemize}

\subsection{Chern-Ricci flow on Inoue surfaces}

Inoue surfaces are divided into three distinct families: \( S_M \), \( S^+ \), and \( S^- \), each depending on parameters.

We first recall the construction and the geometric properties of the {\bfseries Inoue–Bombieri surface of type \( S_M \)} \cite{inoue, bombieri}. Let \( M \in {\rm SL}(3, \mathds{Z}) \) be a matrix with one real eigenvalue \( \lambda > 1 \), and a pair of complex conjugate eigenvalues \( \mu, \bar{\mu} \in \mathds{C} \). Fix a real eigenvector \( (\ell_1, \ell_2, \ell_3) \in \mathds{R}^3 \) corresponding to \( \lambda \), and a complex eigenvector \( (m_1, m_2, m_3) \in \mathds{C}^3 \) corresponding to \( \mu \).
For later use, we observe that \( \lambda \) is an irrational algebraic number, and at least one of the slopes of the real eigenvector, say \( \sfrac{\ell_1}{\ell_2} \), is irrational.
On \( \mathds{C} \times \mathds{H} \), where \( \mathds{H} = \{ w \in \mathds{C} : \Im w > 0 \} \) is the upper half-plane, consider complex coordinates \( (z = x_1 + \sqrt{-1} y_1, w = x_2 + \sqrt{-1} y_2) \), where $y_2>0$. We define an action of the group generated by the following automorphisms:
\[
f_0(z,w) = (\mu z, \lambda w), \qquad
f_j(z,w) = (z + m_j,\; w + \ell_j) ,
\]
for \( j \in \{1,2,3\} \).
The quotient
\[ S_M := \mathds{C} \times \mathds{H} / \langle f_0, f_1, f_2, f_3 \rangle \]
defines a compact, holomorphic, non-Kähler surface, called {\em Inoue–Bombieri surface of type \( S_M \)}. The fact that it does not admit any Kähler metrics follows from \( b_1(S_M) = \rk \left( \sfrac{\pi_1(S_M)}{[\pi_1(S_M), \pi_1(S_M)]} \right) = 1 \) being odd.

The Inoue–Bombieri surfaces $S_M$ possess a rich differentiable and holomorphic structure, which we now describe. The projection \( \mathds{C} \times \mathds{H} \ni (z, w) \mapsto \Im w \in \mathds{R}^{>0} \) descends to the quotient and induces the structure of a differentiable torus bundle:
\[
\xymatrix{
T^3 \ar@{^{(}->}[r] & S_M \ar[d]^\pi \\
& S^1
}
\]
where the fibre is a \( 3 \)-torus \( T^3 \simeq \sfrac{\mathds R^3}{\langle f_1,f_2,f_3\rangle} \simeq \sfrac{\mathds R^3}{\mathds Z^3} \), and the base is the circle obtained as \( \sfrac{\mathds R^{>0}}{\langle\lambda\rangle} \).
From the holomorphic perspective, consider the real, closed $(1,1)$-form
\[
\alpha := \frac{\sqrt{-1}}{y_2^2} \, dw \wedge d\bar{w},
\]
which is non-negative and represents \( c_1(K_{S_M}) = -c_1^{BC}(S_M) \). Its kernel defines a holomorphic parabolic foliation whose leaves are biholomorphic to \( \mathds{C} \).
A key point is the strict relation between the differentiable and holomorphic structures, arising from the irrationality in the construction. More precisely, the leaves of the holomorphic foliation are dense in the fibres of the smooth torus bundle \cite{inoue}, see \cite[Lemma 5.2]{tosatti-weinkove-surfacesCRF} for details.
Concerning the metric structure, we note that 
\[
\omega_T := 4\alpha + \sqrt{-1} \, y_2 \, dz \wedge d\bar{z}
\]
defines a Gauduchon and locally conformally Kähler metric on \( S_M \), called the Tricerri metric \cite{tricerri-Torino}, with the interesting property that its restriction to any leaf is a flat Kähler metric.

For details on the construction of the {\bfseries Inoue-Bombieri surfaces of type \( S^+ \)}, we refer to \cite{inoue}, see also \cite{angella-tosatti}. These surfaces share many geometric features with those of type \( S_M \). The main difference is that \( S^+ \) surfaces are bundles over \( S^1 \), whose fibres are compact quotients of the three-dimensional Heisenberg group. Once again, they admit a holomorphic parabolic foliation with leaves biholomorphic to \( \mathds{C}^\times \), which are dense in the fibres of the bundle.
These surfaces also carry a Gauduchon metric, (but not necessarily locally conformally Kähler,) originally constructed by Tricerri \cite{tricerri-Torino} and by Vaisman \cite{vaisman-PoliTo}.

Furthermore, every {\bfseries Inoue–Bombieri surface of type \( S^- \)} admits an unramified double cover given by an Inoue surface of type \( S^+ \).

In the following, we consider the Chern-Ricci flow on an Inoue–Bombieri surface \( S \), of type either \( S_M \) or \( S^\pm \), starting from a Gauduchon metric.  
Since \( c_1^{BC}(S) \leq 0 \), the flow exists for all times \( t \geq 0 \).  
To study its long-time behaviour, it is convenient to rescale the solution. Therefore we study the {\em normalized Chern-Ricci flow}, defined by
\[
\frac{\partial}{\partial t} \omega(t) = -\mathrm{Ric}^{Ch}(\omega(t)) - \omega(t), \quad \omega(0) = \omega_0,
\]
where \( \omega_0 \) is a Gauduchon metric on \( S \).

In \cite{tosatti-weinkove-surfacesCRF}, the explicit solution to the normalized Chern–Ricci flow starting from the Tricerri/Vaisman metric \( \omega_T \) is computed and analyzed. It is shown that \( \omega(t) \to \alpha \) uniformly as \( t \to +\infty \), and that the associated metric spaces \( (S, \omega(t)) \) converge in the Gromov–Hausdorff sense to a circle \( S^1 \). This convergence has been extended to the case where the initial metric is any locally homogeneous metric \cite{lauret-TAMS, lauret-valencia}, and more generally, to metrics of the form \( \omega_{\text{SFL}} + \sqrt{-1}\partial\overline{\partial} u \), for $u$ smooth real function, where \( \omega_{\text{SFL}} \) is strongly flat along the leaves \cite{fang-tosatti-weinkove-zheng}.
Here, the condition of being \emph{strongly flat along the leaves} means that the restriction of \( \omega_{\text{SFL}} \) to each leaf \( \mathcal{L} \) of the holomorphic foliation is a flat Kähler metric. Moreover, on the universal cover \( p \colon \mathds{C} \times \mathds{H} \to S \), one has \( p^* \omega_{\text{SFL}} = c \cdot \sqrt{-1} dz \wedge d\bar{z} \), where \( c = c(\Im w) \) is a smooth positive function depending only on the imaginary part of \( w \) in the case of \( S_M \), while it is a positive constant \( c > 0 \) in the case of \( S^\pm \).
We then proved that any Gauduchon metric can be deformed within its \(\sqrt{-1}\partial\overline{\partial}\)-class to a metric that is strongly flat along the leaves, thereby extending the convergence result to arbitrary Gauduchon initial metrics.

\begin{theorem}[{\cite{fang-tosatti-weinkove-zheng, angella-tosatti}}]
Let \( S \) be an Inoue–Bombieri surface endowed with a Gauduchon metric \( \omega \). Then the solution \( \omega(t) \) of the normalized Chern–Ricci flow starting at \( \omega \) exists for all time and converges uniformly to a positive multiple of the degenerate form \( \alpha \). Moreover, the associated metric spaces \( (S, \omega(t)) \) converge in the Gromov–Hausdorff sense to the circle \( S^1 \) equipped with the standard round metric.
\end{theorem}

\begin{proof}
Thanks to \cite{fang-tosatti-weinkove-zheng}, it suffices to show that any Gauduchon metric \(\omega\) can be deformed to a metric of the form \(\omega + \sqrt{-1}\partial\overline{\partial} u\), for some smooth real-valued function \(u\), to obtain a metric that is strongly flat along the leaves.
This condition is equivalent to requiring that \( (\omega + \sqrt{-1}\partial\overline{\partial} u) \wedge \alpha = c\, \omega_T^2 \), where $c>0$ is a positive constant and \( \omega_T \) denotes the Tricerri/Vaisman metric.
In other words, we are looking for a smooth solution \( u \) of the degenerate elliptic equation
\[
\frac{\sqrt{-1} \partial \overline{\partial} u \wedge \alpha}{\omega_T^2} = -\frac{\omega \wedge \alpha}{\omega_T^2} + \frac{\int_S \omega \wedge \alpha}{\int_S \omega_T^2},
\]
which can expressed in terms of the leafwise Laplacian
\[
\Delta_{\mathcal{D}} u := \frac{\sqrt{-1} \partial \overline{\partial} u \wedge \alpha}{\omega_T^2}
\]
as
\[
\Delta_{\mathcal{D}} u = G(\omega), \qquad \text{ where } G(\omega):=-\frac{\omega \wedge \alpha}{\omega_T^2} + \frac{\int_S \omega \wedge \alpha}{\int_S \omega_T^2}.
\]

We first observe that the necessary condition for the solvability of the equation \( \Delta_{\mathcal{D}} u = G(\omega) \) is indeed satisfied, due to the Gauduchon condition on \( \omega \). 
To see this, consider a general Hermitian metric of the form
\[
\omega = \sqrt{-1} \, r^2 y_2 \, dz \wedge d\bar{z} + \sqrt{-1} \, s^2\, \frac{1}{y_2^2} \, dw \wedge d\bar{w} 
+ u \frac{1}{\sqrt{y_2}} \, dz \wedge d\bar{w} 
- \bar{u} \frac{1}{\sqrt{y_2}} \, dw \wedge d\bar{z},
\]
where \( r^2 > 0 \), \( s^2 > 0 \), and \( r^2 s^2 - |u|^2 > 0 \) hold to ensure that \( \omega \) is positive-definite.
The condition for \( \omega \) to be Gauduchon, that is \( \partial\overline\partial\omega = 0 \), can be reformulated in terms of the fibration structure as
\[
d\left( \int_{\pi^{-1}(y_2)} \overline\partial \omega \right) = 0,
\]
which is equivalent to requiring that the function
\[
R(y_2) := \int_{\pi^{-1}(y_2)} r \cdot \mathrm{vol}(\omega_T|_{\pi^{-1}(y_2)})
\]
is constant in \( y_2 \).
On the other hand, the equation \( \Delta_{\mathcal{D}} u = G(\omega) \) admits a solution only if \( G(\omega) \) is orthogonal to \( \ker \Delta_{\mathcal{D}} \) with respect to the \( L^2 \)-inner product induced by \( \omega_T \). 
Due to the density of the leaves in the fibres, we have \( \ker \Delta_{\mathcal{D}} = \pi^* \mathcal{C}^\infty(S^1,\mathds{R}) \), and thus the orthogonality condition becomes:
\[
\int_{\pi^{-1}(y_2)} G(\omega) \, \mathrm{vol}(\omega_T|_{\pi^{-1}(y_2)}) = 0 \quad \text{for all } y_2 \in S^1.
\]
A direct computation shows that the left-hand side equals
\[
-\frac{1}{8} R(y_2) + \frac{1}{8} \dashint_S R(y_2) \, \omega_T^2,
\]
where \( R(y_2) := \int_{\pi^{-1}(y_2)} r \cdot \mathrm{vol}(\omega_T|_{\pi^{-1}(y_2)}) \). Hence, the solvability condition is satisfied whenever \( \omega \) is Gauduchon.

\medskip

We now prove that the above equation admits a solution, focusing on the case of Inoue–Bombieri surfaces of type \( S_M \). These can be viewed as mapping tori of the form
\[
S_M = \frac{ T^3 \times [1, \lambda] }{ (q,1) \sim (\psi(q), \lambda) },
\]
where \( \psi \) is a diffeomorphism of the \(3\)-torus \( T^3 \).  
Here, \( T^3 = \sfrac{\mathds{R}^3}{\mathds{Z}\langle \partial_1, \partial_2, \partial_3 \rangle} \), where \( \partial_i := (\Re m_i, \Im m_i, \ell_i) \) for \( i \in \{ 1,2,3 \} \).
We aim to solve the equation \( \Delta_{\mathcal{D}} u = g \) on \( T^3 \times [1, \lambda] \), under the assumption that the datum satisfies the compatibility condition \( g(q, 1) = g(\psi(q), \lambda) \). We will later show that the solution \( u \) descends to the quotient \( S_M \).

Explicit computations show that the leafwise Laplacian \( \Delta_{\mathcal{D}} \), when expressed in coordinates adapted to the lattice, takes the form
\[
\Delta_{\mathcal{D}} = \frac{1}{32 y_2} 
\begin{pmatrix}
\partial_1 & \partial_2 & \partial_3
\end{pmatrix}
Z
\begin{pmatrix}
\partial_1 \\
\partial_2 \\
\partial_3
\end{pmatrix},
\]
where \( Z \) is a symmetric, semi-positive definite matrix of rank 2. Its kernel is generated by the vector \( (\ell_1, \ell_2, \ell_3) \).

We exploit the Fourier expansion along the torus. More precisely, we write the datum as
\[
g(t, y_2) = \sum_{k \in \mathds{Z}^3} g_k(y_2) \exp(2\pi \sqrt{-1} \langle t, k \rangle ),
\]
where \( g_k(y_2) \) are smooth functions, and the zero mode \( g_0 \) vanishes due to the Gauduchon assumption on \(\omega\).
Note that, since \(\lambda \not\in \mathds{Q}\), we have \(\ker Z \cap \mathds{Z}^3 = \{0\}\). Therefore, we obtain a distributional solution by setting
\[
u_k(y_2) = \frac{1}{(2\pi \sqrt{-1})^2} \, \frac{32 \, y_2 \cdot g_k(y_2)}{k Z k^t}.
\]

The Liouville approximation theorem for irrational algebraic numbers ensures that the solution \( u \) is indeed smooth on \( T^3 \). More precisely, since \(\sfrac{\ell_1}{\ell_2}\) is assumed to be irrational but algebraic of degree \( d \), we have the estimates
\begin{align*}
|k Z k^t| &\gsim {\rm dist}((k_1, k_2, k_3), {\rm span}(\ell_1, \ell_2, \ell_3)) \geq {\rm dist}((k_1, k_2), (\ell_1, \ell_2)) \\
& \gsim \left| \frac{k_2}{k_1} - \frac{\ell_2}{\ell_1} \right|^2 \cdot |k_1|^2 
\gsim \frac{1}{|k_1|^{2d - 2}} \geq \frac{1}{|k|^{2(d-1)}},
\end{align*}
where \( k = (k_1, k_2, k_3) \in \mathds{Z}^3 \setminus \{0\} \).
Therefore, since our datum \( g \) is smooth, for any \( N > 0 \) there exists a constant \( C_N \) such that
\[
|u_k(y_2)| < C_N |k|^{-N} \quad \text{as } |k| \to +\infty.
\]
This decay property guarantees that the Fourier series defines \( u \) smoothly function on \( T^3 \).
The smoothness in the variable \( y_2 \) follows from the smoothness of the Fourier coefficients and the continuous dependence.

Finally, uniqueness of the solution ensure that it descends to $S_M$.

\medskip

We briefly sketch the case of the Inoue surface \( S^+ \), while the remaining case of \( S^- \) reduces to \( S^+ \) via the unramified double cover.
In this case, we employ partial Fourier expansion on nilmanifolds, a technique that has proven useful in many other contexts, see {\itshape e.g.} \cite{auslander-tolimieri, deninger-singhof, richardson, holt-zhang-MRL, holt-zhang-AdvMath, ruberman-saveliev}, and that beautifully connects complex geometry, pde's, and number theory.

In this case, an arbitrary function on the fibres over \(y_2\) can be identified with a smooth function \( f(x_2, y_1, x_1) \) satisfying the periodicity conditions
\[
\begin{cases}
f(x_2, y_1, x_1) = f(x_2 + a_1,\, y_1 + y_2 b_1,\, x_1 + b_1 x_2 + c_1), \\
f(x_2, y_1, x_1) = f(x_2 + a_2,\, y_1 + y_2 b_2,\, x_1 + b_2 x_2 + c_2), \\
f(x_2, y_1, x_1) = f(x_2,\, y_1,\, x_1 + c_3),
\end{cases}
\]
where \( a_1, b_1, c_1, a_2, b_2, c_2, c_3 \) are parameters arising in the construction.
We perform a partial expansion 
\[
f(x_2,y_1,x_1) = \sum_{k \in \mathds{Z}} f_k(x_2,y_1) \exp\left( 2 \pi \sqrt{-1} \frac{x_1}{c_3} k \right)
\]
with fast-decay coefficients satisfying the periodicities
\begin{eqnarray*}
f_k(x_2,y_1) &=& f_k(x_2 + a_1, y_1 + y_2 b_1) \exp\left( 2 \pi \sqrt{-1} \left( \frac{b_1}{c_3} x_2 + \frac{c_1}{c_3} \right) k \right) \\
&=& f_k(x_2 + a_2, y_1 + y_2 b_2) \exp\left( 2 \pi \sqrt{-1} \left( \frac{b_2}{c_3} x_2 + \frac{c_2}{c_3} \right) k \right).
\end{eqnarray*}
In particular,
\[
f_0(x_2,y_1) = f_0(x_2 + a_1, y_1 + y_2 b_1) = f_0(x_2 + a_2, y_1 + y_2 b_2),
\]
therefore,
\[
f_0(x_2,y_1) = \sum_{m,n \in \mathds{Z}} f_{0,m,n} \exp \left(
\frac{2\pi \sqrt{-1} \left( y_2 b_2 x_2 - a_2 y_1 \right)}{y_2 (a_1 b_2 - a_2 b_1)} m
+ \frac{2\pi \sqrt{-1} \left( - y_2 b_1 x_2 + a_1 y_1 \right)}{y_2 (a_1 b_2 - a_2 b_1)} n
\right).
\]

By using the above expansion, the equation \( \Delta_{\mathcal{D}} u = g \) on fibres is reduced to the following system of linear ordinary differential equations of second order:
\[
\left(\frac{\partial}{\partial y_1}\right)^2 u_k(x_2,y_1) - 4\pi^2 \frac{k^2}{c_3^2} u_k(x_2,y_1) = g_k(x_2,y_1),
\]
for \( k \in \mathds{Z} \).

To solve the equation for \( k \neq 0 \), we apply the classical theory of ode's, while for \( k=0 \) it reduces to the algebraic equation
\[
-4\pi^2 u_{0,m,n} \left( \frac{n a_1 - m a_2}{y_2(a_1b_2 - a_2b_1)} \right)^2 = g_{0,m,n}.
\]
Arguments as before allow us to conclude.
\end{proof}

\ifdefined\frommain
\else
    \bibliographystyle{alpha}
    \bibliography{biblio}
\fi

\subsection{Problems}

\begin{exercise}
Let $X$ be a holomorphic manifold endowed with a Hermitian metric $\omega$. Prova that the first Ricci form of the Chern connection, ${\rm Ric}^{Ch}(\omega)$, can be locally expressed as
$$ {\rm Ric}^{Ch}(\omega) \stackrel{\text{loc}}{=} -\sqrt{-1}\partial\overline\partial\log\det\omega . $$
In particular, it defines a cohomology class in $H^{1,1}_{BC}(X)$, which is independent of the chosen metric. It is denoted as $c_1^{BC}(X)$, the first Bott-Chern class of $X$, and maps to the usual first Chern class $c_1(X)$ under the natural map $H^{1,1}_{BC}(X) \to H^2_{dR}(X,\mathds C)$.
\end{exercise}

\begin{exercise}
Let \( X \) be a compact holomorphic manifold and let \( L \to X \) be a holomorphic line bundle. Given a Hermitian metric \( h \) on the fibres of \( L \), its curvature is locally defined by
\[
R_h \stackrel{\text{loc}}{:=} -\sqrt{-1} \, \partial \overline{\partial} \log h,
\]
which gives a real, closed \((1,1)\)-form on \( X \). The Bott–Chern cohomology class of \( R_h \) in \( H^{1,1}_{BC}(X) \) is independent of the choice of \( h \), and is denoted by \( c_1^{BC}(L) \).
Show that a Hermitian metric \( g \) on \( X \) induces a Hermitian metric \( \det g \) on the fibres of the anticanonical line bundle \( K_X^{-1} \), that is, the determinant of the holomorphic tangent bundle, and that the Chern-Ricci form satisfies \( \mathrm{Ric}^{Ch}(\omega) = R_{\det g} \). In particular, \( c_1^{BC}(X) = c_1^{BC}(K_X^{-1}) \).
\end{exercise}

\begin{exercise}
Let \( X \) be a holomorphic manifold endowed with a Hermitian metric \( \omega \). Show that under a conformal change of metric, the Chern-Ricci form transforms as
\[
\mathrm{Ric}^{Ch}(e^f \omega) = \mathrm{Ric}^{Ch}(\omega) + \sqrt{-1} \, \partial \overline{\partial} f.
\]
Deduce that \( c_1^{BC}(X) = 0 \) if and only if \( X \) admits a Hermitian metric with vanishing Chern-Ricci curvature.
\end{exercise}

\begin{exercise}[{\cite{tosatti}}]
Let \( X \) be a compact holomorphic manifold of complex dimension \( n \) such that \( c_1^{BC}(X) = 0 \). Show that, under any one of the following assumptions, it follows that \( c_1(X) = 0 \) in \( H^2_{dR}(X, \mathds{R}) \):
\begin{itemize}
\item \( \mathrm{Kod}(X) \geq 0 \);
\item \( b_1(X) = 0 \);
\item \( X \) admits Kähler metrics;
\item \( n = 2 \);
\item \( X \) belongs to the Fujiki class \( \mathcal{C} \).
\end{itemize}
\end{exercise}

\begin{exercise}[{\cite{magnusson}}]
Let \( Y \) be a compact Kähler Calabi–Yau manifold of complex dimension \( n \), namely, \( K_Y \simeq \mathcal{O}_Y \). Let \( f\colon Y \to Y \) be an automorphism such that the induced action on \( H^0(Y, K_Y) \simeq \mathds{C} \) has infinite order.
Construct a complex manifold \( X \) of dimension \( n+1 \) as follows: consider the product \( Y \times \mathds{C} \), with an action of \( \mathds{Z}^2 \) generated by
\[
(1,0)\cdot(x,z) = (x, z+1), \quad \text{and} \quad (0,1)\cdot(x,z) = (f(x), z+\tau),
\]
for some fixed non-real \( \tau \in \mathds{C} \). Define
\[
X := (Y \times \mathds{C}) / \mathds{Z}^2,
\]
which is a holomorphic fibre bundle over the elliptic curve \( E = \mathds{C} / (\mathds{Z} + \tau \mathds{Z}) \), with fibre \( Y \).
Prove that \( X \) is not Kähler, its canonical bundle \( K_X \) is not torsion, even if the first Bott–Chern class vanishes, \( c_1^{BC}(X) = 0 \).
\end{exercise}

\begin{proof}[Hint]
Note that the image of the monodromy map \( \pi_1(E) \to \mathrm{Aut}(H^n(Y, \mathds{R})) \) contains an element of infinite order. This is not possible when \( X \) is Kähler, as shown in \cite[Corollary 4.10]{fujiki-InvMath}.
If \( \Omega \in H^0(Y, K_Y) \) is a nowhere vanishing holomorphic volume form, then \( f^* \Omega = \lambda \Omega \) for some \( \lambda \in \mathds{C} \) with \( |\lambda| = 1 \), and \( \lambda \) not a root of unity. This implies that \( {\rm Kod}(X) = -\infty \), in particular \( K_X \) is not torsion.
Finally, let \( \omega_Y \) be a Kähler Ricci-flat metric on \( Y \) and let \( \omega_E \) denote the standard flat metric on \( \mathds{C} \). Then \( \omega_Y^n \wedge \omega_E \) defines a Ricci-flat volume form on \( Y \times \mathds{C} \), which is invariant under the \( \mathds{Z}^2 \)-action and then descends to \( X \). This provides a Hermitian metric on \( X \) with vanishing Chern-Ricci curvature, and thus \( c_1^{BC}(X) = 0 \).
\end{proof}

\begin{exercise}
Prove that the Hopf manifold $X$ has $c_1(X)=0$ but $c_1^{BC}(X)\neq0$.
\end{exercise}

\begin{proof}[Hint]
Recall that $b_2(X)=b_2(S^1\times S^{2n-1})=0$. Note that the standard metric $\omega$ has ${\rm Ric}^{Ch}(\omega)\geq0$ but ${\rm Ric}^{Ch}(\omega)\neq0$.
\end{proof}

\begin{exercise}
Prove that the hypothetical complex structure on the six sphere $S^6$ has $c_1(X)=0$ but $c_1^{BC}(X)\neq 0$.
\end{exercise}

\begin{proof}[Hint]
Prove that the conditions \( c_1(X) = 0 \) and \( b_2(X) = 0 \) imply that \( K_X \) is holomorphically torsion. On the other hand, since \( H^1(X,\mathds{Z}) = H^3(X,\mathds{Z}) = 0 \), the exponential sequence gives \( {\rm Pic}(X) \simeq H^1(X,\mathcal{O}_X) \), which is torsion-free. Then \( K_X \) is in fact holomorphically trivial, so there exists a nowhere vanishing holomorphic $3$-form \( \Omega \).
Since \( \Omega \) is closed and \( H^3_{\mathrm{dR}}(X, \mathds{C}) = 0 \), it is exact: \( \Omega = d\eta \). Then
\[
0<(\sqrt{-1})^{n^2}\, \int_X \Omega \wedge \overline{\Omega} = (\sqrt{-1})^{n^2}\, \int_X d(\eta \wedge \overline{\Omega}) = 0,
\]
where $n=3$, which is a contradiction.
\end{proof}

\begin{exercise}
Let \( X \) be a compact holomorphic surface. Let \( \omega(t) \) be a family of Hermitian metrics evolving along the Chern-Ricci flow starting at \( \omega_0 \). Show that the volumes of $X$ and of any curve $D$ evolve according to the following formulas:
\begin{align*}
{\rm Vol}(X, \omega(t)) &= \frac{1}{2} \int_X \left(\omega_0 - t\, {\rm Ric}^{Ch}(\omega_0)\right)^2, \\
{\rm Vol}(D, \omega(t)) &= \int_D \left(\omega_0 - t\, {\rm Ric}^{Ch}(\omega_0)\right).
\end{align*}
\end{exercise}

\begin{exercise}
Let \( X \) be a compact holomorphic manifold endowed with a Hermitian metric \( \omega_0 \).
Show that, up to reparametrization of time and rescaling of the metric, there is a one-to-one correspondence between the Chern–Ricci flow
\begin{equation}\tag{CRF}\label{eq:crf-ex}
\frac{\partial}{\partial t} \omega(t) = -\mathrm{Ric}^{Ch}(\omega(t)), \quad \omega(0) = \omega_0,
\end{equation}
and the normalized Chern–Ricci flow
\begin{equation}\tag{NCRF}\label{eq:ncrf-ex}
\frac{\partial}{\partial t} \omega(t) = -\mathrm{Ric}^{Ch}(\omega(t)) - \omega(t), \quad \omega(0) = \omega_0.
\end{equation}
\end{exercise}

\begin{proof}[Hint]
If \(\tilde{\omega}(s)\) solves \eqref{eq:crf-ex}, then setting
\[
\omega(t) := \frac{\tilde{\omega}(s)}{s+1} \quad \text{with} \quad t = \log(s+1),
\]
gives a solution to \eqref{eq:ncrf-ex}
Conversely, if \(\omega(t)\) solves \eqref{eq:ncrf-ex}, then
\[
\tilde{\omega}(s) := e^{t} \omega(t) \quad \text{with} \quad s = e^{t} - 1,
\]
solves \eqref{eq:crf-ex}.
\end{proof}

\begin{exercise}[{\cite{tosatti-weinkove-CRF}}]
Let \( X \) be a compact holomorphic surface. Consider the maximal existence time \( T \) of the Chern–Ricci flow starting at any Gauduchon metric \( \omega_0 \). Prove that:
\begin{itemize}
\item If \( T = +\infty \), then \( X \) is minimal.
\item If the Chern–Ricci flow collapses in finite time, {\itshape i.e.} \( T < +\infty \) and \( {\rm Vol}(X, \omega(t)) \to 0 \) as \( t \to T^-\), then \( {\rm Kod}(X)=-\infty \), therefore \( X \) is either birational to a ruled surface or belongs to class VII.
\item If \( T < +\infty \) and the flow is non-collapsing, {\itshape i.e.} \( {\rm Vol}(X, \omega(t)) \geq c > 0 \) as \( t \to T^-\), then \( X \) contains \((-1)\)-curves.
\end{itemize}
\end{exercise}

\begin{exercise}[{\cite{tosatti-weinkove-surfacesCRF}}]
Compute explitic solution of the (normalized) Chern-Ricci flow on Hopf surface and Inoue surface, starting at the standard metrics.
\end{exercise}

\ifdefined\frommain
\else
    \bibliographystyle{alpha}
    \bibliography{biblio}
\fi

\cleardoublepage

\section{Open problems in complex non-K\"ahler geometry}\label{sec:open}

\begin{flushright}
{\itshape
Mathematics education can mean\\
disempowerment or empowerment.\\
}
(O.~Skovsmose, M.~Niss, \cite{skovsmose-niss})
\end{flushright}

The following list presents a selection of open problems in complex non-K\"ahler geometry, and related boundary regions. It is by no means intended to be complete or comprehensive and naturally reflects the authors’ interests and limited perspective. Nonetheless, the authors wish to emphasize the importance of identifying meaningful directions of research, --- also in light of ethical considerations and social justice.

\subsection{Geometric representatives of cohomological classes}

Even though it pertains primarily to K\"ahler algebraic geometry, we cannot omit mentioning the Hodge Conjecture, which seeks to relate geometric objects to cohomology classes. For a complete exposition, we refer to \cite{voisin-FR}.

Consider an analytic subset of a holomorphic manifold \( X \), namely, a closed subset \( Z \subset X \) that is locally defined as the zero set of holomorphic functions. The analytic subset \( Z \) is said to be irreducible if its smooth part \( Z_{\mathrm{smooth}} \) is connected. In this case, we define the dimension of \( Z \) as the complex dimension of \( Z_{\mathrm{smooth}} \). By the Weierstrass Preparation Theorem, an irreducible analytic subset \( Z \subset X \) of codimension \( r \) defines a cohomology class \( [Z] \in H^{2r}(X, \mathds{Z}) \), given by the Poincaré dual of the smooth part \( Z_{\mathrm{smooth}} \).
We define the group of analytic cycles of codimension \( r \) as the free Abelian group \( \mathcal{Z}^r(X) \) generated by irreducible analytic subsets of codimension \( r \). By the previous argument, there is a well-defined cycle class map
\[
\mathcal{Z}^r(X) \to H^{2r}(X, \mathds{Z}),
\]
whose image,
\[
H^{2r}(X, \mathds{Z})_{\mathrm{an}} := \im ( \mathcal{Z}^r(X) \to H^{2r}(X, \mathds{Z})) ,
\]
is called the group of \emph{analytic classes}.
In the projective algebraic case, we speak of \emph{algebraic classes}, denoted by \( H^{2r}(X,\mathds{Z})_{\mathrm{alg}} \).
Note that, when \( X \) is K\"ahler, analytic classes lie in \( H^{2r}(X,\mathds{Z}) \cap H^{r,r}(X) \), namely, they are {\em Hodge classes}.

\begin{problem}[rational Hodge conjecture]
Let $X$ be a projective manifold. Then any Hodge class of degree $2r$ is represented by an integer multiple of the class of an algebraic cycle, that is,
$$ H^{2r}(X,\mathds Q)_{\mathrm{alg}} = H^{2r}(X,\mathds Q) \cap H^{r,r}(X) . $$
\end{problem}

As far as now, the conjecture is confirmed for projective threefolds, as consequence of the Leftschetz $(1, 1)$ theorem, as well for special classes of projective manifolds.

The statement is no longer true in the broader Kähler setting, as proved by Voisin \cite{voisin-IMRN}. It would be interesting to understand to what extent the Hodge problem can be formulated in the more general complex setting, possibly beyond the Kähler category, or at least within suitable subclasses of complex manifolds \cite{bassanelli}.

\subsection{Topological obstructions to holomorphic geometry}

A fundamental question is to study algebraic invariants defined on compact holomorphic, or even differentiable, manifolds, and to clarify to what extent they depend on the presence of an integrable complex structure (see also \Cref{sec:dolbeault-almost-complex}).
We take inspiration and reference from the beautiful survey \cite{stelzig-survey}, focusing on cohomology groups, Chern classes, rational homotopy groups, and higher operations, both in their {\em realization problem} (namely, which sets of invariants can actually be realized by a space, or what linear relations they must satisfy), and their {\em classification problem} (namely, to what extent the given invariants determine the space).

For example, every sequence of non-negative integers can be realized as the sequence of Betti numbers of some topological space, for instance, by taking wedges of spheres. Even more deeply, it is a classical result in Rational Homotopy Theory that any finite-dimensional, simply-connected, rational commutative graded algebra can be realized as the singular cohomology ring with rational coefficients of a simply-connected topological space.

When dealing with compact differentiable manifolds \( X \) of dimension \( n \), there exists a perfect pairing \( H^k(X, \mathds{Q}) \times H^{n-k}(X, \mathds{Q}) \to \mathds{Q} \), which, in particular, implies the symmetry of Betti numbers: \( b_k(X) = b_{n-k}(X) \). Moreover, when \( n = 2k \) is even, the induced pairing on \( H^k(X, \mathds{Q}) \) is \( (-1)^k \)-symmetric. As a consequence, if \( n \equiv 2 \mod 4 \), then the middle Betti number satisfies \( b_{n/2}(X) \equiv 0 \mod 2 \).
By \cite{kotschick-schreider}, the relations above are the only relations that hold universally for all Betti numbers of compact smooth manifolds. The realization problem for the cohomology ring of compact smooth manifolds was studied by Sullivan \cite{sullivan} and Barge \cite{barge}; see {\itshape e.g.} \cite[Theorem 1.5]{stelzig-survey} for a precise statement.

The almost-complex version of the Sullivan-Barge theorem was stated in \cite{sullivan} and later proved by Milivojevi\'c \cite{milivojevic}; see \cite[Theorem 1.7]{stelzig-survey} for a precise statement. In this case, the realization problem involves the Stong congruences satisfied by the Chern numbers (namely, integrality conditions coming out of the Atiyah-Singer Index Theorem), the requirement that the signature of the intersection form equals a polynomial combination of Pontryagin numbers, and the condition that the Euler characteristic equals the top Chern number.

One advantage of working with integrable complex structures is the availability of various cohomology theories.
In the compact Kähler case, a result by Kotschick and Schreieder \cite{kotschick-schreider} shows that the only universal rational linear relations between Hodge numbers, Betti numbers, and Chern numbers are those implied by the real structure and Poincaré duality, the Hodge decomposition, and the Riemann-Roch theorem.
The following questions naturally arise in this context: see \cite{stelzig-survey} and the references therein for further details and partial results.

\begin{problem}[{\cite{stelzig-survey}}]
What are the linear relations that hold universally between the dimensions of the various cohomology theories associated with compact complex manifolds of a given dimension?
Are there other universal rational relations beyond those identified in \cite{stelzig-JLMS}, which arise from the real structure, duality, constraints involving only dots in the corner, Frölicher spectral sequence degeneration for compact complex surfaces, and the known quadratic relation in complex dimension \( n = 2 \)?
Are there linear combinations of the multiplicities of zigzags and Chern numbers that are topological invariants in complex dimension \( n \geq 3 \), other than the linear combinations of Betti and Pontryagin numbers?
Are there linear combinations of cohomological invariants that are bimeromorphic invariants, other than those identified in \cite{stelzig-IMRN, stelzig-JLMS}?
\end{problem}

This question is also related to a problem posed by Hirzebruch, see \cite[Problem 31]{hirzebruch-problems}. It essentially reduces to a construction problem, see \cite{stelzig-AdvMath22, stelzig-survey} for a precise formulation.
We also note the existence of universal inequalities among cohomological invariants, see \cite{angella-tomassini-InvMath, angella-tardini-PAMS, popovici-stelzig-ugarte, stelzig-wilson}.

We briefly mention two further facts. First, a fundamental result by Taubes \cite{taubes}, see {\itshape e.g.} \cite[Corollary 1.66]{amoros-burger-corlette-kotshchick-toledo}, which states that every finitely presentable group arises as the fundamental group of some compact complex threefold. Second, we highlight recent work by Stelzig \cite{stelzig-AdvMath25}, who developed a pluripotential analogue of rational homotopy theory for compact complex manifolds.

Understanding to what extent the presence of a complex structure constrains the underlying topology is a fundamental open question, to which the results and questions discussed above aim to contribute.
Since spheres represent the simplest topological spaces, a natural question is whether a sphere admits an integrable complex structure. For \( S^2 \) and \( S^6 \), almost-complex structures can be constructed using the quaternions and octonions, respectively. However, a classical result by Borel and Serre \cite{borel-serre} shows that no other sphere admits an almost-complex structure.
The natural almost-complex structure on \( S^2 \) is integrable and corresponds to the complex manifold \( \mathds CP^1 \). In contrast, due to the non-associativity of the octonions, the natural almost-complex structure on \( S^6 \) is not integrable. Whether \( S^6 \) admits any integrable almost-complex structure remains an open question, known as the Hopf problem. We refer to \cite{agricola-bazzoni-goertsches-konstantis-rollenske} for a detailed history and recent developments on this problem.

\begin{problem}[Hopf problem {\cite{hopf} and Yau problem \cite{vandeven, yau-problems}}]
Does the six-dimensional sphere \( S^6 \) admit any integrable complex structure?
More generally, is there a differentiable manifold of real dimension \( 2n \geq 6 \) that admits an almost-complex structure but no integrable complex structure?
\end{problem}

As a tentative approach to the Hopf problem, and inspired by an analogue picture in real fluid dynamics, Sullivan proposed the following conjecture, as communicated in \cite{albanese-milivojevic}.

\begin{problem}[Sullivan Conjecture, see {\cite{albanese-milivojevic}}]
Is the minimal sum of Betti numbers of a compact complex manifold of complex dimension \( n \geq 3 \) equal to four?
\end{problem}

For compact complex manifolds of complex dimension \( n \geq 4 \), it is expected that the minimal total Betti number is \(3\), achieved for instance by Calabi-Eckmann manifolds and Hopf manifolds.  
For compact complex curves, the minimal value is \(2\) (attained by \(\mathds CP^1\)), while for compact complex surfaces it is \(3\) (attained by \(\mathds CP^2\)).
A positive answer to the Sullivan conjecture would therefore exclude the existence of a complex structure on \( S^6 \), whose total Betti number is only \(2\), while the next minimal value is attained by \(\mathds CP^3\).

We notice that the existence of a complex structure on \( S^6 \) would imply the existence of an exotic complex structure on \( \mathds CP^3 \), obtained by blowing up the complex \( S^6 \) at a point, see {\itshape e.g.} \cite{tosatti-Expo}.  
Note that it remains an open question whether every complex structure on \( \mathds CP^3 \) or, more generally, on any compact Kähler manifold of dimension greater than or equal to \( 3 \), is necessarily Kähler.

\subsection{Almost-complex cohomologies}\label{sec:dolbeault-almost-complex}

Let us assume that \( M \) is a differentiable manifold endowed with a (non-integrable) almost-complex structure. In this case, \( \overline\partial^2 \neq 0 \), but once an almost-Hermitian metric $g$ is fixed, we can still define the elliptic operator 
\[
\square_g = \overline\partial\, \overline\partial^* + \overline\partial^*\, \overline\partial.
\]
In the integrable case, Hodge theory guarantees that the dimension of its kernel is independent of the metric. However, this is not necessarily true in the non-integrable setting. This question appears in Hirzebruch's famous problem list \cite{hirzebruch-problems} as Problem 20, known as the Kodaira–Spencer problem.

\begin{problem}[Kodaira--Spencer, Hirzebruch~\cite{hirzebruch-problems}]
Let \( M \) be a differentiable manifold endowed with an almost-complex structure $J$. Choose an almost-Hermitian metric \( g \) and consider the numbers 
\[
h^{p,q}_g := \dim\left( \ker \square_g \cap \Omega^{p,q}(M,J) \right).
\]
Is \( h^{p,q}_g \) independent of the choice of the almost-Hermitian structure?
If not, find another definition of ``Hodge numbers'' of \( (M, J) \) which depends only on the almost-complex structure and generalizes the Hodge numbers in the integrable complex case.
\end{problem}

Recently, Holt and Zhang \cite{holt-zhang-AdvMath} gave a negative answer to the problem. By developing an effective method to determine the \( \Box \)-harmonic forms on the Kodaira-Thurston manifold, they showed that the quantity \( h^{0,1} \) can vary with different choices of almost-Hermitian metrics. With different techniques, involving locally conformally almost-K\"ahler metrics, in \cite{tardini-tomassini-MRL} it was shown that also \( h^{1,1} \) can vary with different choices of almost-Hermitian metrics. More precisely, let $(X^4,J)$ be a compact almost-complex manifold of dimension $4$ and let $\omega$ be an almost-Hermitian metric, then 
if $\omega$ is globally conformally almost-K\"ahler (in particular if it is almost-K\"ahler), it holds
$
h^{1,1}=b_-+1.
$
If $\omega$ is (strictly) locally conformally almost-K\"ahler, 
$
h^{1,1}=b_-.
$
Here $b_-$ is the dimension of the space of anti self-dual harmonic forms. In fact, explicit examples of almost-complex manifolds admitting these two kinds of of metrics can be constructed. Notice that by \cite{vaisman}, on compact complex manifolds, globally conformally K\"ahler metrics and (strictly) locally conformally K\"ahler metrics cannot coexist.

A recent attempt to define an analogue of Dolbeault cohomology in the almost-complex setting is due to Cirici and Wilson \cite{cirici-wilson}. Given the decomposition of the exterior differential
\[
d = \mu + \partial + \overline{\partial} + \overline{\mu}
\]
into pure-type components, where \( \mu \colon \Omega^{p,q}(M,J) \to \Omega^{p+2,q-1}(M,J) \), they define the Dolbeault cohomology of \( (M, J) \) as
\[
H^{p,q}_{\mathrm{Dol}}(M,J) := H^q\left(H^{p}(\Omega^{\bullet,\bullet}(M,J), \overline{\mu}), \overline{\partial} \right),
\]
which arises as the first page of a functorial spectral sequence converging to the de Rham cohomology of \( M \).
It often happens that this cohomology is infinite-dimensional, see \cite{cohelo-placini-stelzig}.

In general, it would be useful to extend classical invariants of complex geometry to the non-integrable case, see for instance \cite{chen-zhang} regarding the Kodaira dimension. We refer to \cite{stelzig-survey} for a recent and thorough survey on related problems in this direction.

\subsection{Topological obstructions to Hermitian geometry}

As we have seen, the existence of Kähler metrics imposes strong topological obstructions on the underlying manifold. On the other hand, there are no known topological obstructions for weaker metric classes such as balanced manifolds, although they admit characterizations in terms of currents \cite{michelsohn}.
Moreover, it can be easily shown that a Hermitian metric which is both pluriclosed and balanced is necessarily Kähler \cite{alexandrov-ivanov}. In all known examples, the existence of both a pluriclosed and a balanced metric (not necessarily the same metric) on the same manifold occurs only in the Kähler case.

\begin{problem}[Fino--Vezzoni conjecture {\cite{fino-vezzoni-JGP, fino-vezzoni-PAMS}}]
Let \(X\) be a compact holomorphic manifold that does not admit any Kähler metric. Is it true that the existence of balanced Hermitian metrics and the existence of pluriclosed Hermitian metrics are mutually exclusive properties on \(X\)?
Where does this incompatibility arise from, and are there other pairs of special Hermitian metrics exhibiting similar incompatibility phenomena?
\end{problem}

In the same spirit, the Streets-Tian Conjecture \cite{streets-tian-IMRN} states that a compact holomorphic manifold admitting a Hermitian-symplectic metric (that is, a Hermitian metric which is the \((1,1)\)-component of a closed real \(2\)-form) must also admit a Kähler metric. This can be viewed as a higher-dimensional, integrable analogue of Donaldson's tamed-to-compatible question for almost-complex \(4\)-manifolds \cite{donaldson-Chern}.

It would be interesting to investigate whether special metrics exist on certain classes of complex manifolds. For instance, all known examples of compact complex manifolds satisfying the \(\partial\overline{\partial}\)-Lemma property, including Mo\v{\i}\v{s}ezon manifolds and manifolds in Fujiki class \(\mathcal{C}\), twistor spaces, and manifolds arising from small deformations, indeed admit balanced metrics. The existence of balanced metrics may follow from solving Hessian equations on Hermitian manifolds related to Gauduchon-Calabi-Yau \cite{szekelyhidi-tosatti-weinkove, popovici}, as we first learned from Demailly.

\begin{problem}[Demailly]
Does every compact complex manifold satisfying the \(\partial\overline{\partial}\)-Lemma property admit a balanced Hermitian metric?
\end{problem}

We conclude this discussion by noting that construction techniques for special Hermitian metrics may involve several approaches: elliptic non-linear partial differential equations, even in weak sense (note that, when dealing with non-Kähler metrics, deformations within a fixed $\sqrt{-1}\partial\overline\partial$-class do not, in general, preserve the total volume, see \cite{guedj-lu-1, guedj-lu-2, guedj-lu-3, angella-guedj-lu} for related results and applications); geometric flows (since Gauduchon metrics are ubiquitous in any conformal class, the existence of a geometric flow preserving the Gauduchon condition in all complex dimensions would be of particular interest); variational methods (see, for instance, \cite{gauduchon-mathann, angella-istrati-otiman-tardini}); and more recently, probabilistic techniques in the spirit of \cite{berman-Yau, berman-ICM}, see also the references therein and \cite{ma-marinescu}.

\subsection{Canonical Hermitian metrics}

We already discussed the Chern-Yamabe problem, namely the problem of finding a metric with constant scalar curvature with respect to the Chern connection in any conformal class of Hermitian manifolds. As we said, the problem is still open when the expected constant curvature is positive, and variational techniques in the standard sense can not be applied.
Globally speaking, such a metric would be highly non-canonical, since one exists for any conformal class: one may asks how to make it canonical by coupling it with another condition, for example, requiring it to be Gauduchon. Recall that any conformal class of Hermitian metrics contains a Gauduchon representative, as proved in \cite{gauduchon-cras}, thus suggesting a possible iterative procedure.
The choice of the Chern connection is motivated by its role in complex geometry, especially from the analytic perspective. However, other choices are possible: see, for instance, the results by Barbaro \cite{barbaro} concerning Bismut connection.
We also suggested that the role of the scalar curvature might be replaced by other, more significant quantities, for example, the moment map constructed in \cite{angella-calamai-pediconi-spotti}.
We then state the problem in the following way.

\begin{problem}[Chern-Yamabe-type problems {\cite{angella-calamai-spotti, angella-pediconi-scarpa-spotti-windare}}]
Let $X$ be a compact holomorphic manifold. In any conformal class of Hermitian metric, does there exist a representative characterized by the constancy of the scalar curvature of the Chern connection, or of other related geometric quantities?
What can be said about the geometry of the moduli space of such metrics, {\itshape e.g.} compactness?
Is there a meaningful way to couple the curvature condition with a ``cohomological'' condition, {\itshape e.g.}, to construct constant Chern scalar curvature Gauduchon metrics?
\end{problem}

\subsection{Classification of complex manifolds and surfaces}

The Global Spherical Shell (GSS) Conjecture claims that all minimal class VII surfaces with positive second Betti number admit a global spherical shell, that is, a smooth \(3\)-sphere with a neighbourhood biholomorphic to a neighbourhood of \(S^3 \subset \mathds{C}^2\), whose complement is connected. Thanks to \cite{dloussky-oeljeklaus-toma}, the GSS Conjecture is equivalent to the following standard conjecture, whose proof would complete the classification of class VII surfaces, showing that they consist only of Hopf, Inoue, and Kato surfaces.

\begin{problem}[standard conjecture on class VII surfaces]
Any minimal compact holomorphic surface in class VII with positive second Betti number \( b_2 > 0 \) has exactly \( b_2(X) \) rational curves.
\end{problem}

By combining techniques from complex geometry and gauge theory, Teleman \cite{teleman-InvMath, teleman-AnnMath, teleman-indam, teleman-PSPUM} proved that any minimal class VII surface \( X \) with \( 1 \leq b_2(X) \leq 3 \) contains a cycle of curves. A result of Nakamura \cite{nakamura-Tohoku} then shows that \( X \) is a degeneration of a one-parameter family of blown-up primary Hopf surfaces, therefore completing the classification up to deformation equivalence for small second Betti number.
We refer to \cite{teleman-LNM} for a recent account of the latest developments in the classification of non-K\"ahlerian surfaces.

It is natural to wonder whether geometric flows can be used to reconstruct the global spherical shell, see \cite{streets-tian-IMRN, fang-tosatti-weinkove-zheng, tosatti-weinkove-surveyCRF}.

Concerning the classification of compact holomorphic manifolds in arbitrary dimension, we must mention the Minimal Model Program initiated by Mori and its analytic counterpart developed by Song and Tian \cite{song-tian} in the setting of projective algebraic varieties. In the more general, complex and non-K\"ahler case, one may expect only partial classification results for specific classes of manifolds, such as locally conformally Kähler threefolds.

We also recall a conjecture of Bogomolov \cite{bogomolov-RIMS}, asserting that any compact complex manifold can be realized as a submanifold transverse to an algebraic foliation on a projective variety. A partial result in this direction was obtained by Demailly and Gaussier \cite{demailly-gaussier}, who established an embedding theorem for compact almost complex manifolds as subvarieties transverse to an algebraic distribution in a complex affine algebraic manifold.

\ifdefined\frommain
\else
    \bibliographystyle{alpha}
    \bibliography{biblio}
\fi

\cleardoublepage

\bibliographystyle{alpha}
\bibliography{biblio}

\end{document}